\documentclass[10pt]{amsart}

\usepackage[margin=1.35in]{geometry}
\usepackage{amsmath}
\usepackage{mathtools}
\usepackage{amsfonts,amssymb}
\usepackage{graphicx}
\usepackage{subcaption}
\usepackage{float}
\usepackage{soul}
\usepackage{color}
\usepackage[colorlinks=true,linkcolor=blue,citecolor=blue,urlcolor=blue]{hyperref}

\newtheorem{theorem}{Theorem}[section]
\newtheorem{prop}[theorem]{Proposition}
\newtheorem{lem}[theorem]{Lemma}
\newtheorem{cor}[theorem]{Corollary}

\theoremstyle{definition}

\newtheorem{ass}[theorem]{Assumption}

\theoremstyle{remark}
\newtheorem{remark}[theorem]{Remark}

\numberwithin{equation}{section}

\DeclareMathOperator{\Null}{Null}

\newcommand{\be}{\begin{equation}}
\newcommand{\ee}{\end{equation}}

\newcommand{\interior}{\text{int}}
\newcommand{\clo}{\text{cl}}

\newcommand{\lb}{\left(}
\newcommand{\rb}{\right)}
\newcommand{\lsb}{\left[}
\newcommand{\rsb}{\right]}
\newcommand{\lcb}{\left\{}
\newcommand{\rcb}{\right\}}

\newcommand{\lan}{\left\langle}
\newcommand{\ran}{\right\rangle}

\newcommand{\norm}[1]{\left\Vert #1 \right\Vert}
\newcommand{\abs}[1]{\left\vert #1 \right\vert}

\newcommand{\ip}[1]{\lan #1 \ran}

\newcommand{\mb}[1]{\mathbf{#1}}
\newcommand{\bs}[1]{\boldsymbol{#1}}

\newcommand{\wt}[1]{\widetilde{#1}}
\newcommand{\ve}{\varepsilon}
\newcommand{\vf}{\psi}
\newcommand{\vfn}{{\bs\vf}}

\newcommand{\vr}{\varrho}

\newcommand{\B}{\mathcal{B}}
\newcommand{\C}{\mathcal{C}}
\newcommand{\D}{\mathcal{D}}

\newcommand{\M}{\mathbb{G}}
\newcommand{\CC}{\mathbb{C}}

\newcommand{\N}{\mathbb{N}}
\newcommand{\R}{\mathbb{R}}
\newcommand{\Z}{\mathbb{Z}}

\newcommand{\pd}{\partial}

\newcommand{\gmatrix}{H}

\newcommand{\Id}{I}
\newcommand{\MF}{J}

\newcommand{\buniq}{\beta_0}
\newcommand{\bmono}{\beta_1}
\newcommand{\bnbhdone}{\beta_2}
\newcommand{\bnbhdzero}{\beta_3}
\newcommand{\bzoa}{\beta_4}
\newcommand{\bzob}{\beta_5}

\newcommand{\stable}{\mathbb{S}}
\newcommand{\unstable}{\mathbb{U}}

\newcommand{\ball}{\mathbb{B}}

\newcommand{\sigminus}{\sigma_{-1}}
\newcommand{\sigminusz}{\sigma_{-0}}

\newcommand{\Iset}{\mathcal{I}}
\newcommand{\Jset}{\mathcal{J}}

\newcommand{\interval}{\mathbb{I}}
\newcommand{\banach}{X}
\newcommand{\banachY}{Y}
\newcommand{\field}{\mathbb{F}}

\newcommand{\mono}{M}
\newcommand{\monoU}{U}

\newcommand{\sops}{p}

\newcommand{\sopsn}{{\bf p}}

\newcommand{\xbeta}{p^\beta}
\newcommand{\nxbeta}{\bar\sops^\beta}

\renewcommand{\Re}{\operatorname{Re}}
\renewcommand{\Im}{\operatorname{Im}}

\renewcommand{\star}{\diamond}

\newcommand{\orbit}{\mathcal{O}}

\newcommand{\period}{\omega}

\newcommand{\hyper}{\mathcal{H}}

\begin{document}

\title[Synchronized periodic solutions of coupled delay equations]{Stability of synchronous slowly oscillating periodic solutions for delay differential equations with coupled nonlinearity}

\date{\today}

\author[David Lipshutz]{David Lipshutz*}
\address{Faculty of Electrical Engineering \\ 
               Technion --- Israel Institute of Technology \\ 
               Haifa 32000 \\ Israel}
\email{lipshutz@gmail.com}

\author[Robert J.\ Lipshutz]{Robert J.\ Lipshutz}
\address{Palo Alto, CA 94301 \\ USA}
\email{roblipshutz@gmail.com}



\thanks{*DL was supported in part at the Technion by the Israel Science Foundation (grant 1184/16) and a Zuckerman Fellowship.}

\begin{abstract}
We study stability of so-called synchronous slowly oscillating periodic solutions (SOPSs) for a system of identical delay differential equations (DDEs) with linear decay and nonlinear delayed negative feedback that are coupled through their nonlinear term.
Under a row sum condition on the coupling matrix, existence of a unique SOPS for the corresponding scalar DDE implies existence of a unique synchronous SOPS for the coupled DDEs.
However, stability of the SOPS for the scalar DDE does not generally imply stability of the synchronous SOPS for the coupled DDEs.
We obtain an explicit formula, depending only on the spectrum of the coupling matrix, the strength of the linear decay and the values of the nonlinear negative feedback function near plus/minus infinity, that determines the stability of the synchronous SOPS in the asymptotic regime where the nonlinear term is heavily weighted.
We also treat the special cases of so-called weakly coupled systems, near uniformly coupled systems, and doubly nonnegative coupled systems, in the aforementioned asymptotic regime.
Our approach is to estimate the characteristic (Floquet) multipliers for the synchronous SOPS. 
We first reduce the analysis of the multidimensional variational equation to the analysis of a family of scalar variational-type equations, and then establish limits for an associated family of monodromy-type operators.
We illustrate our results with examples of systems of DDEs with mean-field coupling and systems of DDEs arranged in a ring.
\end{abstract}

\maketitle

\tableofcontents


\section{Introduction}
\label{sec:intro}

\subsection{Background and motivation}
\label{sec:background}

Models for networks of coupled oscillators with time delays arise in the study of complex biological systems.
Examples include models of neuronal networks where the spatial distribution of neurons leads to propagation delays \cite{Campbell2007,Dhamala2004,Marcus1988,Rosenblum2004,wu2011introduction,Wu2015} and networks of genetic oscillators where lengthy transcription processes lead to signaling delays \cite{glass2016signaling,Mather2014,novak2008design}.
Of special interest for such a system is the synchronous state in which the elements of the system oscillate in unison.
A natural problem in this context, which has received considerable attention in the literature, is to characterize the stability of the synchronous state; see, for example, \cite{Atay2006,Campbell2018,campbell2006multistability,Chen2000,Flunkert2010,Flunkert2014,Guo2003,Guo2007,Guo2009,Sieber2013,Wang2017,Wu1999}.

In this work we consider the system of delay differential equations (DDEs)
	\be\label{eq:cdde}\dot x^j(t)=-\alpha x^j(t)+\beta\sum_{k=1}^nG^{jk}f(x^k(t-1)),\qquad j=1,\dots,n,\qquad t>0,\ee
where $n\ge 2$, $x^1,\dots,x^n$ are real-valued continuous functions on $[-1,\infty)$ that are continuously differentiable on $(0,\infty)$, $\alpha\ge0$ and $\beta>0$ are constants, $G=(G^{jk})$ is an $n\times n$ real-valued coupling matrix whose rows each sum to 1, and $f$ is a continuously differentiable real-valued function on $\R=(-\infty,\infty)$ satisfying the negative feedback condition $\xi f(\xi)<0$ for all $\xi\ne 0$.
We refer to the system of equations \eqref{eq:cdde} as the $n$-dimensional coupled DDE associated with $(\alpha,\beta,f,G)$.
Due to the dependence of the dynamics on the past in \eqref{eq:cdde}, the natural state descriptor for a solution ${\bf x}=(x^1,\dots,x^n)$ of the coupled DDE \eqref{eq:cdde} at time $t\ge0$ is the continuous $n$-dimensional function ${\bf x}_t$ on the delay interval $[-1,0]$ defined by ${\bf x}_t(\theta)={\bf x}(t+\theta)$ for $\theta\in[-1,0]$, which takes values in $C([-1,0],\R^n)$, the vector space of $n$-dimensional continuous functions on $[-1,0]$ equipped with the uniform norm.
Our focus is on stability of so-called synchronous slowly oscillating periodic solutions for the coupled DDE \eqref{eq:cdde}, in the asymptotic regime $\beta\to\infty$.
By focusing on this asymptotic regime, we are able to treat a broad class of nonlinearities $f$ and coupling matrices $G$, which is in contrast to many related works that impose additional symmetry conditions on the system, such as restricting to oscillators arranged in a ring \cite{campbell2006multistability,Chen2000,Guo2003,Guo2007,Guo2009,Wu1999} or, more generally, symmetric circulant coupling matrices \cite{Campbell2018,Wu1998}.

Before defining a synchronous slowly oscillating periodic solution for the coupled DDE \eqref{eq:cdde}, we first define a slowly oscillating periodic solution for the related scalar DDE
\begin{equation}\label{eq:dde}
	\dot x(t)=-\alpha x(t)+\beta f(x(t-1)),\qquad t>0,
\end{equation}
where $x$ is a real-valued continuous function on $[-1,\infty)$ that is continuously differentiable on $(0,\infty)$ and $\alpha$, $\beta$ and $f$ are as in system \eqref{eq:cdde}. 
We refer to equation \eqref{eq:dde} as the scalar DDE associated with $(\alpha,\beta,f)$. 
Analogous to the coupled DDE setting, the natural state descriptor for a solution $x$ of the scalar DDE \eqref{eq:dde} at time $t\ge0$ is the continuous real-valued function $x_t$ on the delay interval $[-1,0]$ defined by $x_t(\theta)=x(t+\theta)$ for all $\theta\in[-1,0]$, which takes values in $C([-1,0],\R)$, the vector space of real-valued continuous functions on $[-1,0]$ equipped with the uniform norm.
A \textbf{slowly oscillating periodic solution (SOPS)} for the scalar DDE \eqref{eq:dde} is a solution $\sops$ satisfying the following property: there are constants $z_0\ge-1$, $z_1>z_0+1$ and $z_2>z_1+1$ such that 
\begin{align*}
	\sops(t)&>0,\qquad z_0<t<z_1,\\
	\sops(t)&<0,\qquad z_1<t<z_2.
\end{align*}
and $\sops$ is periodic with period $\period=z_2-z_0$; that is, $\sops(t+\period)=\sops(t)$ for all $t\ge-1$.
Here ``slowly oscillating'' refers to the fact that $\sops$ oscillates about zero and the separation between zeros is greater than the delay length, which is normalized to be 1 in both the scalar DDE \eqref{eq:dde} and coupled DDE \eqref{eq:cdde}.
Given an SOPS $\sops$ of the scalar DDE \eqref{eq:dde}, its {\bf orbit} in $C([-1,0],\R)$ is given by $\orbit_\sops=\{\sops_t,t\ge0\}$.
We say an SOPS $\sops$ of the scalar DDE \eqref{eq:dde} is {\bf orbitally unique} if given any other SOPS $\tilde\sops$ of the scalar DDE \eqref{eq:dde} their orbits are equal; that is, $\orbit_\sops=\orbit_{\tilde\sops}$.
We say $\sops$ is {\bf asymptotically stable with an exponential phase} if the follow property holds: there are positive constants $C$ and $\gamma$, and an open set $W\subset C([-1,0],\R)$ containing the orbit $\orbit_\sops$ such that if $\phi\in W$ and $x$ is the solution of the scalar DDE \eqref{eq:dde} with initial condition $x_0=\phi$ then there exists $s=s(\phi)\ge0$ such that
	$$\sup_{\theta\in[-1,0]}|x_t(\theta)-\sops_{s+t}(\theta)|\le Ce^{-\gamma t}\sup_{\theta\in[-1,0]}|\phi(\theta)-\sops_s(\theta)|,\qquad t\ge0.$$
There are a number of works on conditions for existence, uniqueness, and stability of an SOPS for the scalar DDE \eqref{eq:dde}; see, for example, the discussion in \cite[Chapter XV.9]{Diekmann1991}.
Here we impose the following conditions on $f$, under which Xie \cite{Xie1991} proved uniqueness and stability of an SOPS for the scalar DDE \eqref{eq:dde} when $\beta$ is sufficiently large.

\begin{ass}\label{ass:main}
The function $f:\R\to\R$ satisfies the following conditions:
\begin{itemize}
        \item[1.] \textit{Regularity}: $f$ is continuously differentiable.
        \item[2.] \textit{Bounded negative feedback}: $f'(0)<0$, $f(\xi)\xi<0$ for all $\xi\ne0$, and there exist constants $a,b>0$ such that $\lim_{\xi\to\infty}f(\xi)=-a$ and $\lim_{\xi\to-\infty}f(\xi)=b$.
        \item[3.] \textit{Super-linear decay of the derivative}: $\int_{-\infty}^\infty\abs{f'(\xi)}d\xi<\infty$ and $\lim_{|\xi|\to\infty}\xi f'(\xi)=0$.
\end{itemize}
\end{ass}

\begin{remark}
Let $a,b>0$.
An example of a function $f:\R\to\R$ satisfying Assumption \ref{ass:main} is
\be\label{eq:fmf}
f(\xi)=
\begin{cases}
	\displaystyle
	-a\tanh\lb \xi/a\rb,&\xi\ge0,\\
	\displaystyle
	-b\tanh\lb \xi/b\rb,&\xi<0.
\end{cases}\ee
\end{remark}

\begin{theorem}
[{\cite[Theorem 1]{Xie1991}}]
\label{thm:xie}
Suppose $\alpha\geq0$ and $f$ satisfies Assumption \ref{ass:main}. There is a $\buniq=\buniq(\alpha,f)>0$ such that for each $\beta>\buniq$ there exists an orbitally unique SOPS $\sops$ of the scalar DDE \eqref{eq:dde} associated with $(\alpha,\beta,f)$.
Furthermore, $\sops$ is asymptotically stable with an exponential phase.
\end{theorem}

We are interested in so-called synchronous SOPSs of the coupled DDE \eqref{eq:cdde}, defined as follows.
A \textbf{synchronous SOPS} of the coupled DDE \eqref{eq:cdde} is a periodic solution $\sopsn$ whose components are all equal to an SOPS $\sops$ of the scalar DDE \eqref{eq:dde}; that is, $\sopsn=(\sops,\dots,\sops)^T$, where superscript ``$T$'' denotes the transpose operation.
Analogously to the scalar setting, given a synchronous $\sopsn$ of the coupled DDE \eqref{eq:cdde}, we define its {\bf orbit} in $C([-1,0],\R^n)$ by ${\orbit}_\sopsn=\{\sopsn_t,t\ge0\}$ and say it is {\bf orbitally unique} if given any other synchronous SOPS $\tilde\sopsn$ of the coupled DDE \eqref{eq:cdde} their orbits are equal; that is, $\orbit_\sopsn=\orbit_{\tilde\sopsn}$.
Under a row sum condition on $G$, it is readily verified that $\sops$ is a solution of the scalar DDE \eqref{eq:dde} if and only if $\sopsn=(\sops,\dots,\sops)^T$ is a solution of the coupled DDE \eqref{eq:cdde}.
Indeed, we have the following corollary of Theorem \ref{thm:xie}, which is stated without proof.

\begin{cor}
\label{cor:1}
Suppose $\alpha\ge0$ and $f$ satisfies Assumption \ref{ass:main}.
Given $\beta>\buniq$ let $\sops$ denote the orbitally unique SOPS for the scalar DDE \eqref{eq:dde} associated with $(\alpha,\beta,f)$.
For each $n\ge2$ and $n\times n$ real-valued matrix $G$ with row sums equal to 1, $\sopsn=(\sops,\dots,\sops)^T$ is the orbitally unique synchronous SOPS for the $n$-dimensional coupled DDE \eqref{eq:cdde} associated with $(\alpha,\beta,f,G)$.
\end{cor}

According to Theorem \ref{thm:xie}, an SOPS of the scalar DDE \eqref{eq:dde} is always asymptotically stable provided $\beta$ is sufficiently large; however, this is not the case for synchronous SOPS of the coupled DDE \eqref{eq:cdde}.
Our main results, presented below, are conditions on $\alpha$, $f$ and $G$ that describe the stability of a synchronous SOPS, in the asymptotic regime $\beta\to\infty$.

\subsection{Statements of main results}
\label{sec:main}

We have four main results on local stability of a synchronous SOPS.
Our first main result treats systems with general coupling.
Our second, third and fourth main results respectively treat the special cases of so-called weakly coupled systems, near uniformly coupled systems and doubly nonnegative coupled systems.
The proofs of our main results are given in Section \ref{sec:mainproofs}.
In Section \ref{sec:examples} we illustrate our main results with examples of mean-field coupled systems and systems arranged in a ring with nearest neighbor coupling.

Analogously to the scalar setting, we say that the synchronous SOPS $\sopsn$ is \textbf{asymptotically stable with an exponential phase} if the following property holds: there are positive constants $N$ and $\gamma$, and an open set $W\subset C([-1,0],\R^n)$ containing the orbit $\orbit_\sopsn$ such that if ${\bs\phi}\in W$ and ${\bf x}$ is the solution of the coupled DDE \eqref{eq:cdde} with initial condition ${\bf x}_t={\bs\phi}$ then there exists $s=s({\bs\phi})\ge0$ such that
	$$\sup_{\theta\in[-1,0]}|{\bf x}_t(\theta)-\sopsn_{s+t}(\theta)|\le Ne^{-\gamma t}\sup_{\theta\in[-1,0]}|{\bs\phi}(\theta)-\sopsn_s(\theta)|,\qquad t\ge0.$$
We say $\sopsn$ is {\bf unstable} if there is an open set $W\subset C([-1,0],\R^n)$ containing the orbit $\orbit_{\sopsn}$ with the property that for every subset $V\subset W$ containing the orbit $\orbit_\sopsn$ there exists ${\bs\phi}\in V$ such that if ${\bf x}$ is the solution of the coupled DDE \eqref{eq:cdde} with initial condition ${\bf x}_0={\bs\phi}$, then ${\bf x}_t\not\in W$ for some $t>0$.

\subsubsection{Systems with general coupling}
\label{sec:general}

Our first main result treats systems with general coupling.
``General coupling'' refers to any system with coupling matrix $G$ that has row sums equal to 1.
In order to present our first main result we need the following definitions.
Given $\alpha\geq0$ and $f$ satisfying Assumption \ref{ass:main}, define constants the $0<\vr_1,\vr_2<e^{-\alpha}$ by
\begin{equation}\label{eq:rho}
	\vr_1=\frac{a}{a+b}e^{-\alpha}\qquad\text{and}\qquad \vr_2=\frac{b}{a+b}e^{-\alpha},
\end{equation}
and define the complex quadratic function $\nu_\star:\CC\to\CC$ by 
\begin{equation}\label{eq:nuast}
	\nu_\star(\lambda)=\frac{(\lambda - \vr_1)(\lambda - \vr_2)}{(1 - \vr_1)(1 - \vr_2)},\qquad\lambda\in\CC,
\end{equation}
where $\CC$ denotes the complex plane.
Throughout the paper, the diamond symbol $\star$ is used to distinguish objects associated with the asymptotic regime $\beta\to\infty$ from their prelimit counterparts.
 
\begin{theorem}
	\label{thm:1}
	Suppose $\alpha\ge0$ and $f$ satisfies Assumption \ref{ass:main}.
	Let $\delta\in(0,1)$ and $K$ be a compact subset of $\CC$.
	There exists $\bmono=\bmono(\alpha,f,\delta,K)>\buniq$ such that for every $\beta>\bmono$, $n\ge2$ and $n\times n$ real-valued matrix $G$ with row sums all equal to 1 and whose spectrum, $\sigma(G)$, is contained in $K$, the following hold:
	\begin{itemize}
		\item[(i)] If the eigenvalue $\lambda=1$ of $G$ has simple algebraic multiplicity and $\abs{\nu_\star(\lambda)}<1-\delta$ for all $\lambda\in\sigma(G)\setminus\{1\}$,
		then the unique synchronous SOPS of the $n$-dimensional coupled DDE \eqref{eq:cdde} associated with $(\alpha,\beta,f,G)$ is asymptotically stable with an exponential phase.
		\item[(ii)] If $\abs{\nu_\star(\lambda)}>1+\delta$ for some $\lambda\in\sigma(G)$,
		then the unique synchronous SOPS of the $n$-dimensional coupled DDE \eqref{eq:cdde} associated with $(\alpha,\beta,f,G)$ is unstable.
	\end{itemize}
\end{theorem}

In view of Theorem \ref{thm:1} the placement of the spectrum of $G$ relative to the regions of the complex plane
	\be\label{eq:stablestarsets}\stable_\star(\delta)=\lcb\lambda\in\CC:\abs{\nu_\star(\lambda)}<1-\delta\rcb\quad\text{and}\quad\unstable_\star(\delta)=\lcb\lambda\in\CC:\abs{\nu_\star(\lambda)}>1+\delta\rcb\ee
is relevant for determining the stability of a synchronous SOPS when $\delta\in(0,1)$ and $\beta$ is sufficiently large.
The boundaries $\pd\stable_\star(\delta)=\{\lambda\in\CC:\abs{\nu_\star(\lambda)}=1-\delta\}$ and $\pd\unstable_\star(\delta)=\{\lambda\in\CC:\abs{\nu_\star(\lambda)}=1+\delta\}$, for $\delta\ge0$, are classical curves known as \textit{Ovals of Cassini}, which are described in \cite[Section 5.16]{Lawrence1972} and include the Lemniscate of Bernoulli as a special case.

\subsubsection{Weakly coupled systems}
\label{sec:weak}

Our second main result addresses the special case of so-called weakly coupled systems \cite{Ermentrout2010,Hoppensteadt,Kopell2002}.
``Weakly coupled'' refers to systems where the coupling matrix is a small perturbation of the identity matrix; that is, for some $n\ge2$ and $n\times n$ real-valued matrix $\gmatrix$ with row sums equal to zero, the coupling matrix is given by 
	$$G=I_n+\eta \gmatrix,$$ 
where $I_n$ denotes the $n\times n$ identity matrix and $\eta\ne0$ is small.
The identity $I_n+\eta\gmatrix-(1+\eta\lambda)I_n=\eta(\gmatrix-\lambda I_n)$ implies
	\be\label{eq:sigmaweakmatrix}\sigma(I_n+\eta\gmatrix)=\{1+\eta\lambda:\lambda\in\sigma(\gmatrix)\}.\ee
Since $\nu_\star(\cdot)$ is continuous and $\nu_\star(1)=1$, for every fixed $\delta\in(0,1)$ there exists $\eta\ne0$ sufficiently small such that $\lambda\not\in\stable_\star(\delta)$ and $\lambda\not\in\unstable_\star(\delta)$ for all $\lambda\in\sigma(I_n+\eta\gmatrix)$.
Therefore, Theorem \ref{thm:1} is not informative about the stability of the synchronous SOPS when $\beta$ is fixed and $\eta$ is (infinitesimally) close to 0.
Our main result for weakly coupled systems addresses the stability of such a synchronous SOPS in terms of the placement of the spectrum of $H$ relative to the imaginary axis.

\begin{theorem}
\label{thm:2}
Suppose $\alpha\geq0$ and $f$ satisfies Assumption \ref{ass:main}.  
There is a $\bnbhdone=\bnbhdone(\alpha,f)\ge\buniq$ such that for every $\beta>\bnbhdone$, $n\ge2$ and an $n\times n$ real-valued matrix $\gmatrix$ with row sums equal to zero, there exists $\eta_H>0$ such that the following hold:
\begin{itemize}
	\item[(i)] If the eigenvalue $\lambda=0$ of $H$ has simple algebraic multiplicity and $\Re\lambda<0$ (resp.\ $\Re\lambda>0$) for all $\lambda\in\sigma(H)\setminus\{0\}$, then the unique synchronous SOPS of the $n$-dimensional coupled DDE \eqref{eq:cdde} associated with $(\alpha,\beta,f,\Id_n+\eta\gmatrix)$ (resp.\  $(\alpha,\beta,f,\Id_n-\eta\gmatrix)$) is asymptotically stable with an exponential phase for all $\eta\in(0,\eta_H)$.
	\item[(ii)] If $\Re\lambda>0$ (resp.\ $\Re\lambda<0$) for some eigenvalue $\lambda\in\sigma(\gmatrix)$, then the unique synchronous SOPS of the $n$-dimensional coupled DDE \eqref{eq:cdde} associated with $(\alpha,\beta,f,\Id_n+\eta\gmatrix)$ (resp.\ $(\alpha,\beta,f,\Id_n-\eta\gmatrix)$) is unstable for all $\eta\in(0,\eta_H)$.
\end{itemize}
\end{theorem}

\subsubsection{Near uniformly coupled systems}
\label{sec:uniform}

Our third main result, which parallels Theorem \ref{thm:2} on weakly coupled systems in many respects, treats the special case of so-called near uniformly coupled systems.
``Near uniformly coupled'' refers to systems where the coupling matrix is a small perturbation of the matrix whose entries are all identical; that is, for some $n\ge2$ and an $n\times n$ real-valued matrix $H$ with row sums equal to zero, the coupling matrix is given by 
	$$G=\MF_n+\eta \gmatrix,$$ 
where $\MF_n=(\MF_n^{jk})$ denotes the $n\times n$ matrix with $\MF_n^{jk}=\frac{1}{n}$ for all $j,k=1,\dots,n$, and $\eta\ne0$ is small.
Part (i) of the result treats the case $\alpha>0$ and is a corollary of our first result, Theorem \ref{thm:1}.

On the other hand, to treat the case $\alpha=0$, we claim that for $\eta$ sufficiently small, 1 is an eigenvalue of $\MF_n+\eta\gmatrix$ with simple algebraic multiplicity and all other eigenvalues $\lambda$ satisfy $|\lambda|\le|\eta| \rho(\gmatrix)$, where $\rho(\gmatrix)=\max\{|z|:z\in\sigma(\gmatrix)\}$.
To see the claim holds, first observe that ${\bf 1}_n=(1,\dots,1)^T$ is in the null space of $\gmatrix$ and is an eigenvector of $\MF_n+\eta\gmatrix$ associated with eigenvalue 1.
Let ${\bf v}$ be an eigenvector of $\MF_n+\eta\gmatrix$, ${\bf u}=n^{-\frac12}\ip{{\bf v},{\bf 1}_n}{\bf 1}_n$ denote the orthogonal projection of ${\bf v}$ onto the span of ${\bf 1}_n$, and ${\bf w}={\bf v}-{\bf u}$ denote the orthogonal component.
Choose $|\eta|<[\rho(\gmatrix)]^{-1}$.
Then $(\MF_n+\eta\gmatrix){\bf v}={\bf u}+\eta\gmatrix{\bf w}$, so either ${\bf v}$ is in the span of ${\bf 1}_n$ or ${\bf v}$ is an eigenvector of $\gmatrix$ orthogonal to ${\bf 1}_n$ associated with eigenvalue $\lambda$ satisfying $|\lambda|<|\eta|\rho(\gmatrix)$.
This establishes the claim.

Thus, by the continuity of $\nu_\star(\cdot)$ and the fact that $\nu_\star(0)=1$ when $\alpha=0$, it follows that for every fixed $\delta\in(0,1)$ there exists $\eta\ne0$ sufficiently small such that $\lambda\not\in\stable_\star(\delta)$ and $\lambda\not\in\unstable_\star(\delta)$ for all $\lambda\in\sigma(\MF_n+\eta\gmatrix)$.
Therefore, Theorem \ref{thm:1} is not informative about the stability of the synchronous SOPS when $\beta$ is fixed and $\eta$ is (infinitesimally) close to 0.
Parts (ii) and (iii) of our main result for near uniformly coupled systems address the stability of a synchronous SOPS when $\alpha=0$ in terms of the placement of the spectrum of $H$ relative to the imaginary axis.

\begin{theorem}
\label{thm:3}
Suppose $\alpha\ge0$ and $f$ satisfies Assumption \ref{ass:main}.
There is a $\bnbhdzero=\bnbhdzero(\alpha,f)\ge\buniq$ such that given $\beta>\bnbhdzero$, $n\ge2$ and an $n\times n$ real-valued matrix $\gmatrix$ with rows sums equal to zero, there exists $\eta_H>0$ such that the following hold:
\begin{itemize}
	\item[(i)] If $\alpha>0$, the unique synchronous SOPS of the $n$-dimensional coupled DDE \eqref{eq:cdde} associated with $(\alpha,\beta,f,\MF_n+\eta\gmatrix)$ is asymptotically stable with an exponential phase for all $\eta\in(-\eta_H,\eta_H)$.
	\item[(ii)] If $\alpha=0$, the eigenvalue $\lambda=0$ of $H$ has simple algebraic multiplicity, and $\Re\lambda>0$ (resp.\ $\Re\lambda<0$) for all $\lambda\in\sigma(H)\setminus\{0\}$, then the unique synchronous SOPS of the $n$-dimensional coupled DDE \eqref{eq:cdde} associated with $(0,\beta,f,\MF_n+\eta\gmatrix)$ (resp.\ $(0,\beta,f,\MF_n-\eta\gmatrix)$) is asymptotically stable with an exponential phase for all $\eta\in(0,\eta_H)$.
	\item[(iii)] If $\alpha=0$ and $\Re\lambda<0$ (resp.\ $\Re\lambda>0$) for some eigenvalue $\lambda\in\sigma(\gmatrix)$, then the unique synchronous SOPS of the $n$-dimensional coupled DDE \eqref{eq:cdde} associated with $(0,\beta,f,\MF_n+\eta\gmatrix)$ (resp.\ $(0,\beta,f,\MF_n-\eta\gmatrix)$) is unstable for all $\eta\in(0,\eta_H)$.
\end{itemize}
\end{theorem}

\subsubsection{Doubly nonnegative coupled systems}
\label{sec:posdef}

Our final result treats stability of synchronous SOPS for doubly nonnegative coupled systems.
``Doubly nonnegative coupled'' refers to the case that the coupling matrix $G$ is irreducible and doubly nonnegative, where we recall that a real-valued $n\times n$ matrix is \textbf{doubly nonnegative} if it has nonnegative entries and is positive semidefinite.
Doubly nonnegative systems arise, for example, in the study of systems with so-called mean-field coupling and systems of DDEs arranged in a ring with symmetric coupling strengths (see Section \ref{sec:examples}).
Given $\alpha\ge0$ and $f$ satisfying Assumption \ref{ass:main} recall definition \eqref{eq:rho} for $\vr_1,\vr_2\in(0,e^{-\alpha})$ and define the constants
	\be\label{eq:r0}r_0=\frac{e^{-\alpha}}{2}\ee 
and
	\be\label{eq:Delta}\Delta=\lb\frac{\vr_1-\vr_2}{2}\rb^2-(1-\vr_1)(1-\vr_2)=\lb\frac{a-b}{a+b}\rb^2\lb\frac{e^{-\alpha}}{2}\rb^2-\lb1-\frac{ae^{-\alpha}}{a+b}\rb\lb1-\frac{be^{-\alpha}}{a+b}\rb.\ee

\begin{theorem}
\label{thm:4}
Suppose $\alpha\ge0$ and $f$ satisfies Assumption \ref{ass:main}.
Suppose $\Delta<0$.
There exists $\bzoa=\bzoa(\alpha,f)\ge\buniq$ such that for every $\beta>\bzoa$, $n\ge2$ and $n\times n$ irreducible doubly nonnegative coupling matrix $G$ with row sums all equal to 1 (if $\alpha=0$ then additionally assume $G$ is positive definite) the following holds:
\begin{itemize}
	\item [(i)] The unique synchronous SOPS of the $n$-dimensional coupled DDE \eqref{eq:cdde} associated with $(\alpha,\beta,f,G)$ is asymptotically stable with an exponential phase.
\end{itemize}
On the other hand, suppose $\Delta>0$.
Let
	\be\label{eq:ve}0<\ve<\min\lb\sqrt{\Delta},r_0-\sqrt{\Delta}\rb.\ee   
There exists $\bzob=\bzob(\alpha,f,\ve)\ge\buniq$ such that for every $\beta>\bzob$, $n\ge2$ and $n\times n$ irreducible doubly nonnegative coupling matrix $G$ with row sums all equal to 1 (if $\alpha=0$ then additionally assume $G$ is positive definite) the following hold:
\begin{itemize}
	\item[(ii)] If $\abs{\lambda-r_0}>\sqrt{\Delta}+\ve$ for all $\lambda\in\sigma(G)\setminus\{1\}$, then the unique synchronous SOPS of the $n$-dimensional coupled DDE \eqref{eq:cdde} associated with $(\alpha,\beta,f,G)$ is asymptotically stable with an exponential phase.
	\item[(iii)] If $\abs{\lambda-r_0}<\sqrt{\Delta}-\ve$ for some $\lambda\in\sigma(G)$, then the unique synchronous SOPS of the $n$-dimensional coupled DDE \eqref{eq:cdde} associated with $(\alpha,\beta,f,G)$ is unstable.
\end{itemize}
\end{theorem}

\begin{remark}
A straightforward calculation shows if $\alpha>\ln\frac{1+\sqrt{2}}{2}$, where $\ln$ denotes the natural logarithm, then $\Delta<0$ for all $a,b>0$.
In this case, for every $f$ satisfying Assumption \ref{ass:main}, $\beta>\bzoa=\bzoa(\alpha,f)$, $n\ge2$ and $n\times n$ irreducible doubly nonnegative coupling matrix $G$ with row sums all equal to 1, the unique synchronous SOPS of the $n$-dimensional coupled DDE \eqref{eq:cdde} associated with $(\alpha,\beta,f,G)$ is asymptotically stable with an exponential phase.
\end{remark}

\subsection{Relation to prior work}
\label{sec:related}

There is considerable literature on stability of synchronous oscillatory periodic solutions for coupled systems of DDEs.
Here we mention the works that treat systems of DDEs with a single fixed delay and allow for a nonlinearity $f$ satisfying a negative feedback condition.
The first body of related work is on systems of three coupled oscillators. 
The works \cite{Guo2009,Wu1999} employ the theory of symmetric local Hopf bifurcation (see, e.g., \cite{Wu1998}) to study pattern formations that arise under different coupling strengths.
The next class considers systems of DDEs arranged in a ring.
In \cite{Chen2000} Chen, Huang and Wu show that the synchronous SOPS is unstable when either $n$ is even or $n$ is odd and sufficiently large.
The works \cite{Guo2003,Guo2007} prove related results when allowing for a more general nonlinear feedback structure.
We obtain a comparable result, in the asymptotic regime $\beta\to\infty$, for DDEs arranged in a ring (see Remark \ref{rem:chen}).

The next body of related work allows for general coupling and uses the so-called master stability function to characterize the stability of the coupled system \cite{Choe2010,Dhamala2004,Flunkert2010,Flunkert2014,Kinzel2009,Sieber2013}.
Of particular interest, Flunkert et al.\ \cite{Flunkert2014} and Sieber et al.\ \cite{Sieber2013} consider a system of multidimensional periodic oscillators and allow for a general class of nonlinear feedback as well as a general coupling structure with delayed coupling.
They show that the characteristic multipliers for the synchronous periodic solution can be expressed in terms of the spectra of monodromy-type maps associated with a single system, which is an approach that we adopt here. 
The main difference between their work and ours is that they consider systems where the coupling delay significantly exceeds the period of oscillation, which is relevant in the context of  optical networks.
This is in contrast to our work, where we assume that the period is greater than the delay, which is relevant in systems where the oscillations arise because of the delay (e.g., in gene regulatory networks \cite{glass2016signaling,Mather2014}).
Moreover, this difference in structure leads to an interesting contrast in the results we obtain.
For instance, in \cite{Flunkert2014} the authors demonstrate that if the spectrum of the coupling matrix lies in a disc in the complex plane centered at the origin, then for sufficiently large delays, the synchronous oscillating periodic solution is stable; whereas in Theorem \ref{thm:1} we show that if the spectrum of the coupling matrix lies in the region $\stable_\star(0)$, defined in \eqref{eq:stablestarsets}, then for sufficiently large delays, the synchronous SOPS is stable, and the region $\stable_\star(0)$ is not necessarily convex or even simply connected (see Remark \ref{rem:OvalsofCassini}).

Finally we mention the work of Campbell and Wang \cite{Campbell2018,Wang2017} on stability of synchronous periodic solutions in the case of circulant coupling and when the coupling between oscillators is weak (this setting is closely related to the weakly coupled regime considered in Section \ref{sec:weak}).
They approximate their system with a phase coupled model, which is in contrast to our amplitude coupled model.
They obtain simple criteria for the stability of synchronous oscillatory periodic solutions, which are expressed in terms of the relation between the spectrum of the coupling matrix and the imaginary axis (\cite[Theorem 2]{Campbell2018}).
We obtain a comparable result for weakly coupled systems in the asymptotic regime $\beta\to\infty$ for a more general class of coupling matrices (Theorem \ref{thm:2}).

\subsection{Outline}
\label{sec:outline}

In Section \ref{sec:msf} we introduce the basis for our main results, Theorem \ref{thm:msf}, which states that the characteristic multipliers associated with the synchronous SOPS can be expressed in terms of the spectra of a family of monodromy-type operators $\{\mono_\lambda\}$, indexed by $\lambda$ in the spectrum of the coupling matrix $G$, associated with following parameterized family of (complex) scalar linear variational-type equations about the SOPS $\sops$ of the scalar DDE:
	\be\label{eq:eve}\dot{y}(t)=-\alpha y(t)+\lambda\beta f'(\sops(t-1))y(t-1).\ee
We refer to linear DDE \eqref{eq:eve} as the extended variational equation (associated with $\lambda$) along $\sops$, and we refer to $\{\mono_\lambda\}$ as extended monodromy operators associated with $\sops$.
In this way stability of the synchronous SOPS is reduced to characterizing the spectra of $\{\mono_\lambda\}$ and in Section \ref{sec:asym_char} we state our four results on the spectra of $\{\mono_\lambda\}$ in the asymptotic regime $\beta\to\infty$.
The approach is in the spirit of works by Pecora and Carroll \cite{Pecora1998} on stability of synchronous solutions for coupled systems of ordinary differential equations, and by Flunkert et al.\ \cite{Flunkert2010,Flunkert2014}, who generalized the approach to coupled system of equations, with a delay introduced due to the coupling.
The proof of Theorem \ref{thm:msf} is given in Section \ref{sec:msfproof}.

The advantage of analyzing the system in the limiting regime $\beta\to\infty$ is that the limiting solutions of the extended variational equation \eqref{eq:eve} have explicit expressions.
In Section \ref{sec:normalizedSOPS} we recall results from \cite{Xie1991} on the convergence of certain normalized SOPS, as $\beta\to\infty$.
We say $z$ is the {\bf dominant multiplier} of $M_\lambda$ if $z$ is in the spectrum of $M_\lambda$ and $|z|>|\tilde z|$ for all other eigenvalues in the spectrum of $M_\lambda$.
In Section \ref{sec:limit_char} we prove that the dominant multiplier of $\mono_\lambda$ exists for $\beta$ sufficiently large and converges to $\nu_\star(\lambda)$ as $\beta\to\infty$.
One of the principal challenges is that the term $\beta f'(\sops(t-1))$ appearing in extended variational equation \eqref{eq:eve} does not have a functional limit.
In order to identify the correct limit for solutions to the extended variational equation \eqref{eq:eve}, it is advantageous to instead view $\beta f'(\sops(t-1))dt$ as a signed measure on the real line.
After identifying the correct limiting measure we prove convergence properties for solutions of the extended variational equation \eqref{eq:eve} and the dominant multipliers of  $\{\mono_\lambda\}$, as $\beta\to\infty$.

In Section \ref{sec:boundarycases} we prove that the extended monodromy operators $\{\mono_\lambda\}$ are holomorphic in $\lambda$, and the derivative of the dominant multiplier $\mono_\lambda$ converges to $\partial\nu_\star(\lambda)/\partial\lambda$ as $\beta\to\infty$.
In Section \ref{sec:description} we use this convergence result to characterize the extended characteristic multipliers near $\lambda=1$ and $\lambda=0$, in the regime $\beta\to\infty$.
As a corollary, we also characterize the extended characteristic multipliers for $\lambda$ in the interval $[0,1)$.
The points $\lambda=1$ and $\lambda=0$ are special because the dominant multipliers of $\mono_1$ (if $\alpha\ge0$) and $\mono_0$ (if $\alpha=0$) are both equal to 1, for all $\beta$ sufficiently large (see Proposition \ref{prop:nu_lambda_simple}).
In Section \ref{sec:mainproofs} we use Theorem \ref{thm:msf} and the asymptotic characterization of the extended characteristic multipliers to prove our main results.

In Section \ref{sec:examples} we illustrate our main results with two interesting applications of a mean-field coupled system and a system of DDEs arranged in a ring with nearest-neighbor coupling.
In Section \ref{sec:numerics} we plot solutions of a 3-dimensional coupled DDE for different parameter values.

\subsection{Notation }
\label{sec:notation}

We now collect some commonly used notation. 

\subsubsection{Basic notation}
\label{sec:basicnotation}
 
Let $\N$ denote the natural numbers, $\Z$ denote the integers, $\R$ denote the real line, $\R_+$ denote the nonnegative axis and $\CC$ denote the complex plane.
For $r\in\R$ let $\lfloor r\rfloor=\max\{j\in\Z:j\le r\}$.
Let $i=\sqrt{-1}$. 
For $z\in\CC$ let $\Re z$ and $\Im z$ denote the real and imaginary parts of $z$, respectively, and let $\bar z$ denote the complex conjugate of $z$.
For $r\in\R$ define the open half planes
\begin{equation}\label{eq:halfplane}
	\CC_{<r}=\{z\in\CC:\Re z<r\}\qquad\text{and}\qquad\CC_{>r}=\{z\in\CC:\Re z>r\}.
\end{equation}
For a set $Z\subset\CC$, let $\interior(Z)$ denote its interior, $\clo(Z)$ denote its closure, $\partial Z=\clo(Z)\setminus\interior(Z)$ denote its boundary, $1_Z(\cdot)$ denote the indicator function on $Z$ and, given $r\in\R$, let $rZ=\{rz:z\in Z\}$.
Given another set $Y \subset \CC$ let $d_H(Y,Z)$ denote the Hausdorff metric defined by
	$$d_H(Y,Z)=\max\lcb\sup_{y\in Y}\inf_{z\in Z}|y-z|,\sup_{z\in Z}\inf_{y\in Y}|y-z|\rcb.$$

For an integer $n\geq 2$, let $\R^n$ denote $n$-dimensional Euclidean space and $\CC^n$ denote $n$-dimensional complex space. 
We use bold typeface for vectors in $\R^n$ and $\CC^n$. 
Given a column vector ${\bf v}$ in $\R^n$ or $\CC^n$, let ${\bf v}^T$ denote its transpose and let $v^j$ denote its $j$\textsuperscript{th} component, for $j=1,\dots,n$.
Let ${\bf 1}_n=(1,\dots,1)^T$ denote the column vector in $\R^n$ whose entries all equal 1.
For $r\in\{0,1\}$, let $\M_n^r$ denote the subset of real-valued $n\times n$ matrices whose row sums are all equal to $r$; that is, 
	\be\label{eq:Mnr}\M_n^r=\lcb n\times n\text{ real-valued matrices $\gmatrix$ satisfying }\gmatrix{\bf 1}_n=r{\bf 1}_n\rcb.\ee
Let $\Id_n\in\M_n^1$ denote the $n\times n$ identity matrix and $\MF_n\in\M_n^1$ denote the $n\times n$ matrix whose entries are all equal to $1/n$, i.e., 
\be\label{eq:Jn}\MF_n^{jk}=\frac{1}{n},\qquad j,k=1,\dots,n.\ee

Given a Banach space $(\banach,\norm{\cdot}_\banach)$, an element $\chi\in\banach$ and $r>0$ we let $\ball_\banach(\chi,r)=\{\wt\chi\in\banach:\norm{\wt\chi-\chi}_\banach<r\}$ denote the ball of radius $r$ centered at $\chi$.
When $\banach=\CC$, we drop the subscript $\CC$ and write $\ball(\chi,r)$ for $\ball_\CC(\chi,r)$.

We abbreviate ``such that'' as ``s.t.''.

\subsubsection{Function spaces}
\label{sec:notation_function_spaces}

Given an interval $\interval\subset\R$ and $\banach\in\{\R,\CC,\R^n,\CC^n\}$ let $D(\interval,\banach)$ denote the set of functions from $\interval$ to $\banach$ that are right continuous and have finite left limits.
Let $C(\interval,\banach)$ denote the subset of continuous functions in $D(\interval,\banach)$, and let $C_c(\interval,\banach)$ denote the subset of functions in $C(\interval,\banach)$ with compact support. 
We equip $D(\interval,\banach)$ and its subsets with the topology of uniform convergence on compact subsets of $\interval$. 
When $\interval=[-1,0]$, we use the abbreviated notation $\D(\banach)=D([-1,0],\banach)$ and $\C(\banach)=C([-1,0],\banach)$. 
Define the uniform norm on $\D(\banach)$ by 
	$$\norm{\phi}_{[-1,0]}=\sup_{\theta\in[-1,0]}\abs{\phi(\theta)}<\infty.$$
Then $(\D(\banach),\norm{\cdot}_{[-1,0]})$ and $(\C(\banach),\norm{\cdot}_{[-1,0]})$ are Banach spaces.

The following definitions are stated for $\field\in\{\R,\CC\}$, $n\ge 2$ and $s\ge0$. 
We use bold typeface for $\field^n$-valued functions. 
For $x\in\D([s-1,\infty),\field)$ (resp.\ ${\bf x}\in\D([s-1,\infty),\field^n)$) and $t\ge s$ define $x_t\in\D(\field)$ (resp.\ ${\bf x}_t\in\D(\field^n)$) by $x_t(\theta)=x(t+\theta)$ (resp.\ ${\bf x}_t(\theta)={\bf x}(t+\theta)$) for $\theta\in[-1,0]$. 
Given an interval $\mathbb{I}\subset\R$ a function $x\in C(\mathbb{I},\field)$ (resp.\ ${\bf x}\in C(\mathbb{I},\field^n)$) and $t\in\mathbb{I}$ such that $x$ (resp. ${\bf x}$) is differentiable at $t$, we let $\dot x(t)$ (resp.\ $\dot{\bf x}(t)$) denote the derivative of $x$ (resp.\ ${\bf x}$) at $t$.

For $p\in\{1,\infty\}$ we let $L^p(\R)$ denote the Banach space of Lebesgue measurable functions $g:\R\to\R$ with finite $L^p$-norm, where functions that are almost everywhere equal are identified.
Given a constant $N<\infty$ and a subset $\interval\subset\R$ let
	$$L_N^\infty(\R)=\lcb g\in L^\infty(\R):\norm{g}_{L^\infty(\R)}\le N\rcb$$
denote the subset of functions in $L^\infty(\R)$ bounded by $N$ and let
	\be\label{eq:LCAbounded}L_{N,\interval}^\infty(\R)=\lcb g\in L_N^\infty(\R):g\text{ is Lipschitz continuous on }\interval\text{ with Lipschitz constant }N\rcb.\ee
denote the further subset of functions that are Lipschitz continuous on $\interval$ with Lipschitz constant $N$.

\subsubsection{Solutions of the DDEs}

Let $\alpha\ge0$, $\beta>0$, $f:\R\to\R$ and $G\in\M_n^1$ be given.
Under Assumption \ref{ass:main}, $f$ is uniformly Lipschitz continuous and it follows from \cite[Chapter 2, Theorem 2.3]{Hale1993} that for each $\phi\in\C(\R)$ there exists a unique solution of the scalar DDE \eqref{eq:dde} associated with $(\alpha,\beta,f)$ starting at $\phi$, which we denote by $x(\phi)$. 
Similarly, for each ${\bs\phi}\in\C(\R^n)$ there exists a unique solution of the coupled DDE \eqref{eq:cdde} associated with $(\alpha,\beta,f,G)$ starting at ${\bs\phi}$, which we denote by ${\bf x}({\bs{\phi}})$.

\subsubsection{Bounded and compact linear operators}
\label{sec:notation_linear_operators}

Given Banach spaces $(\banach,\norm{\cdot}_{\banach})$ and $(\banachY,\norm{\cdot}_{\banachY})$, let $B(\banach,\banachY)$ denote the vector space of bounded (continuous) linear operators from $\banach$ to $\banachY$ equipped with the operator norm $\norm{\cdot}$. 
Let $B_0(\banach,\banachY)$ denote the subset of compact linear operators in $B(\banach,\banachY)$.
When $\banachY=\banach$, we use the abbreviations $B(\banach)=B(\banach,\banach)$ and $B_0(\banach)=B_0(\banach,\banach)$.
We let $\Id_\banach$ denote the identity operator in $B(\banach)$.  
Given $A \in B(\banach)$ let $\sigma(A)$ denote the spectrum of $A$ in $\CC$, and let $\rho(A)$ denote the spectral radius of $A$.
If $\lambda\in\sigma(A)$ is such that $|\lambda|>|z|$ for all $z\in\sigma(A)\setminus\{\lambda\}$ then $\lambda$ is referred to as the dominant multiplier of $A$.
For $r\in\{0,1\}$ define
\begin{equation}
\label{eq:sigmaminus1}
	\sigma_{-r}(A)= 
	\begin{cases}
		\sigma(A)\setminus\{r\} & \text{if $r$ is a simple eigenvalue of $A$}, \\
		\sigma(A) & \text{otherwise}.
	\end{cases}
\end{equation}
Recall that $\sigma(\cdot)$ is a continuous function from $B_0(\banach)$ to $\B(\CC)$, where $\B(\CC)$ denotes the Borel subsets of $\CC$ equipped with the Hausdorff metric.
For $\lambda\in\sigma(A)$ let $m_A(\lambda)$ equal the dimension of the generalized eigenspace $E_\lambda=\cup_{j=1}^\infty \Null(A-\lambda\Id_\banach)^j$.
For a Banach space $\banach$ over $\CC$ let $\banach^\ast=B(\banach,\CC)$ denote the dual of $\banach$ equipped with the operator norm.  
For $A\in B_0(\banach)$, let $A^\ast\in B_0(\banach^\ast)$ denote the adjoint of $A$. 
Then $\sigma(A)=\sigma(A^\ast)$ and $m_A(\lambda)=m_{A^\ast}(\lambda)$ for all $\lambda\in\CC$ (see, e.g., \cite[Chapter 6, Theorem 6.1]{Conway1990}).

Let $\mathcal{M}(\R)$ denote the set of signed measures on the $\sigma$-algebra of Borel subset of $\R$.
Each element $\mu\in\mathcal{M}(\R)$ can be viewed as an element of $B(C_c(\R,\R),\R)$, as follows:
	$$\mu(g)=\int_\R g(t)d\mu(t),\qquad g\in C_c(\R,\R).$$
Given a family $\{\mu_r\}_{r>0}$ in $\mathcal{M}(\R)$, we say $\mu_r$ converges to $\mu\in\mathcal{M}(\R)$ in the weak-$^\ast$ topology if $\lim_{r\to\infty}\mu_r(g)=\mu(g)$ for all $g\in C_c(\R,\R)$.

\section{Extended characteristic multipliers}
\label{sec:msf}

In this section we express the characteristic multipliers of a synchronous SOPS $\sopsn=\sops{\bf 1}_n$ in terms of so-called extended characteristic multipliers of the SOPS $\sops$.
We then describe the extended characteristic multipliers in the regime where $\beta$ is large.
In Section \ref{sec:mainproofs} we use this relation between the characteristic multipliers of $\sopsn$ and the extended characteristic multipliers of $\sops$, along with the asymptotic characterization of the extended characteristic multipliers, to prove our main results.

\subsection{Background and definitions}
\label{sec:msf_background}

Recall the definition of $\M_n^1$ given in \eqref{eq:Mnr}.
Fix $\alpha\ge0$, $\beta>\buniq$, $n\ge2$ and $G\in\M_n^1$, and let $\sopsn=\sops{\bf 1}_n$ denote the unique synchronous SOPS with period $\period$ of the $n$-dimensional SOPS for the coupled DDE \eqref{eq:cdde} associated with $(\alpha,\beta,f,G)$.
Linearizing the coupled DDE \eqref{eq:cdde} about $\sopsn$ yields the $n$-dimensional (linear) variational equation
	\be\label{eq:nvariational}\dot{y}^j(t)=-\alpha y^j(t)+\beta\sum_{k=1}^nG^{jk}f'(\sops(t-1))y^k(t-1),\qquad j=1,\dots,n.\ee
Define the \textbf{monodromy operator} ${\bf\mono}:\C(\CC^n)\to\C(\CC^n)$ associated with $\sopsn$ by
	\be\label{eq:nmono}{\bf\mono}{\vfn}={\bf y}_\period({\vfn}),\qquad{\vfn}\in\C(\CC^n),\ee
where ${\bf y}={\bf y}({\bs\vf})=(y^1({\bs\vf}),\dots,y^n({\bs\vf}))^T$ denotes the unique function in $C([-1,\infty),\CC^n)$ that satisfies ${\bf y}_0={\bs{\vf}}$, is continuously differentiable on $(0,\infty)$ and satisfies the variational equation \eqref{eq:nvariational} for all $t>0$.
Then ${\bf M}$ is a compact linear operator (see, e.g., \cite[Chapter III, Corollary 4.7]{Diekmann1991}), the elements of its spectrum, $\sigma({\bf M})$, are invariant under time-translations of $\sopsn$ (see, e.g., \cite[Chapter XIV, p.\ 367]{Diekmann1991}), and the nonzero elements of $\sigma({\bf\mono})$ are referred to as \textbf{characteristic (Floquet) multipliers} of the periodic solution $\sopsn$.
We say $\mu$ is a \textbf{simple} characteristic multiplier if $\mu$ is a simple eigenvalue of ${\bf\mono}$.
Since $\sopsn$ satisfies the coupled DDE \eqref{eq:cdde}, it follows that its time derivative $\dot\sopsn$ is periodic with period $\period$ and satisfies the variational equation \eqref{eq:nvariational}, and so $\mu=1$ is always an eigenvalue of ${\bf\mono}$ with associated eigenvector $\dot\sopsn_0$, and we refer to $\mu=1$ as the trivial characteristic multiplier.
A synchronous SOPS $\sopsn$ for the coupled DDE \eqref{eq:cdde} is {\bf linearly stable} if the trivial characteristic multiplier $\mu=1$ of $\sopsn$ is simple and $\abs{\mu}<1$ for every nontrivial characteristic multiplier of $\sopsn$; and a synchronous SOPS is {\bf linearly unstable} if $\abs{\mu}>1$ for some characteristic multiplier of $\sopsn$.

\begin{remark}
	\label{rem:linearstability}
	Recall the definitions of ``asymptotically stable with an exponential phase'' and ``unstable'' given in Section \ref{sec:main}.
	If a synchronous SOPS is linearly stable then it is asymptotically stable with an exponential phase.
	On the other hand, if a synchronous SOPS is linearly unstable then it is unstable (see, e.g., \cite[Chapter 10, Theorem 3.1 \& Corollary 3.1]{Hale1977}).
\end{remark}

We now define extended characteristic multipliers associated with the SOPS $\sops$.
Recall the extended variational equation \eqref{eq:eve} along $\sops$, parameterized by $\lambda\in\CC$. 
For $\lambda\in\CC$ and $\vf\in\C(\CC)$ let $y=y(\lambda,\vf)$ denote the unique function in $C([-1,\infty),\CC)$ that satisfies $y_0=\vf$, is continuously differentiable on $(0,\infty)$ and satisfies the extended variational equation \eqref{eq:eve} for all $t>0$.
Define the family $\{\mono_\lambda\}=\{\mono_\lambda,\lambda\in\CC\}$ of \textbf{extended monodromy operators} associated with $\sops$ as follows.
For each $\lambda\in\CC$ define $\mono_\lambda:\C(\CC)\to\C(\CC)$ by
	\be\label{eq:gmonodromy}\mono_\lambda\vf=y_\period(\lambda,\vf),\qquad \vf\in\C(\CC).\ee
For each $\lambda\in\CC$, $\mono_\lambda$ is a compact linear operator and its spectrum $\sigma(\mono_\lambda)$ is invariant under time translations of the SOPS (see, e.g., \cite[Chapter XIII.3]{Diekmann1991}).
We refer to the nonzero eigenvalues of the extended monodromy operators $\{\mono_\lambda\}$ as \textbf{extended characteristic (Floquet) multipliers} of the SOPS $\sops$.

When $\lambda=1$, the extended variational equation \eqref{eq:eve} reduces to the following variational equation, which is obtained by linearizing the scalar DDE \eqref{eq:dde} about the SOPS $\sops$:
\begin{align}\label{eq:1variational}
	\dot y(t)=-\alpha y(t)+\beta f'(p(t-1))y(t-1).
\end{align}
The {\bf monodromy operator} associated with the SOPS $\sops$ is the compact linear operator $M$ that maps an element $\psi\in\C(\CC)$ to the solution $y$ of the variational equation \eqref{eq:1variational} with initial condition $y_0=\psi$.
Thus, the monodromy operator $M$ is equal to the extended monodromy operator $M_1$.
The nonzero elements of $\sigma(M)$ are referred to as the characteristic (Floquet) multipliers of the periodic solution $\sops$.
Thus, according to \cite[Theorem 1]{Xie1991}, if $\beta>\buniq$ then $\mu=1$ is a simple characteristic multiplier of $\mono$ and every other multiplier satisfies $\abs{\mu}<1$. 

\subsection{Relation to stability of a synchronous SOPS}

The following theorem, whose proof is given in Section \ref{sec:msfproof}, expresses the characteristic multipliers of the synchronous SOPS $\sopsn=\sops{\bf 1}_n$ in terms of the extended characteristic multipliers of the SOPS $\sops$.
Recall the definitions of $\M_n^1$ and $\sigminus(\cdot)$ given in \eqref{eq:Mnr} and \eqref{eq:sigmaminus1}, respectively.

\begin{theorem}
	\label{thm:msf}
	Suppose $\alpha\ge0$, $f$ satisfies Assumption \ref{ass:main}.
	Given $\beta>\buniq$, $n\ge2$ and $G\in\M_n^1$, let $\sopsn=\sops{\bf 1}_n$ denote the unique synchronous SOPS associated with $(\alpha,\beta,f,G)$.
	Let $\bf\mono$ denote the monodromy operator associated with the synchronous SOPS $\sopsn$ and let $\{\mono_\lambda\}$ denote the family of extended monodromy operators associated with the SOPS $\sops$.
	Then
	\be\label{eq:Floquet}\sigma({\bf\mono})=\bigcup_{\lambda\in\sigma(G)}\sigma(\mono_\lambda).\ee
	Consequently, the following hold:
	\begin{itemize}
		\item[(i)] If $\rho(\mono_\lambda)<1$ for all $\lambda\in\sigminus(G)$ then the unique synchronous SOPS of the $n$-dimensional coupled DDE \eqref{eq:cdde} associated with $(\alpha,\beta,f,G)$ is linearly stable.
		\item[(ii)] If $\rho(\mono_\lambda)>1$ for some $\lambda\in\sigma(G)$ then the unique synchronous SOPS of the $n$-dimensional coupled DDE \eqref{eq:cdde} associated with $(\alpha,\beta,f,G)$ is linearly unstable.
	\end{itemize}
\end{theorem}

It follows from Theorem \ref{thm:msf} that the regions $\stable$ and $\unstable$ in $\CC$, defined by
	\be\label{eq:stablesets}\stable=\{\lambda\in\CC:\rho(\mono_\lambda)<1\},\qquad\unstable=\{\lambda\in\CC:\rho(\mono_\lambda)>1\},\ee
which do not depend on the coupling matrix, are important for determining if a synchronous SOPS is linearly stable or unstable.
In particular, given a characterization of $\stable$ and $\unstable$, linear stability of the synchronous SOPS $\sopsn$ associated with a given coupling matrix $G\in\M_n^1$ depends on the placement of its spectrum $\sigma(G)$ --- specifically, the associated synchronous SOPS $\sopsn$ is linearly stable if $\sigminus(G)\subset\stable$, and linearly unstable if $\sigma(G)\cap\unstable\ne\emptyset$.
The remainder of this work is largely devoted to characterizing the sets $\stable$ and $\unstable$ when $\beta$ is large.

\subsection{Asymptotic characterization}
\label{sec:asym_char}

We now state four results describing the regions $\stable$ and $\unstable$ when $\beta$ is sufficiently large, which, along with Theorem \ref{thm:msf}, will respectively be used to prove our four main results in Section \ref{sec:mainproofs}.

\subsubsection{Limits of the extended characteristic multipliers}

In the following theorem, whose proof is given in Section \ref{sec:proofmonodromy}, we describe the limit of the spectrum $\sigma(\mono_\lambda)$ as $\beta\to\infty$.
Recall the definition of $\nu_\star(\lambda)$ given in \eqref{eq:nuast} and the definitions for the sets $\stable_\star(\delta)$ and $\unstable_\star(\delta)$ given in \eqref{eq:stablestarsets}.

\begin{theorem}\label{thm:1a}
	Suppose $\alpha\ge0$ and $f$ satisfies Assumption \ref{ass:main}.
	Let $\delta\in(0,1)$ and $K$ be a compact subset of $\CC$.
	There is a $\bmono=\bmono(\alpha,f,K,\delta)\ge\buniq$ such that if $\beta>\bmono$ then
		\be\label{eq:dHsigUnu} d_H(\sigma(\mono_\lambda),\{0,\nu_\star(\lambda)\})<\delta,\qquad\lambda\in K.\ee
	Consequently, $\stable_\star(\delta)\cap K\subset\stable\cap K$ and $\unstable_\star(\delta)\cap K\subset\unstable\cap K$ for all $\beta>\bmono$.
\end{theorem}

\begin{figure}
	\includegraphics[width=1\textwidth]{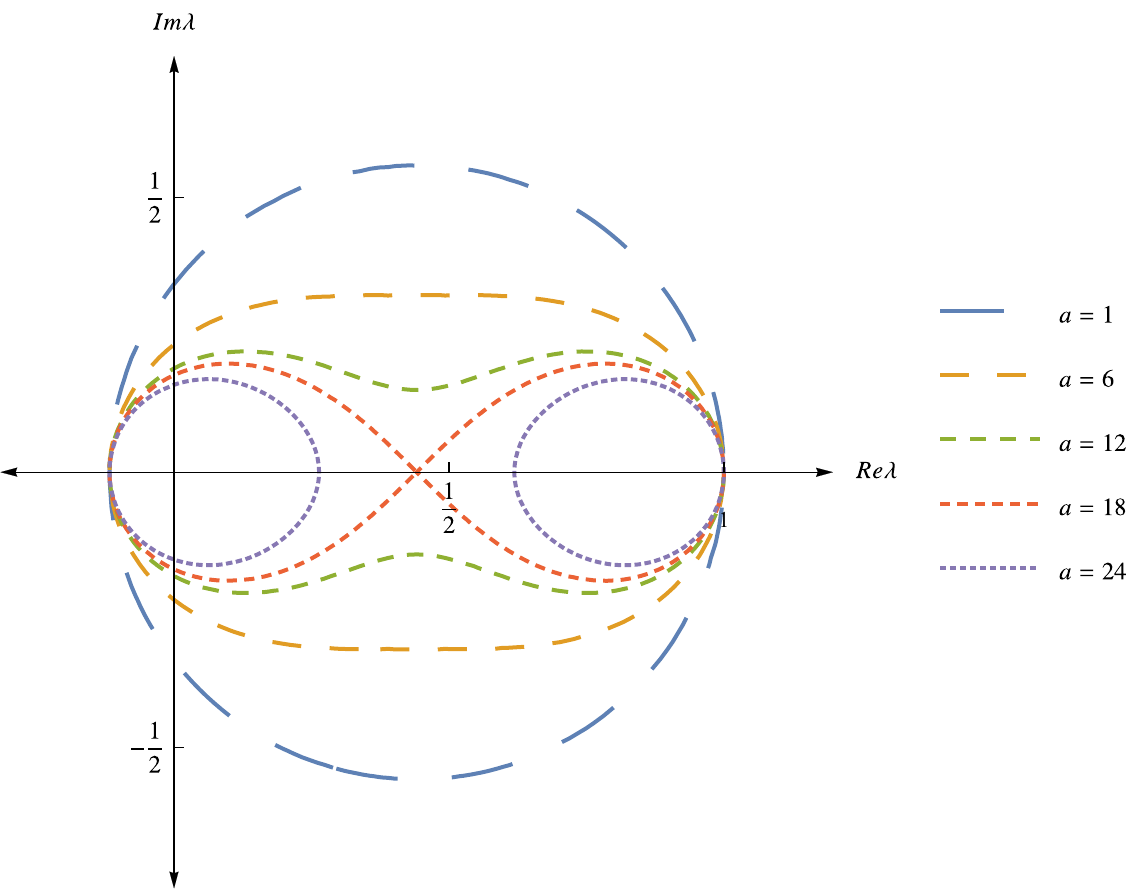}
	\caption[Plots of $\pd\stable_\star(0)$]{
		Plots of the boundaries $\pd\stable_\star(0)=\{\lambda\in\CC:\abs{\nu_\star(\lambda)}=1\}$, where $\nu_\star(\lambda)$ is defined in \eqref{eq:nuast}, for fixed $\alpha=\frac{1}{8}$ and $b=1$, and $a$ varies between 1 and 24.}
	\label{fig:ovals}
\end{figure}

\begin{figure}
	\begin{subfigure}{\textwidth}
		\centering
		\includegraphics{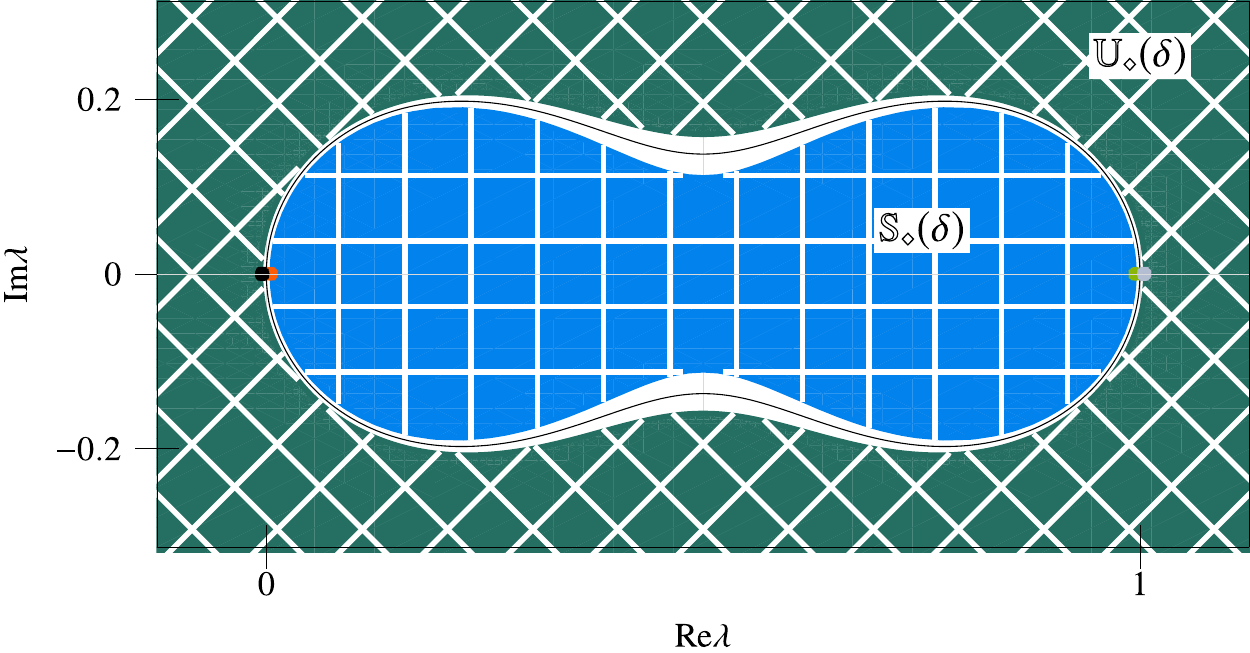}
		\caption{The sets $\stable_\star(\delta)$ and $\unstable_\star(\delta)$.}
	\end{subfigure}\newline\newline\newline
	\begin{subfigure}{0.49\textwidth}
		\centering
		\includegraphics[width=\linewidth]{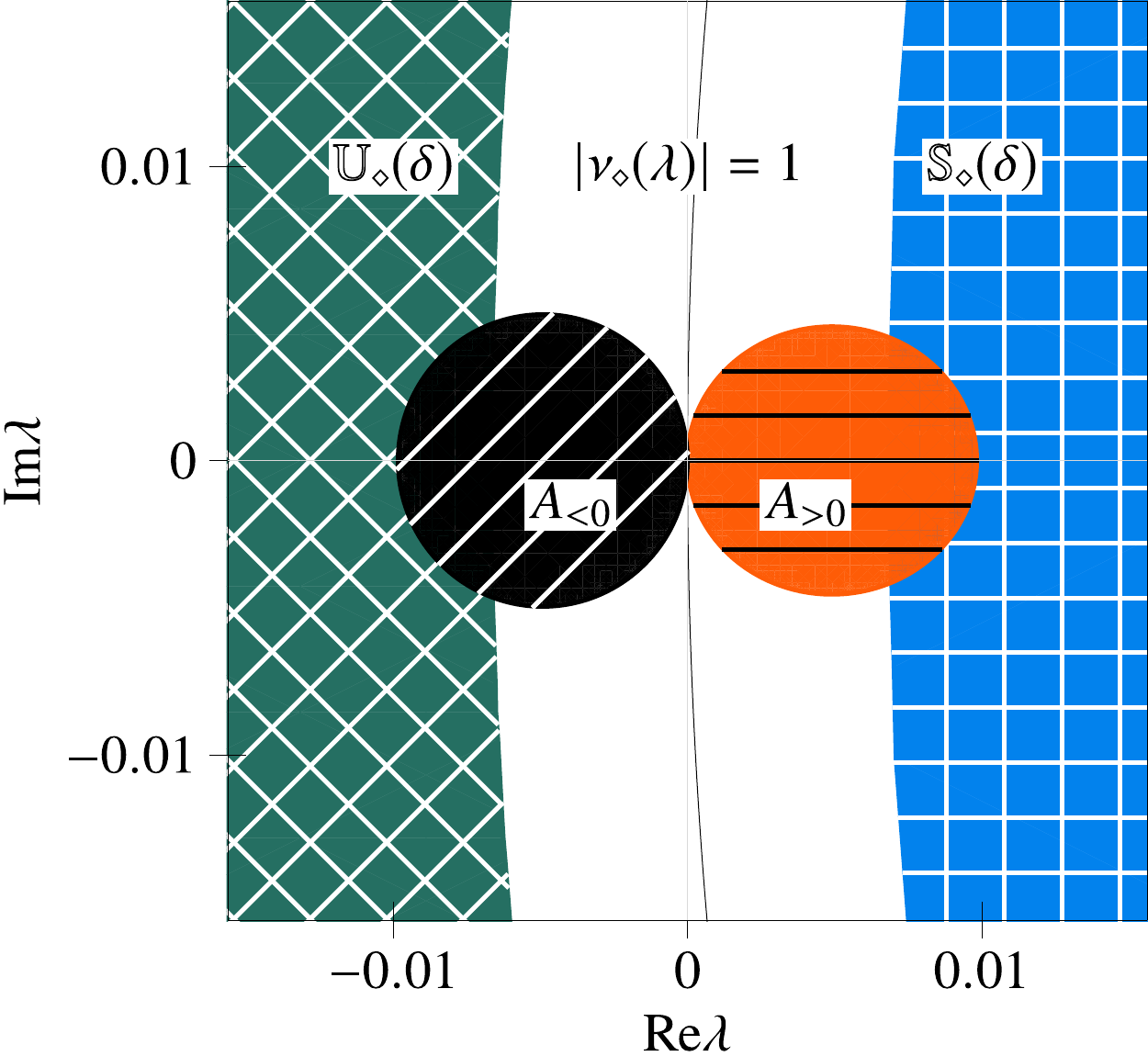}
		\caption{The sets $A_{<0}$ and $A_{>0}$.}
	\end{subfigure}
	\begin{subfigure}{0.49\textwidth}
		\centering
		\includegraphics[width=\linewidth]{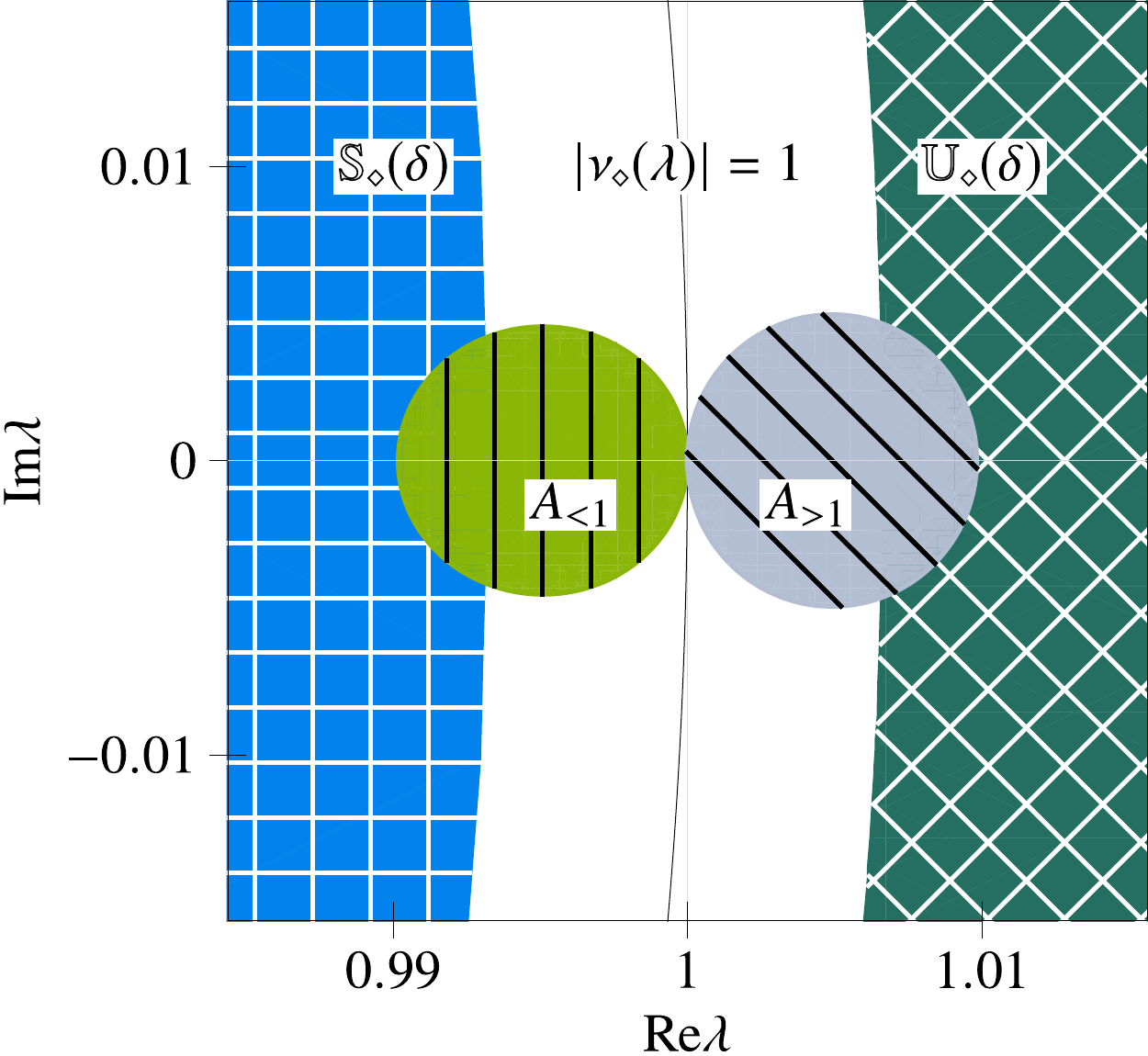}
		\caption{The sets $A_{<1}$ and $A_{>1}$.}
	\end{subfigure}
	\caption[The sets $\stable_\star(\delta)$, $\unstable_\star(\delta)$, $A_{<0}$, $A_{>0}$, $A_{<1}$ and $A_{>1}$]{Subfigure (A) depicts the sets $\stable_\star(\delta)$ (blue region with white horizontal-vertical cross-hatching) and $\unstable_\star(\delta)$ (dark green region with white diagonal cross-hatching) defined in \eqref{eq:stablestarsets}, and the sets $A_{<0}$ (black region with white diagonal lines), $A_{>0}$ (orange region with black horizontal lines), $A_{<1}$ (light green region with vertical lines) and $A_{>1}$ (gray region with black diagonal lines) defined in \eqref{eq:A<0}, \eqref{eq:A>0}, \eqref{eq:A<1} and \eqref{eq:A>1}, respectively, with $\alpha=0$, $a=21$, $b=4$ and $\delta=\frac{1}{20}$.
		The black curve between the dark green and blue regions denotes the set $\{\lambda\in\CC:\abs{\nu_\star(\lambda)}=1\}$.
		Subfigure (B) shows an enlargement of the region containing the sets $A_{<0}$ and $A_{>0}$. Subfigure (C) shows an enlargement of the region containing the sets $A_{<1}$ and $A_{>1}$. Note that the cross-hatching in subfigure (A) is not to scale with the cross-hatching in subfigures (B) and (C).
}
	\label{fig:stability}
\end{figure}

\begin{remark}
	\label{rem:OvalsofCassini}
	In view of Theorem \ref{thm:1a} we see that given $\delta\in(0,1)$ and a compact set $K$ in $\CC$, we have $\stable\cap K\subset\stable_\star(\delta)\cap K$ and $\unstable\cap K\subset\unstable_\star(\delta)\cap K$ for all $\beta$ sufficiently large.
	In particular, the sets $\stable\cap K$ and $\unstable\cap K$ respectively converge to $\stable_\star(0)\cap K$ and $\unstable_\star(0)\cap K$ in the Hausdorff {\color{red}metric}, as $\beta\to\infty$.
	In Figure \ref{fig:ovals} we plot the boundaries $\pd\stable_\star(0)$ when $\alpha=\frac{1}{8}$, $b=1$ and $a$ varies between 1 and 24.
	In Figure \ref{fig:stability}(A) we plot the sets $\stable_\star(\delta)$ and $\unstable_\star(\delta)$ when $\alpha=0$, $a=21$, $b=4$ and $\delta=\frac{1}{20}$.
\end{remark}

For $\beta>\bmono(\alpha,f,\delta,K)$, Theorem \ref{thm:1a} is only informative about the relation between $\rho(\mono_\lambda)$ and 1 if $\lambda\in\stable_\star(\delta)$ or $\lambda\in\unstable_\star(\delta)$ --- see the blue and dark green regions in Figure \ref{fig:stability}.
In particular, Theorem \ref{thm:1a} is not informative about the relation between $\rho(\mono_\lambda)$ and 1 when $\abs{\nu_\star(\lambda)}$ is infinitesimally close to 1 (i.e., when $\lambda$ is infinitesimally close to the black line in Figure \ref{fig:stability}).
Our next two results further characterize the relation between $\rho(\mono_\lambda)$ and 1 for such $\lambda$.

\subsubsection{Extended characteristic multipliers for $\lambda$ near 1}

From the definition of $\nu_\star$ given in \eqref{eq:nuast}, $\nu_\star(1)=1$ and $\nu_\star(\cdot)$ is continuous.
Thus, for fixed $\beta$ large, Theorem \ref{thm:1a} is not informative about the sets $\stable$ and $\unstable$ for $\lambda$ infinitesimally close to $1$, and this region is relevant in the analysis of weakly coupled systems (see Section \ref{sec:weakproof}).
Our next theorem, whose proof is given in Section \ref{sec:proofsnbhd}, further characterizes the sets $\stable$ and $\unstable$ for $\lambda$ infinitesimally close to $1$.

\begin{theorem}
	\label{thm:2a}
	Suppose $\alpha\ge0$ and $f$ satisfies Assumption \ref{ass:main}.
	There are open sets $A_{<1}\subset\CC_{<1}$ and $A_{>1}\subset\CC_{>1}$ with $1\in\partial A_{<1}\cap\partial A_{>1}$ such that the following hold:
	\begin{itemize}
		\item[(i)] For each $\lambda\in\CC_{<0}$ there exists $\eta_\lambda>0$ such that $1+\eta\lambda\in A_{<1}$ for all $\eta\in(0,\eta_\lambda)$.
		\item[(ii)] For each $\lambda\in\CC_{>0}$ there exists $\eta_\lambda>0$ such that $1+\eta\lambda\in A_{>1}$ for all $\eta\in(0,\eta_\lambda)$.
	\end{itemize}
	In addition, there is a $\bnbhdone=\bnbhdone(\alpha,f)\ge\buniq$ such that if $\beta>\bnbhdone$ then the following hold:
	\begin{itemize}
		\item[(iii)] $A_{<1}\subset\stable$, i.e., $\rho(\mono_\lambda)<1$ for all $\lambda\in A_{<1}$.
		\item[(iv)] $A_{>1}\subset\unstable$, i.e., $\rho(\mono_\lambda)>1$ for all $\lambda\in A_{>1}$.
	\end{itemize}
\end{theorem}

\begin{remark}
Explicit expressions for the sets $A_{<1}$ and $A_{>1}$ are given in equations {\color{red}\eqref{eq:A<1}--\eqref{eq:A>1}} of the proof.
In Figure \ref{fig:stability}(C) we depict the sets $A_{<1}$ and $A_{>1}$ in the case $\alpha=0$, $a=21$ and $b=4$.
\end{remark}

\subsubsection{Extended characteristic multipliers for $\lambda$ near 0}

Next we consider the case that $\lambda$ is infinitesimally close to 0, which is relevant in the analysis of near uniformly coupled systems (see Section \ref{sec:uniformproof}).
By the definition of $\nu_\star$ in \eqref{eq:nuast},
	$$\nu_\star(0)=\lb\frac{e^{-\alpha}-\vr_1}{1-\vr_1}\rb\lb\frac{e^{-\alpha}-\vr_2}{1-\vr_2}\rb\in(0,1].$$
If $\alpha>0$, then $\abs{\nu_\star(0)}<1$ and we can apply Theorem \ref{thm:1a} to show that $0\in\stable$ for $\beta$ sufficiently large.
However, if $\alpha=0$, then $\nu_\star(0)=1$ and since $\nu_\star(\cdot)$ is continuous, it follows that for fixed $\beta$ large, Theorem \ref{thm:1a} is not informative about the sets $\stable$ and $\unstable$ for $\lambda$ infinitesimally close to $0$.
Our next theorem, whose proof is given in Section \ref{sec:proofnbhd0}, further characterizes the sets $\stable$ and $\unstable$ for $\lambda$ infinitesimally close to $0$.

\begin{theorem}\label{thm:nbhdof0}
	Suppose $\alpha=0$ and $f$ satisfies Assumption \ref{ass:main}.
	There is a $\bnbhdzero^\dag=\bnbhdzero^\dag(f)\ge\buniq(0,f)$ and open sets $A_{>0}\subset\CC_{>0}$ and $A_{<0}\subset\CC_{<0}$ with $0\in\partial A_{>0}\cap\partial A_{<0}$ such that
	\begin{itemize}
		\item[(i)] For each $\lambda\in\CC_{>0}$ there exists $\eta_\lambda>0$ such that $\eta\lambda\in A_{>0}$ for all $\eta\in(0,\eta_\lambda)$.
		\item[(ii)] For each $\lambda\in\CC_{<0}$ there exists $\eta_\lambda>0$ such that $\eta\lambda\in A_{<0}$ for all $\eta\in(0,\eta_\lambda)$.
	\end{itemize}
	In addition, if $\beta>\bnbhdzero^\dag$ then
	\begin{itemize}
		\item[(iii)] $A_{>0}\subset\stable$, i.e., $\rho(\mono_\lambda)<1$ for all $\lambda\in A_{>0}$;
		\item[(iv)] $A_{<0}\subset\unstable$, i.e., $\rho(\mono_\lambda)>1$ for all $\lambda\in A_{<0}$.
	\end{itemize}
\end{theorem}

\begin{remark}
Explicit expressions for the sets $A_{>0}$ and $A_{<0}$ are in \eqref{eq:A>0}--\eqref{eq:A<0} of the proof.
In Figure \ref{fig:stability}(B) we depict the sets $A_{<0}$ and $A_{>0}$ in the case $\alpha=0$, $a=21$, $b=4$.
\end{remark}

\subsubsection{Extended characteristic multipliers for $\lambda\in[0,1)$}

As a corollary of Theorems \ref{thm:1a}, \ref{thm:2a} and \ref{thm:nbhdof0}, we have the following characterization of the $\stable\cap[0,1)$ and $\unstable\cap[0,1)$, which are relevant in the analysis of doubly nonnegative coupled systems (see Section \ref{sec:01proof}).
A detailed proof of the corollary is given in Section \ref{sec:01}.
Recall the definitions of the constants $r_0$ and $\Delta$ given in \eqref{eq:r0} and \eqref{eq:Delta}, respectively.

\begin{cor}\label{cor:01}
	Suppose $\alpha\ge0$ and $f$ satisfies Assumption \ref{ass:main}.
	Suppose $\Delta<0$. 
	Then there exists $\bzoa=\bzoa(\alpha,f)\ge\buniq$ such that for every $\beta>\bzoa$, the following hold:
	\begin{itemize}
		\item[(i)] If $\alpha>0$ then $\stable_\star(0)\cap[0,1)=[0,1)$ and $[0,1)\subset\stable$.
		\item[(ii)] If $\alpha=0$ then $\stable_\star(0)\cap[0,1)=(0,1)$ and $(0,1)\subset\stable$. 
	\end{itemize}
	On the other hand, suppose $\Delta>0$.
	Let $\ve$ satisfy inequality \eqref{eq:ve}.
	Then there exists $\bzob=\bzob(\alpha,f,\ve)\ge\buniq$ such that for every $\beta>\bzob$, the following hold:
	\begin{itemize}
		\item[(iii)] If $\alpha>0$ then $\stable_\star(0)\cap[0,1)=[0,r_0-\sqrt{\Delta})\cup(r_0+\sqrt{\Delta},1)$ and $[0,r_0-\sqrt{\Delta}-\ve)\cup(r_0+\sqrt{\Delta}+\ve,1)\subset\stable$.
		\item[(iv)] If $\alpha=0$ then $\stable_\star(0)\cap[0,1)=(0,r_0-\sqrt{\Delta})\cup(r_0+\sqrt{\Delta},1)$ and $(0,r_0-\sqrt{\Delta}-\ve)\cup(r_0+\sqrt{\Delta}+\ve,1)\subset\stable$.
		\item[(v)] $\unstable_\star(0)\cap[0,1]=(r_0-\sqrt{\Delta},r_0+\sqrt{\Delta})$ and $(r_0-\sqrt{\Delta}+\ve,r_0+\sqrt{\Delta}-\ve)\subset\unstable$.
	\end{itemize}
\end{cor}

\begin{remark}
In Figure \ref{fig:stability} we depict the regions $\stable_\star(\delta)$, $\unstable_\star(\delta)$, $A_{<0}$, $A_{>0}$, $A_{<1}$ and $A_{>1}$ in the case that $\alpha=0$, $\Delta<0$ and $\delta\in(0,1)$ is sufficiently small so that $(0,1)\subset A_{>0}\cup\stable_\star(\delta)\cup A_{<1}$, which corresponds to the setting in Corollary \ref{cor:01}(ii).
In this case, if $\beta>\bzoa$, then $(0,1)\subset\stable$.
\end{remark}

\section{Proof of Theorem \ref{thm:msf}}
\label{sec:msfproof}

Before proving Theorem \ref{thm:msf} we first state and prove the following useful lemma.

\begin{lem}\label{lem:decomp}
	Suppose $\gmatrix$ is an $n\times n$ real-valued matrix and $L_1,L_2\in C(\R_+,B(\C(\CC)))$.
	Define ${\bf V}\in C(\R_+,B(\C(\CC^n)))$ by ${\bf V}(t){\bs\vf}={\bf u}_t$ for $t\ge0$ and ${\bs\vf}\in\C(\CC^n)$, where ${\bf u}=(u^1,\dots,u^n)^T\in C([-1,\infty),\CC^n)$ is continuously differentiable on $(0,\infty)$, and satisfies ${\bf u}_0={\bs\vf}$ and
		\be\label{eq:ulinear}\dot{u}^j(t)=L_1(t)u_t^j+\sum_{k=1}^n\gmatrix^{jk}L_2(t)u_t^k,\qquad j=1,\dots,n.\ee
	For $\lambda\in\CC$, define $V_\lambda\in C(\R_+,B(\C(\CC)))$ by $V_\lambda(t)\vf=u_t$ for $t\ge0$ and $\vf\in\C(\CC)$, where $u\in C([-1,\infty),\CC)$ is continuously differentiable on $(0,\infty)$, and satisfies $u_0=\vf$ and
		\be\label{eq:uoned}\dot u(t)=L_1(t)u_t+\lambda L_2(t)u_t.\ee
	Then for each $t_0\ge1$, ${\bf V}(t_0)\in B_0(\C(\CC^n))$ and $V_\lambda(t_0)\in B_0(\C(\CC))$ for all $\lambda\in\CC$, and
		$$\sigma({\bf V}(t_0))=\bigcup_{\lambda\in\sigma(H)}\sigma(V_\lambda(t_0)).$$
\end{lem}

\begin{proof}
	First note that the continuity of ${\bf V}(\cdot)$ and the continuity of $V_\lambda(\cdot)$, $\lambda\in\CC$, follow from the continuity of the solutions ${\bf u}$ and $u$. 
	The compactness of ${\bf V}(t)$ and $V_\lambda(t)$, $\lambda\in\CC$, for $t\ge1$ follow, for example, from \cite[Chapter III, Corollary 4.7]{Diekmann1991}.

	Fix $t_0\ge1$.
	Suppose $\lambda\in\sigma(\gmatrix)$, $\mu\in\sigma(V_\lambda(t_0))$ and $\psi_\mu\in\C(\CC)$ is such that $\psi_\mu\not\equiv0$ and $V_\lambda(t_0)\psi_\mu=\mu\psi_\mu$.
	Let $u\in C([-1,\infty),\CC)$ denote the solution of the scalar linear DDE \eqref{eq:uoned} with $u_0=\psi_\mu$ so that $u_{t_0}=\mu\psi_\mu$.
    Choose an eigenvector ${\bf v}_\lambda\in \CC^n$ associated with $\lambda$ so that $\gmatrix{\bf v}_\lambda=\lambda{\bf v}_\lambda$ and define ${\bf u}=u{\bf v}_\lambda\in C([-1,\infty),\CC^n)$. 
	Then for all $t>0$, we have
	\begin{align*}
		\dot{\bf u}(t)&={\bf v}_\lambda L_1(t)u_t+\gmatrix{\bf v}_\lambda L_2(t)u_t
		={\bf L}_1(t){\bf u}_t+\gmatrix{\bf L}_2(t){\bf u}_t,
	\end{align*}
	where ${\bf L}_j(t){\bf u}_t=(L_j(t)u_t^1,\dots,L_j(t)u_t^n)^T$, for $j=1,2$.
	Thus, ${\bf u}$ is a solution of the system of linear DDEs \eqref{eq:ulinear} with ${\bf u}_0={\bf v}_\lambda\psi_\mu$ and ${\bf u}_{t_0}=\mu{\bf v}_\lambda\psi_\mu$.
	It follows that $\mu$ is an eigenvalue of ${\bf V}(t_0)$ with associated eigenvector ${\bs\vf}_\mu={\bf v}_\lambda\vf_\mu$.
	This proves that $\sigma(V_\lambda(t_0))\subset\sigma({\bf V}(t_0))$ for all $\lambda\in\sigma(\gmatrix)$.
	
	Now suppose $\mu\in\sigma({\bf V}(t_0))$ and $\bs{\vf}_\mu\in\C(\CC^n)$ is such that $\bs{\vf}_\mu \not\equiv 0$ and ${\bf V}(t_0){\bs\vf}_\mu=\mu{\bs\vf}_\mu$.
	Let ${\bf u}$ denote the solution of the system of linear DDEs \eqref{eq:ulinear} with ${\bf u}_0={\bs\psi}_\mu$ so that ${\bf u}_{t_0}=\mu{\bs\psi}_\mu$.
	Writing $\gmatrix$ in its Jordan form we have there is a real invertible $n\times n$ matrix $P$ and a complex upper triangular $n\times n$ matrix $D$ such that $\gmatrix=PDP^{-1}$ and the diagonal elements of $D$, which we denote by $\lambda_1,\dots,\lambda_n$, are the eigenvalues of $\gmatrix$.
	Define ${\bf w}=P^{-1}{\bf u}$.
	Using the definition of ${\bf L}_j(t)$, $j=1,2$, given above, we have
		$$\dot{\bf w}(t)=P^{-1}{\bf L}_1(t){\bf u}_t+DP^{-1}{\bf L}_2(t){\bf u}_t={\bf L}_1(t){\bf w}_t+D{\bf L}_2(t){\bf w}_t.$$
	Since ${\bf w}$ is nonzero, there is a unique $k\in\{1,\dots,n\}$ such that $w^k$ is nonzero and $w^j$ is zero for $k<j\le n$.
	Then, due to the fact that $D$ is upper triangular, we see that for all $t>0$,
		$$\dot w^k(t)=L_1(t)w_t^k+\lambda_kL_2(t)w_t^k.$$
	Thus, $w^k$ is a solution of scalar linear DDE \eqref{eq:uoned} with $w_0^k=[P^{-1}{\bf u}_0]^k$ and $w_{t_0}^k=\mu[P^{-1}{\bf u}_0]^k$.
	Therefore, $\mu$ is an eigenvalue of $V_{\lambda_k}(t_0)$ with associated eigenvector $\vf_\mu=[P^{-1}{\bs\vf}_\mu]^k$.
	Hence, $\sigma({\bf V}(t_0))\subset\cup_{\lambda\in\sigma(\gmatrix)}\sigma(V_\lambda(t_0))$, which completes the proof.
\end{proof}

\begin{proof}[Proof of Theorem \ref{thm:msf}]
	Let ${\bf y}={\bf y}({\bs\vf})$ denote the solution to the variational equation \eqref{eq:nvariational} with initial condition ${\bf y}_0={\bs\vf}$.
	Then ${\bf y}=(y^1,\dots,y^n)$ satisfies the system of linear DDEs \eqref{eq:ulinear} with $\gmatrix=G$ and $L_1,L_2\in C(\R_+,B_0(\C(\CC)))$ given by $L_1(t)\vf=-\alpha\vf(0)$ and $L_2(t)\vf=\beta f'(\sops(t-1))\vf(-1)$, for $t\ge0$ and $\vf\in\C(\CC)$.
	It follows from Lemma \ref{lem:decomp} and the respective definitions of $\bf\mono$ and $\mono_\lambda$ in \eqref{eq:nmono} and \eqref{eq:gmonodromy}, that ${\bf V}(\period)={\bf \mono}$, $V_\lambda(\period)=\mono_\lambda$ for all $\lambda\in\CC$, and the equality in \eqref{eq:Floquet} holds.
	
	Next, by Theorem \ref{thm:xie}, $\sops$ is linearly stable, and so the trivial characteristic multiplier $\mu=1$ is a simple eigenvalue of $\mono_1$ and $\abs{\mu}<1$ for all other eigenvalues of $\mono_1$.
	Suppose $\rho(\mono_\lambda)<1$ for all $\lambda\in\sigminus(G)$.
	Then by the equality in \eqref{eq:Floquet}, we have
	$$\max\lcb\abs{\mu}:\mu\in\sigminus({\bf\mono})\rcb=\max\lcb\abs{\mu}:\mu\in\sigminus(\mono_1)\cup\bigcup_{\lambda\in\sigminus(G)}\sigma(\mono_\lambda)\rcb<1.$$
	This proves that ${\bf\sops}$ is linearly stable.
	Now suppose $\rho(\mono_\lambda)>1$ for some $\lambda\in\sigma(G)$.
	Then by the equality in \eqref{eq:Floquet}, we have $\max\lcb\abs{\mu}:\mu\in\sigma({\bf\mono})\rcb\ge\rho(\mono_\lambda)>1$.
	This proves that ${\bf\sops}$ is linearly unstable.
\end{proof}

\section{Convergence of normalized SOPS}
\label{sec:normalizedSOPS}

In this section we recall results from \cite[Section 3]{Xie1991} on the convergence of normalized SOPS to the solution of a scalar DDE with step function nonlinearity.

\textbf{We make the following assumptions, which will hold throughout the remainder of this work unless otherwise noted}.
Fix $\alpha\ge0$ and $f$ satisfying Assumption \ref{ass:main}. 
Without loss of generality, we assume $a\ge b$, where $a,b>0$ are the constants in condition 2 of Assumption \ref{ass:main}.
Let $\buniq>0$ be as in Theorem \ref{thm:xie} so that for each $\beta>\buniq$, there exists a unique SOPS $p^\beta$ and it is linearly stable.
Let $\period^\beta>2$ denote the (minimal) period of $\xbeta$.
We extend the definition of $\xbeta$ to all of $\R$ so that $\xbeta$ is periodic on $\R$ with period $\period^\beta$.
Furthermore, by possibly performing a time translation, we assume that $p^\beta$ satisfies 
	\be\label{eq:pphasefix}\xbeta(-1)=0\qquad\text{and}\qquad\dot p^\beta(-1)>0.\ee
We let $z_1^\beta>0$ and $z_2^\beta>z_1^\beta+1$ be the unique times such that $\period^\beta=z_2^\beta+1$, $\xbeta(t)>0$ for all $-1<t<z_1^\beta$ and $\xbeta(t)<0$ for all $z_1^\beta<t<z_2^\beta$.

For each $\beta>\buniq$ define the normalized SOPS $\nxbeta\in C(\R,\R)$ by 
	\be\label{eq:xbarbeta}\nxbeta(t)=\beta^{-1}\xbeta(t),\qquad t\in\R.\ee 
Define the step function $f^\star:\R\to\R$ by
\begin{equation} \label{eq_def_f_bar}
f^\star(\xi) = \begin{cases}
-a & \text{if } \xi \ge 0, \\
b & \text{if } \xi < 0.
\end{cases}
\end{equation}
For $\alpha>0$ let
\begin{align} \label{eq:q1bar_q2bar}
q_1^\star = \frac{1}{\alpha}\log\lcb 1+a^{-1}b(1-e^{-\alpha})\rcb\quad\text{and}\quad q_2^\star = \frac{1}{\alpha}\log\lcb 1+ab^{-1}(1-e^{-\alpha})\rcb,
\end{align}
and define the function $\bar\sops^\star\in C(\R,\R)$ by
\begin{equation} \label{eq:xbar}
\bar{\sops}^\star(t) = 
\begin{dcases}
\frac{b}{\alpha}\lb 1-e^{-\alpha(t+1)}\rb & \text{for } -q_2^\star-1 \le t \le 0, \\
-\frac{a}{\alpha}\lb 1-e^{-\alpha(t-q_1^\star)}\rb & \text{for } 0 < t \le q_1^\star + 1.
\end{dcases}
\end{equation}
For $\alpha=0$, $q_1^\star,q_2^\star$ and $\bar\sops$ are taken to be the pointwise limits of the definitions in \eqref{eq:q1bar_q2bar} and \eqref{eq:xbar} as $\alpha\downarrow0$; specifically, 
\begin{align} \label{eq:q1bar_q2baralpha0}
q_1^\star=a^{-1}b\quad\text{and}\quad q_2^\star=ab^{-1},
\end{align}
and 
\begin{equation}\label{eq:xbaralpha0}
\bar\sops^\star(t) = 
\begin{dcases}
-a+b(t+q_2^\star+1)&\text{for } -q_2^\star-1 \le t \le 0, \\
b-at&\text{for } 0 < t \le q_1^\star + 1.
\end{dcases}
\end{equation}
Note that our assumption $a\ge b$ implies that $q_1^\star\le q_2^\star$.
Since $\bar\sops^\star(-q_2^\star-1)=\bar\sops^\star(q_1^\star+1)$, we can (uniquely) extend the definition of $\bar\sops^\star$ to all of $\R$ so that $\bar\sops^\star$ is continuous and periodic with period $\period^\star=q_1^\star+q_2^\star+2$. See Figure \ref{fig:xbar} for a graph of $\bar\sops^\star$.

\begin{figure}[h!]
	\centering
	\includegraphics[width=.9\textwidth]{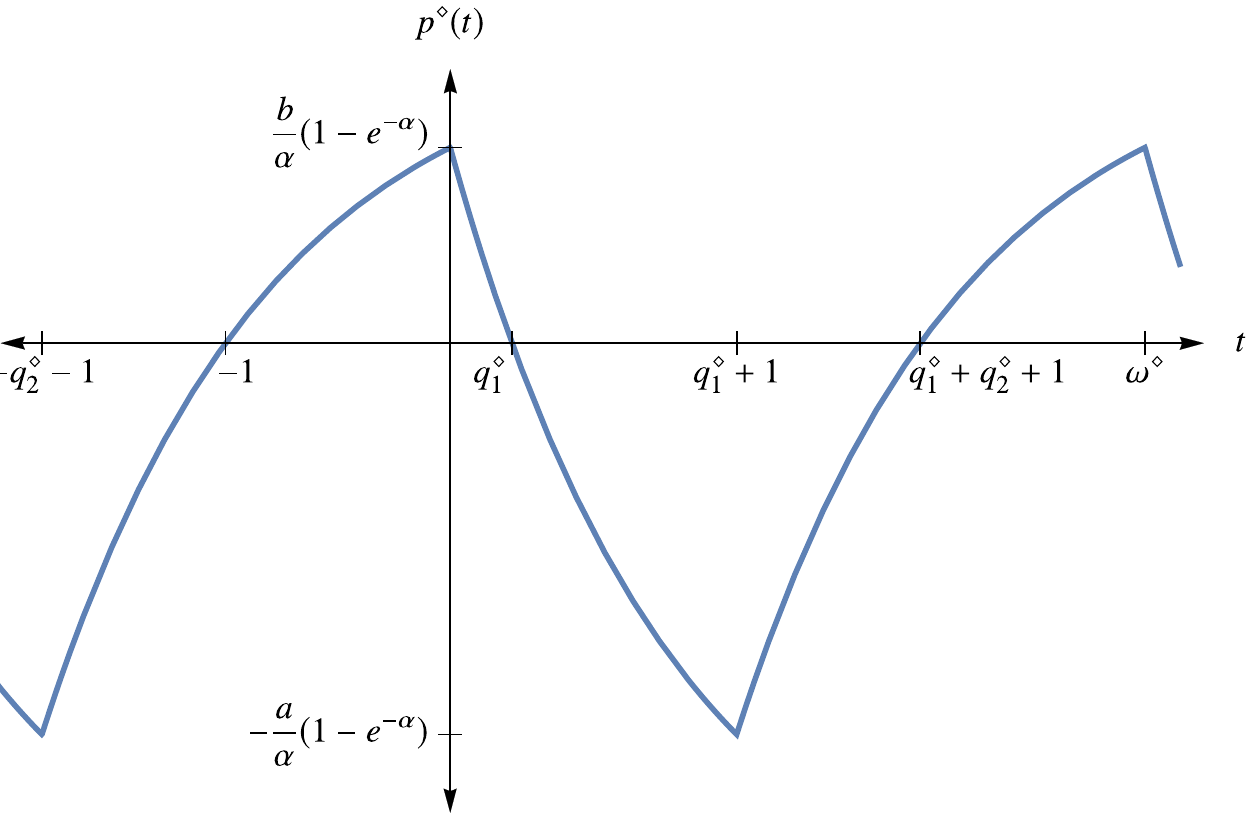}
	\caption[Graph of $\bar\sops^\star$]{Graph of $\bar\sops^\star$ defined in \eqref{eq:xbar} with $\alpha=1$, $a=2$ and $b=1$.}
	\label{fig:xbar}
\end{figure}

Define
\begin{align}
\label{eq:Iset}\Iset&=\{t\in\R:\bar\sops^\star(t)=0\}=\{-1+k\period^\star,q_1^\star+k\period^\star:k\in\Z\},\\ 
\label{eq:Jset}\Jset&=\{t\in\R:\bar\sops^\star(t-1)=0\}=\{k\period^\star,q_1^\star+1+k\period^\star:k\in\Z\}.
\end{align}
Note that $\Jset=\{t+1:t\in\Iset\}$.
Then $\bar\sops^\star$ is continuously differentiable on $\R\setminus\Jset$ and $\dot{\bar\sops}^\star$ satisfies (for $\alpha\ge0$)
\begin{equation} \label{eq:xbardot}
\dot{\bar\sops}^\star(t) = 
\begin{dcases}
be^{-\alpha(t+1)} & \text{if } -q_2^\star-1 < t < 0, \\
-ae^{-\alpha(t-q_1^\star)} & \text{if } 0 < t < q_1^\star + 1.
\end{dcases} 
\end{equation}
We extend the definition of $\dot{\bar\sops}^\star$ to all of $\R$ by setting $\dot{\bar\sops}^\star$ to be right continuous at all $t\in\Jset$, so that $\dot{\bar{\sops}}^\star\in D(\R,\R)$.

We close this section by stating convergence properties of $z_1^\beta$, $z_2^\beta$, $\period^\beta$, $\bar\sops^\beta$  and $\dot{\bar\sops}^\beta(t)$ as $\beta\to\infty$. For $\ve>0$, let
\begin{align}
\label{eq:Isetve}\Iset_\ve&=\{t\in\R:d_\R(\Iset,t)<\ve\},\\
\label{eq:Jsetve}\Jset_\ve&=\{t\in\R:d_\R(\Jset,t)<\ve\}.
\end{align}

\begin{prop}
	[{\cite[Theorem 14 \& Corollary 15]{Xie1991}}]
	\label{prop:xbar_convergence}
	Suppose $\alpha\geq0$ and $f$ satisfies Assumption \ref{ass:main}.
     	Then the following limits hold:
	\begin{align}
	\lim_{\beta\to\infty}|{z_1^\beta-q_1^\star}|&=0,\label{eq:z1_bound}\\
	\lim_{\beta\to\infty}|{z_2^\beta-q_1^\star-q_2^\star-1}|&=0,\label{eq:z2_bound}\\
	\lim_{\beta\to\infty}|{\period^\beta-\period^\star}|&=0,\label{eq:pbound}
	\end{align}
	and for all compact $K\subset\R$ and all $\ve>0$,
	\begin{align}
	\label{eq:pbetaconvergence}\lim_{\beta\to\infty}\sup_{t\in K}|{\bar\sops^\beta(t)-\bar\sops^\star(t)}|&=0,\\
	\lim_{\beta\to\infty}\sup_{t\in K\setminus\Jset_\ve}|{\dot{\bar\sops}^\beta(t)-\dot{\bar\sops}^\star(t)}|&=0.\label{eq:xdot_bound}
	\end{align}
\end{prop}

\section{Limits of the extended characteristic multipliers}
\label{sec:limit_char}

This section is devoted to the proof of Theorem \ref{thm:1a}, which states that the spectrum of the extended monodromy operator, $\sigma(\mono_\lambda)$, converges to $\{0,\nu_\star(\lambda)\}$ as $\beta\to\infty$.
Recall that we have fixed $\alpha\ge0$ and $f$ satisfying Assumption \ref{ass:main}.

\subsection{Convergence of the measure appearing in the extended variational equation}
\label{sec:convmeasures}

In the next section, Section \ref{sec:limitmonodromy}, we analyze limits for solutions to the extended variational equation \eqref{eq:eve}, as $\beta\to\infty$.
The term $\beta f'(p^\beta(t-1))$ appearing on the right hand side of the extended variational equation \eqref{eq:eve} converges, as $\beta\to\infty$, to zero at all $t\in\R\setminus\Jset$ and to negative infinity at all $t\in\Jset$, where we recall the definition of $\Jset$ in \eqref{eq:Jset}.
Therefore, in order to identify the correct limiting object, it is advantageous to instead consider $\beta f'(p^\beta(t-1))dt$ as a (signed) measure on $\R$.  
In this section we prove convergence properties for the measure $\beta f'(p^\beta(t-1))dt$ as $\beta\to\infty$.

To this end, we first make some useful observations.
Recall the assumptions stated at the beginning of Section \ref{sec:normalizedSOPS} and the definitions of $f^\star$, $q_1^\star$, $q_2^\star$, $\sops^\star$ and $\Iset$ in \eqref{eq_def_f_bar}--\eqref{eq:Iset}.
For each $\beta>\buniq$ define the function $h^\beta:\R\to\R$ by 
	\be\label{eq:hbetafsopsbeta}h^\beta(t)=f(\sops^\beta(t)),\qquad t\in\R.\ee
Define $h^\star\in D(\R,\R)$ to be periodic with period $\period^\star$ and satisfy
\be\label{eq:hbar}h^\star(t)=
	\begin{cases}
		-a,&t\in[-1,q_1^\star),\\
		b,&t\in[q_1^\star,q_1^\star+q_2^\star+1).
	\end{cases}
\ee
Then $h^\star$ satisfies $h^\star(t)=f^\star(\bar\sops^\star(t))$ for all $t\in\R\setminus\Iset$ and, in view of Proposition \ref{prop:xbar_convergence} and the fact that $f$ is bounded by part 2 of Assumption \ref{ass:main}, 
	\be\label{eq:hbetaconvergence}\text{$h^\beta$ converges to $h^\star$ uniformly on compact intervals contained in $\R\setminus\Iset$ as $\beta\to\infty$.}\ee
By the definition of $h^\beta$ in \eqref{eq:hbetafsopsbeta}, the chain rule and the definition of $\bar\sops^\beta$ in \eqref{eq:xbarbeta}, we have
	\be\label{eq:hbeta}\beta f'(p^\beta(t))dt=(\dot{\bar\sops}^\beta(t))^{-1}dh^\beta(t),\qquad t\not\in\{s\in\R:\dot p^\beta(s)=0\}.\ee
For $\ve>0$ and $N<\infty$, recall the definitions of $\Iset_\ve$ and $L_{N,\Iset_\ve}^\infty(\R)$ in \eqref{eq:Isetve} and \eqref{eq:LCAbounded}, respectively.
In the next proposition we prove convergence properties for the measure $\beta f'(\sops^\beta(t))dt$, as $\beta\to\infty$.
In particular, the result implies that $\beta f'(\sops^\beta(t))dt$ converges to $(\dot{\bar\sops}^\star(t))^{-1}dh^\star(t)$ in the weak-$^\ast$ topology, as $\beta\to\infty$ (see Remark \ref{rem:dh} following the proposition).

\begin{prop}\label{prop:mubetaconvergence}
	For all $-\infty<t_0<t_1<\infty$,
		\be\label{eq:supbetafprime}\sup_{\beta>\buniq}\beta\int_{t_0}^{t_1}\abs{f'(\sops^\beta(t))}dt<\infty,\ee
	and for all $N<\infty$ and $\ve\in(0,q_1^\star/2)$,
		\be\label{eq:gbetafprimegsopsbardh}\lim_{\beta\to\infty}\sup\lcb\abs{\int_{t_0}^{t_1}g(t)\beta f'(p^\beta(t))dt-\int_{t_0}^{t_1}g(t)(\dot{\bar\sops}^\star(t))^{-1}dh^\star(t)}:g\in L_{N,\Iset_\ve}^\infty(\R)\rcb=0.\ee
\end{prop}

\begin{remark}
\label{rem:dh}
	For $\beta>\buniq$ let $\mu^\beta\in\mathcal{M}(\R)$ denote the measure defined by
		$$\mu^\beta((t_0,t_1])=\int_{t_0}^{t_1}\beta f'(p^\beta(t))dt,\qquad -\infty<t_0<t_1<\infty,$$
	and let $\mu^\star\in\mathcal{M}(\R)$ denote the measure defined by
		$$\mu^\star((t_0,t_1])=\int_{t_0}^{t_1}(\dot{\bar\sops}^\star(t))^{-1}dh^\star(t),\qquad-\infty<t_0<t_1<\infty.$$
	Given any Lipschitz continuous $g\in C_c(\R,\R)$, let $N<\infty$ be its Lipschitz constant and choose $t_0<t_1$ such that the support of $g$ is contained in $[t_0,t_1]$.
	Then Proposition \ref{prop:mubetaconvergence} implies that
		\be\label{eq:mubetaconvergence1}\lim_{\beta\to\infty}\int_\R g(t)d\mu^\beta(t)=\int_\R g(t)d\mu^\star(t).\ee
	The limit \eqref{eq:mubetaconvergence1} can be readily extended to all $g\in C_c(\R,\R)$; however this result is not needed here, so we omit the proof. 
	In particular, $\mu^\beta$ converges to $\mu^\star$ in the weak-$^\ast$ topology on $\mathcal{M}(\R)$ as $\beta\to\infty$. 

    In addition, by the equation for $\dot{\bar\sops}$ in \eqref{eq:xbardot} and the definition of $h^\star$ in \eqref{eq:hbar}, given a Borel set $A$ of $\R$,
		\be\label{eq:barmu}\mu^\star(A)=-\lb 1+ab^{-1}\rb\sum_{k\in\Z}\delta_{-1+k\period^\star}(A)-\lb 1+a^{-1}b\rb\sum_{k\in\Z}\delta_{q_1^\star+k\period^\star}(A),\ee
	where, for $t\in\R$, $\delta_t(\cdot)$ is the Dirac delta measure at $t$ defined by $\delta_t(A)=1$ if $t\in A$ and $\delta_t(A)=0$ otherwise.
\end{remark}

The remainder of this section is devoted to the proof of Proposition \ref{prop:mubetaconvergence}.
We first analyze the limit \eqref{eq:gbetafprimegsopsbardh} for $[t_0,t_1]\subset\R\setminus\Iset$ and $[t_0,t_1]\subset\Iset_\ve$ separately, where $\Iset$ and $\Iset_\ve$ are respectively defined in \eqref{eq:Iset} and \eqref{eq:Isetve}. 
In the next lemma we consider compact intervals that are contained in $\R\setminus\Iset$.

\begin{lem}\label{lem:limbetaintA0}
	Let $-\infty<t_0<t_1<\infty$ be such that $[t_0,t_1]\cap\Iset=\emptyset$. 
	Then
		$$\lim_{\beta\to\infty}\int_{t_0}^{t_1} |\beta f'(p^\beta(t))|dt=0.$$
\end{lem}

\begin{proof}
	Since $[t_0,t_1]\cap\Iset=\emptyset$, it follows from definition for $\sops^\star$ in \eqref{eq:q1bar_q2bar}--\eqref{eq:xbaralpha0} that $c=\inf_{t\in[t_0,t_1]}|\bar\sops^\star(t)|>0$. 
	Along with the convergence shown in Proposition \ref{prop:xbar_convergence}, this implies that for all $\beta$ sufficiently large,
		\be\label{eq:limbetaintoA01}\inf_{t\in [t_0,t_1]}|\bar\sops^\beta(t)|\ge\frac{c}{2}.\ee
	In view of the definition of $\nxbeta$ in \eqref{eq:xbarbeta}, the lower bound on $\nxbeta$ in \eqref{eq:limbetaintoA01}, and the super-linear decay of $f'$ in condition 3 of Assumption \ref{ass:main}, we have
		$$\lim_{\beta\to\infty}\sup_{t\in [t_0,t_1]}|\beta f'(p^\beta(t))|=\lim_{\beta\to\infty}\sup_{t\in [t_0,t_1]}|\beta f'(\beta\bar\sops^\beta(t))|=0.$$
	This completes the proof of the lemma.
\end{proof}

In the next two lemmas we consider intervals that are contained in $\Iset_\ve$.
We will use the fact that by the equation for $\dot{\bar\sops}^\star$ in \eqref{eq:xbardot} and the periodicity of $\bar\sops^\star$, given $k\in\Z$,
\begin{align}
	\label{eq:betakbar}
	\inf_{t\in[-1+k\period^\star-q_1^\star/2,-1+k\period^\star+q_1^\star/2]}|\dot{\bar\sops}^\star(t)|&=be^{-\alpha q_1^\star/2},\\
	\label{eq:betakbar2}
	\inf_{t\in[\bar q_1^\star/2+k\period^\star,3\bar q_1^\star/2+k\period^\star]}|\dot{\bar\sops}^\star(t)|&=ae^{-\alpha q_1^\star/2},
\end{align}
and so, by the convergence of $\dot{\bar\sops}^\beta$ stated in \eqref{eq:xdot_bound}, there exists $\beta^{(k)}\ge\buniq$ such that for all $\beta>\beta^{(k)}$,
\begin{align}
	\label{eq:betak}
	\inf_{t\in[-1+k\period^\star-q_1^\star/2,-1+k\period^\star+q_1^\star/2]}|\dot{\bar\sops}^\beta(t)|&\ge\frac{1}{2}be^{-\alpha q_1^\star/2},\\
	\label{eq:betak2}
	\inf_{t\in[q_1^\star/2+k\period^\star,3q_1^\star/2+k\period^\star]}|\dot{\bar\sops}^\beta(t)|&\ge\frac{1}{2}ae^{-\alpha q_1^\star/2}.
\end{align}

\begin{lem}\label{lem:mubetaconvergence}
	Suppose $\ve\in(0,q_1^\star/2)$. Then for any $k\in\Z$,
	\begin{align}\label{eq:f_bounds_i}
		\lim_{\beta\to\infty}\abs{\int_{-1+k\period^\star-\ve}^{-1+k\period^\star+\ve}\beta f'(p^\beta(t))dt-\int_{-1+k\period^\star-\ve}^{-1+k\period^\star+\ve}(\dot{\bar\sops}^\star(t))^{-1}dh^\star(t)} &=0,\\ \label{eq:f_bounds_ii}
		\lim_{\beta\to\infty}\abs{\int_{q_1^\star+k\period^\star-\ve}^{q_1^\star+k\period^\star+\ve}\beta f'(p^\beta(t))dt-\int_{q_1^\star+k\period^\star-\ve}^{q_1^\star+k\period^\star+\ve}(\dot{\bar\sops}^\star(t))^{-1}dh^\star(t)} &=0.
	\end{align}     
\end{lem}

\begin{proof}
	Fix $k\in\Z$.
	We first consider the limit in \eqref{eq:f_bounds_i}. 
	By the relation in equation \eqref{eq:hbeta} and the definition of $h^\star$ in \eqref{eq:hbar}, for $\beta>\beta^{(k)}$,
	\begin{align}
		&\int_{-1+k\period^\star-\ve}^{-1+k\period^\star+\ve}\beta f'(p^\beta(t))dt-\int_{-1+k\period^\star-\ve}^{-1+k\period^\star+\ve}(\dot{\bar\sops}^\beta(t))^{-1}dh^\star(t) \label{eq:f_bounds_1}\\
		&\qquad=\int_{-1+k\period^\star-\ve}^{-1+k\period^\star+\ve}(\dot{\bar\sops}^\beta(t))^{-1}dh^\beta(t)-\int_{-1+k\period^\star-\ve}^{-1+k\period^\star+\ve}(\dot{\bar\sops}^\beta(t))^{-1}dh^\star(t).\notag
	\end{align}
        Then, using the equality in \eqref{eq:f_bounds_1}, the lower bound for $\dot{\bar\sops}^\beta$ in \eqref{eq:betak} and the respective definitions for $h^\beta$ and $h^\star$ in \eqref{eq:hbetafsopsbeta} and \eqref{eq:hbar}, we obtain the following inequality
    \begin{align}
		&\abs{\int_{-1+k\period^\star-\ve}^{-1+k\period^\star+\ve}\beta f'(p^\beta(t))dt-\int_{-1+k\period^\star-\ve}^{-1+k\period^\star+\ve}(\dot{\bar\sops}^\beta(t))^{-1}dh^\star(t)}\label{eq:f_bounds_2}\\
		&\qquad\le \frac{2e^{\alpha q_1^\star/2}}{b}\abs{\int_{-1+k\period^\star-\ve}^{-1+k\period^\star+\ve}dh^\beta(t)-\int_{-1+k\period^\star-\ve}^{-1+k\period^\star+\ve}dh^\star(t)}. \notag
	\end{align}
	Therefore, by the convergence of $h^\beta$ stated in \eqref{eq:hbetaconvergence} and the definition of $\Iset$ in \eqref{eq:Iset}, taking limits as $\beta\to\infty$ on both sides of inequality \eqref{eq:f_bounds_2}, yields
	\begin{equation}\label{eq:int-1kp1}
		\lim_{\beta\to\infty}\abs{\int_{-1+k\period^\star-\ve}^{-1+k\period^\star+\ve}\beta f'(p^\beta(t))dt-\int_{-1+k\period^\star-\ve}^{-1+k\period^\star+\ve}(\dot{\bar\sops}^\beta(t))^{-1}dh^\star(t)}=0.
	\end{equation}
	Next, by the definition of $h^\star$ in \eqref{eq:hbar},
	\begin{align}
		\abs{\int_{-1+k\period^\star-\ve}^{-1+k\period^\star+\ve}\lsb(\dot{\bar\sops}^\beta(t))^{-1}-(\dot{\bar\sops}^\star(t))^{-1}\rsb dh^\star(t)}&=\lb1+ab^{-1}\rb\abs{ \frac{1}{\dot{\bar\sops}^\beta(-1+k\period^\star)}-\frac{1}{\dot{\bar\sops}^\star(-1+k\period^\star)}} \label{eq:f_bounds_3}.
	\end{align}
	Thus, by the respective lower bounds for $\dot{\bar\sops}^\star$ and $\dot{\bar\sops}^\beta$ in equations \eqref{eq:betakbar} and \eqref{eq:betak}, and the convergence of $\dot{\bar\sops}^\beta$ shown in the limit \eqref{eq:xdot_bound}, taking limits as $\beta\to\infty$ on both sides of equation \eqref{eq:f_bounds_3}, we obtain
	\begin{equation}\label{eq:int-1kp2}
		\lim_{\beta\to\infty}\abs{\int_{-1+k\period^\star-\ve}^{-1+k\period^\star+\ve}\lsb(\dot{\bar\sops}^\beta(t))^{-1}-(\dot{\bar\sops}^\star(t))^{-1}\rsb dh^\star(t)}=0.
	\end{equation}
	Together the limits \eqref{eq:int-1kp1} and \eqref{eq:int-1kp2} imply the limit in \eqref{eq:f_bounds_i}.
	The proof of the limit in \eqref{eq:f_bounds_ii} follows an analogous argument, but uses equations \eqref{eq:betakbar2} and \eqref{eq:betak2} in place of \eqref{eq:betakbar} and \eqref{eq:betak}, respectively.
	To avoid repetition, we omit the details.
\end{proof}

\begin{lem}\label{lem:intgboundedcontinuous}
	Let $N<\infty$ and $\ve\in(0,q_1^\star/2)$.
	Then, for all $k\in\Z$,
	\begin{align}
		\label{eq:gmubeta-1}&\lim_{\beta\to\infty}\sup_{g\in L_{N,\Iset_\ve}^\infty(\R)}\abs{\int_{-1+k\period^\star-\ve}^{-1+k\period^\star+\ve}g(t)\beta f'(p^\beta(t))dt-\int_{-1+k\period^\star-\ve}^{-1+k\period^\star+\ve}g(t)(\dot{\bar\sops}^\star(t))^{-1}dh^\star(t)}=0,\\
		\label{eq:gmubetaq1}&\lim_{\beta\to\infty}\sup_{g\in L_{N,\Iset_\ve}^\infty(\R)}\abs{\int_{q_1^\star+k\period^\star-\ve}^{q_1^\star+k\period^\star+\ve}g(t)\beta f'(p^\beta(t))dt-\int_{q_1^\star+k\period^\star-\ve}^{q_1^\star+k\period^\star+\ve}g(t)(\dot{\bar\sops}^\star(t))^{-1}dh^\star(t)}=0.
	\end{align}
\end{lem}

\begin{proof}
	Fix $k\in\Z$.
	We first consider the limit in \eqref{eq:gmubeta-1}. 
	Given $g\in L_{N,\Iset_\ve}^\infty(\R)$, we have
	\begin{align*}
		\int_{-1+k\period^\star-\ve}^{-1+k\period^\star+\ve}g(t)\beta f'(p^\beta(t))dt&=g(-1+k\period^\star)\int_{-1+k\period^\star-\ve}^{-1+k\period^\star+\ve}\beta f'(p^\beta(t))dt\\
		&\qquad+\int_{-1+k\period^\star-\ve}^{-1+k\period^\star+\ve}(g(t)-g(-1+k\period^\star))\beta f'(p^\beta(t))dt,
	\end{align*}
	so in view of the convergence shown in Lemma \ref{lem:mubetaconvergence}, it suffices to show that
		\be\label{eq:gtgmubeta0}\lim_{\beta\to\infty}\sup_{g\in L_{N,\Iset_\ve}^\infty(\R)}\abs{\int_{-1+k\period^\star-\ve}^{-1+k\period^\star+\ve}(g(t)-g(-1+k\period^\star))\beta f'(\sops^\beta(t))dt}=0.\ee
	Let $\delta\in(0,1)$. 
	In view of the definitions of $\Iset_\ve$ and $L_{N,\Iset_\ve}^\infty(\R)$ in \eqref{eq:Isetve} and \eqref{eq:LCAbounded}, respectively, we can choose $\eta\in(0,\ve)$ sufficiently small so that
		\be\label{eq:gg1ub}\sup_{g\in L_{N,\Iset_\ve}^\infty(\R)}|g(t)-g(-1+k\period^\star)|<\frac{ be^{-\alpha q_1^\star/2}}{2\lb a+b\rb}\delta,\qquad t\in(-1+k\period^\star-\eta,-1+k\period^\star+\eta).\ee
	Then by upper bound in \eqref{eq:gtgmubeta0}, the limit established in Lemma \ref{lem:limbetaintA0} and the definition of $\Iset$ in \eqref{eq:Iset},
	\begin{align}
		&\lim_{\beta\to\infty}\sup_{g\in L_{N,\Iset_\ve}^\infty(\R)}\abs{\int_{-1+k\period^\star-\ve}^{-1+k\period^\star+\ve}(g(t)-g(-1+k\period^\star))\beta f'(\sops^\beta(t))dt} \notag \\
		&\qquad=\lim_{\beta\to\infty}\sup_{g\in L_{N,\Iset_\ve}^\infty(\R)}\abs{\int_{-1+k\period^\star-\eta}^{-1+k\period^\star+\eta}(g(t)-g(-1+k\period^\star))\beta f'(\sops^\beta(t))dt}\label{eq:gg1up1} .
	\end{align}
	By the respective definitions of $h^\beta$ and $h^\star$ in \eqref{eq:hbetafsopsbeta} and \eqref{eq:hbar}, the upper bound in \eqref{eq:gg1ub} and the lower bound for $\dot{\bar\sops}^\beta$ in \eqref{eq:betak}, for all $\beta>\beta^{(k)}$,
	\begin{align}
		&\sup_{g\in L_{N,\Iset_\ve}^\infty(\R)}\abs{\int_{-1+k\period^\star-\eta}^{-1+k\period^\star+\eta}(g(t)-g(-1+k\period^\star))(\dot{\bar\sops}^\beta(t))^{-1}dh^\beta(t)} \notag\\
		&\qquad<\frac{\delta}{a+b}\abs{f(\beta\bar\sops^\beta(-1+k\period^\star+\eta))-f(\beta\bar\sops^\beta(-1+k\period^\star-\eta))}\label{eq:gg1up2}. 
	\end{align}
	Using the limit in \eqref{eq:gg1up1} and the upper bound in \eqref{eq:gg1up2}, the convergence of $\nxbeta$ shown in Proposition \ref{prop:xbar_convergence}, the definition of $\bar\sops^\star$ given in \eqref{eq:xbar} and \eqref{eq:xbaralpha0}, and the limits for $f$ stated in part 2 of Assumption \ref{ass:main}, we see that
	\begin{align*}
		&\lim_{\beta\to\infty}\sup_{g\in L_{N,\Iset_\ve}^\infty(\R)}\abs{\int_{-1+k\period^\star-\ve}^{-1+k\period^\star+\ve}(g(t)-g(-1+k\period^\star))(\dot{\bar\sops}^\beta(t))^{-1}dh^\beta(t)}\\
		&\qquad<\frac{\delta}{a+b}\lim_{\beta\to\infty}\abs{f(\beta\bar\sops^\beta(-1+k\period^\star+\eta))-f(\beta\bar\sops^\beta(-1+k\period^\star-\eta))}\\
		&\qquad<\delta.
	\end{align*}
	Since $\delta\in(0,1)$ was arbitrary, this proves the limit in \eqref{eq:gtgmubeta0} holds, and so the limit in \eqref{eq:gmubeta-1} also holds.
	The proof of the limit in \eqref{eq:gmubetaq1} follows an analogous argument, but uses inequality \eqref{eq:betak2} in place of inequality \eqref{eq:betak}.
	To avoid repetition, we omit the details.
\end{proof}

We now use Lemmas \ref{lem:limbetaintA0} and \ref{lem:intgboundedcontinuous} to prove Proposition \ref{prop:mubetaconvergence}.

\begin{proof}[Proof of Proposition \ref{prop:mubetaconvergence}]
Fix $-\infty<t_0<t_1<\infty$.
We first prove the bound in \eqref{eq:supbetafprime}.
Define 
\begin{align}
	\label{eq:k0}
	k_0&=\max\{k\in\Z:-1+k\period^\star-q_1^\star/2\le t_0\},\\
	\label{eq:k1}
	k_1&=\min\{k\in\Z:q_1^\star+k\period^\star+q_1^\star/2\ge t_1\}.
\end{align}
For $k\in\{k_0,\dots,k_1\}$, let $\beta^{(k)}\ge\buniq$ be sufficiently large so that the uniform lower bounds on $\dot{\bar\sops}^\beta$ in inequalities \eqref{eq:betak} and \eqref{eq:betak2} hold for all $\beta>\beta^{(k)}$.
Then, for $\beta>\max\{\beta^{(k)}:k=k_0,\dots,k_1\}$,
\begin{align}
	\label{eq:1}
	\beta\int_{t_0}^{t_1}\abs{f'(\sops^\beta(t))}dt&\le\beta\int_{-1+k_0\period^\star-q_1^\star/2}^{q_1^\star+k_1\period^\star+q_1^\star/2}\abs{f'(\sops^\beta(t))}dt\\
	\notag
	&\le\sum_{k=k_0}^{k_1}\int_{-1+k\period^\star-q_1^\star/2}^{-1+k\period^\star+q_1^\star/2}\frac{\beta}{\abs{\dot{\sops}^\beta(t)}}\abs{f'(\sops^\beta(t))\dot{\sops}^\beta(t)}dt\\
	\notag
	&\qquad+\sum_{k=k_0}^{k_1}\beta\int_{-1+k\period^\star+q_1^\star/2}^{q_1^\star+k\period^\star-q_1^\star/2}\abs{f'(\sops^\beta(t))}dt\\
	\notag
	&\qquad+\sum_{k=k_0}^{k_1}\int_{q_1^\star+k\period^\star-q_1^\star/2}^{q_1^\star+k\period^\star+q_1^\star/2}\frac{\beta}{\abs{\dot{\sops}^\beta(t)}}\abs{f'(\sops^\beta(t))\dot{\sops}^\beta(t)}dt.
\end{align}
In view of the lower bound for $\dot{\bar\sops}^\beta$ in inequality \eqref{eq:betak}, we see that $\sops^\beta$ is monotone on the intervals $[-1+k\period-q_1^\star/2,-1+k\period+q_1^\star/2]$ and $[q_1^\star+k\period-q_1^\star/2,q_1^\star+k\period+q_1^\star/2]$.
Along with the fact that $f'\in L^1(\R)$ by part 3 of Assumption \ref{ass:main}, this implies
\begin{align}
	\label{eq:3}
	\int_{-1+k\period^\star-q_1^\star/2}^{-1+k\period^\star+q_1^\star/2}\frac{\beta}{\abs{\dot{\sops}^\beta(t)}}\abs{f'(\sops^\beta(t))\dot{\sops}^\beta(t)}dt
	&\le\frac{2e^{\alpha q_1^\star/2}}{b}\norm{f'}_{L^1(\R)},
\end{align}
and
\begin{align}
	\label{eq:4}
	\int_{q_1^\star+k\period^\star-q_1^\star/2}^{q_1^\star+k\period^\star+q_1^\star/2}\frac{\beta}{\abs{\dot{\sops}^\beta(t)}}\abs{f'(\sops^\beta(t))\dot{\sops}^\beta(t)}dt
	&\le\frac{2e^{\alpha q_1^\star/2}}{a}\norm{f'}_{L^1(\R)}.
\end{align}
By Lemma \ref{lem:limbetaintA0}, for each $k=k_0,\dots,k_1$,
\begin{align}
	\label{eq:2}
	\lim_{\beta\to\infty}\beta\int_{-1+k\period^\star+q_1^\star/2}^{q_1^\star+k\period^\star-q_1^\star/2}\abs{f'(\sops^\beta(t))}dt=0.
\end{align}
Combining inequalities \eqref{eq:1}--\eqref{eq:4}, the limit \eqref{eq:2} and using the assumption $a\ge b$ yields
\begin{align*}
	\limsup_{\beta\to\infty}\beta\int_{t_0}^{t_1}\abs{f'(\sops^\beta(t))}dt&\le\frac{4e^{\alpha q_1^\star/2}}{b}(k_1-k_0+1)\norm{f'}_{L^1(\R)}.
\end{align*}
This proves the bound in \eqref{eq:supbetafprime}.
Next we show that the limit in \eqref{eq:gbetafprimegsopsbardh} holds.
Let $N<\infty$ and $\ve\in(0,q_1^\star/2)$.
First recall that $h^\star$ is piecewise constant and its discontinuities are contained in the set $\Iset$, and so, for each $k=k_0,\dots,k_1$,
\begin{align*}
	\int_{-1+k\period^\star+\ve}^{q_1^\star+k\period^\star-\ve}g(t)(\dot{\bar\sops}^\star(t))^{-1}dh^\star(t)=\int_{q_1^\star+k\period^\star+\ve}^{-1+(k+1)\period^\star-\ve}g(t)(\dot{\bar\sops}^\star(t))^{-1}dh^\star(t)=0.
\end{align*}
Thus, using the respective definitions for $k_0$ and $k_1$ in \eqref{eq:k0} and \eqref{eq:k1}, and the definition of $L_{N,\Iset_\ve}^\infty(\R)$ in \eqref{eq:LCAbounded} with $\interval=\Iset_\ve$, we have 
\begin{align}\label{eq:intAgsumm}
	&\sup_{g\in L_{N,\Iset_\ve}^\infty(\R)}\abs{\int_{t_0}^{t_1}g(t)\beta f'(p^\beta(t))dt-\int_{t_0}^{t_1}g(t)(\dot{\bar\sops}^\star(t))^{-1}dh^\star(t)}\\
	\notag&\qquad
	\le\sum_{k=k_0}^{k_1}\sup_{g\in L_{N,\Iset_\ve}^\infty(\R)}\abs{\int_{-1+k\period^\star-\ve}^{-1+k\period^\star+\ve}g(t)dh^\beta(t)-\int_{-1+k\period^\star-\ve}^{-1+k\period^\star+\ve}g(t)(\dot{\bar\sops}^\star(t))^{-1}dh^\star(t)}\\
	\notag&\qquad\qquad
	+\sum_{k=k_0}^{k_1}\sup_{g\in L_{N,\Iset_\ve}^\infty(\R)}\abs{\int_{q_1^\star+k\period^\star-\ve}^{q_1^\star+k\period^\star+\ve}g(t)dh^\beta(t)-\int_{q_1^\star+k\period^\star-\ve}^{q_1^\star+k\period^\star+\ve}g(t)(\dot{\bar\sops}^\star(t))^{-1}dh^\star(t)}\\
	\notag&\qquad\qquad
	+C\sum_{k=k_0}^{k_1}\int_{-1+k\period^\star+\ve}^{q_1^\star+k\period^\star-\ve}\abs{\beta f'(p^\beta(t))}dt+C\sum_{k=k_0}^{k_1}\int_{q_1^\star+k\period^\star+\ve}^{-1+(k+1)\period^\star-\ve}\abs{\beta f'(p^\beta(t))}dt.
\end{align}
By the continuity of $g$ on $\Iset_\ve$ and the limits established in Lemma \ref{lem:intgboundedcontinuous}, the first two terms on the right hand side of the inequality \eqref{eq:intAgsumm} vanish as $\beta\to\infty$.
By the facts that $[-1+k\period^\star+\ve,q_1^\star+k\period^\star-\ve]\cap\Iset=\emptyset$ and $[q_1^\star+k\period^\star+\ve,-1+(k+1)\period^\star-\ve]\cap\Iset=\emptyset$, for each $k=k_0,\dots,k_1$, and the limit established in Lemma \ref{lem:limbetaintA0}, the last two terms on the right hand side of the inequality \eqref{eq:intAgsumm} also vanish as $\beta\to\infty$.
This completes the proof.
\end{proof}

\subsection{Convergence of the extended monodromy operators}
\label{sec:limitmonodromy}

In this section we prove convergence of extended monodromy operators.
Since the eigenvalues of the extended monodromy operator are invariant under time translations of the SOPS, it will be advantageous for us to consider a family of time-dependent extended monodromy operators described as follows.
Given $\beta>\buniq$ let $\sops^\beta\in C(\R,\R)$ denote the SOPS with period $\period^\beta$ that satisfy the relations in \eqref{eq:pphasefix}. 
For $\lambda\in\CC$, $s\in\R$ and $\vf\in\C(\CC)$, let $y^\beta=y^\beta(\lambda,s,\vf)\in C([s-1,\infty),\CC)$ denote the unique function that is continuously differentiable on $(s,\infty)$, satisfies the extended variational equation \eqref{eq:eve} for all $t>s$ and has initial condition $y_s^\beta=\vf$.
Let $\{\monoU_\lambda^\beta\}=\{\monoU_\lambda^\beta(s),\lambda\in\CC,s\in\R\}$ be the family of extended monodromy operators associated with $\xbeta$ defined, for each $\lambda\in\CC$ and $s\in\R$, by 
	\be\label{eq:Ulambdabetavf}\monoU_\lambda^\beta(s)\vf=y_{s+\period^\beta}(\lambda,s,\vf),\qquad\vf\in\C(\CC).\ee
Then, as remarked on in Section \ref{sec:msf}, $\monoU_\lambda^\beta(s)\in B_0(\C(\CC))$ and the spectrum of $\monoU_\lambda^\beta(s)$ is independent of $s\in\R$.
In addition, when $s=0$, the operator $\monoU_\lambda^\beta(0)$ coincides with the monodromy operator $\mono_\lambda$ defined in \eqref{eq:gmonodromy} for fixed $\beta>\buniq$.
Therefore, for fixed $\beta>\buniq$, $\sigma(\mono_\lambda)=\sigma(\monoU_\lambda^\beta(s))$ for all $s\in\R$.

Next, we define a family of limiting extended monodromy operators, which we denote by $\{\monoU_\lambda^\star\}=\{\monoU_\lambda^\star(s),\lambda\in\CC,s\in\R\}$, associated with the limiting SOPS $\bar\sops^\star$. 
In order to define the limiting extended monodromy operator, we first define a limiting extended variational equation.
In the previous section, Section \ref{sec:convmeasures},  we showed that $\beta f'(\sops(t))dt$ converges to $(\dot{\bar{\sops}}^\star(t))^{-1}dh^\star(t)$ in the weak-$^\ast$ topology as $\beta\to\infty$.
We define the limiting extended variational equation to be the formal limit of the extended variational equation \eqref{eq:eve} as $\beta\to\infty$ with $(\dot{\bar{\sops}}^\star(t))^{-1}dh^\star(t)$ serving as the formal limit of $\beta f'(\sops(t))dt$. 
In particular, for $\lambda\in\CC$, the limiting extended variational equation is given by
	\be\label{eq:limitvariational}dy^\star(t)=-\alpha y^\star(t)dt+\lambda y^\star(t-1)(\dot{\bar\sops}^\star(t-1))^{-1}dh^\star(t-1).\ee 
Recall the definitions of $\D(\CC)$ and $D([s-1,\infty),\CC)$ given in Section \ref{sec:notation_function_spaces}.
For $\lambda\in\CC$, $s\in\R$ and $\vf\in\D(\CC)$ we say $y^\star\in D([s-1,\infty),\CC)$ is a solution of the limiting extended variational equation \eqref{eq:limitvariational} on $[s-1,\infty)$ with initial condition $\vf$ if $y^\star_s=\vf$ and
	\be\label{eq:variationalbar}y^\star(t)=\vf(0)-\alpha\int_s^ty^\star(u)du+\lambda\int_{s-1}^{t-1}y^\star(u)(\dot{\bar\sops}^\star(u))^{-1}dh^\star(u),\qquad t>s.\ee

\begin{lem}\label{lem:ybar}
Let $\lambda\in\CC$, $s\in\R$ and $\vf\in\D(\CC)$ be given.
	Then $y^\star\in D([s-1,\infty),\CC)$ satisfies $y_s^\star=\psi$ and \eqref{eq:variationalbar} if and only if $y_s^\star=\psi$ and the following holds:
        \be\label{eq:ybar}y^\star(t)=e^{-\alpha(t-s)}y^\star(s)+\lambda\int_{s-1}^{t-1}e^{-\alpha(t-1-u)}y^\star(u)(\dot{\bar\sops}^\star(u))^{-1}dh^\star(u),\qquad t>s.\ee	
	Consequently, there exists a unique solution of the limiting extended variational equation \eqref{eq:limitvariational} on $[s-1,\infty)$ with initial condition $\vf$, which we denote by $y^\star(\lambda,s,\vf)$.
\end{lem}

\begin{remark}\label{rem:ydiscontinuities}
	It follows from equation \eqref{eq:ybar} for $y^\star$, the definition of $h^\star$ in \eqref{eq:hbar} and equation \eqref{eq:xbardot} for $\dot{\bar{\sops}}^\star$, that given any $\lambda\in\CC$, $s\in\R$ and $\vf\in\C(\CC)$, the function $y^\star(\lambda,s,\vf)$ is continuous on $[s-1,\infty)\setminus\Jset$, where $\Jset$ is the set defined in \eqref{eq:Jset}.
\end{remark}

\begin{proof}
    Suppose $y^\star\in D([s-1,\infty),\CC)$ satisfies $y^\star_s=\vf$ and equation \eqref{eq:variationalbar}.
    Let $t_0=s$ and $t_k=\inf\{u>t_{k-1}:u\in\Jset\}$ for all $k\in\N$.
    From the definition of $h^\star$ in \eqref{eq:hbar} and $\Jset$ in \eqref{eq:Jset} we see that $h^\star$ is constant on $(t_k-1,t_{k+1}-1)$ for all $k\in\N$ and $t_k\to\infty$ as $k\to\infty$.
    We use the principle of mathematical induction to show that $y^\star$ satisfies \eqref{eq:ybar} for all $t\in(s,t_k]$, for $k\in\N\cup\{0\}$.
    The base case $k=0$ trivially holds.
    Now let $k\in\N$ and suppose the following induction hypothesis holds: $y^\star$ satisfies \eqref{eq:ybar} for all $t\in(s,t_k]$.
    By \eqref{eq:variationalbar} and the fact that $h^\star$ is constant on $(t_k-1,t_{k+1}-1)$, we have
    	\begin{align*}
		y^\star(t)=y^\star(t_k)e^{-\alpha(t-t_k)},\qquad t\in(t_k,t_{k+1}).
    	\end{align*}
	Furthermore, again using \eqref{eq:variationalbar} and the fact that $h^\star$ is constant on $(t_k-1,t_{k+1}-1)$, we see that
	\begin{align*}
		y^\star(t_{k+1})=y^\star(t_k)e^{-\alpha(t-t_k)}+\lambda\int_{t_k-1}^{t_{k+1}-1}e^{-\alpha(t_{k+1}-1-u)}y^\star(u)(\dot{\bar\sops}^\star(u))^{-1}dh^\star(u).
	\end{align*}
	Therefore, by the induction hypothesis and the fact that $h^\star$ is constant on $(t_k-1,t_{k+1}-1)$,
	\begin{align*}
		y^\star(t)=e^{-\alpha(t-s)}y^\star(s)+\lambda\int_{s-1}^{t-1}e^{-\alpha(t-1-u)}y^\star(u)(\dot{\bar\sops}^\star(u))^{-1}dh^\star(u),\qquad t\in(t_k,t_{k+1}].
	\end{align*}
	In particular, $y^\star$ satisfies \eqref{eq:ybar} on $(s,t_{k+1}]$.
	This completes the induction step.
	Therefore, by the principle of mathematical induction an the fact that $t_k\to\infty$ as $k\to\infty$, $y^\star$ satisfies \eqref{eq:ybar}.
	To complete the proof, a similar argument can be used to show that if $y^\star$ satisfies \eqref{eq:ybar} then $y^\star$ also satisfies \eqref{eq:variationalbar}.
	To avoid repetition, we omit the details.
\end{proof}

For $\lambda\in\CC$ and $s\in\R$ define $\monoU_\lambda^\star(s)\in B(\D(\CC))$ by
	\be\label{eq:Ubar}\monoU_\lambda^\star(s)\vf=y^\star_{s+\period^\star}(\lambda,s,\vf),\qquad\vf\in \D(\CC).\ee
The following is the main convergence result of this section.

\begin{prop}\label{prop:mono_convergence}
	Let $s\in(-q_1^\star,0)$ and $K$ be a compact set in $\CC$. Then
	\begin{equation*}
	\lim_{\beta\to\infty}\sup_{\lambda\in K}\lVert \monoU_\lambda^\beta(s)-\monoU_\lambda^\star(s)\rVert=0.
	\end{equation*}  
\end{prop}

The remainder of this section is devoted to the proof of Proposition \ref{prop:mono_convergence}.
Our first step is to explicitly construct solutions of the limiting variational equation \eqref{eq:limitvariational}, which yields an explicit form of the limiting extended monodromy operator.
To this end, define $F:(-q_2^\star,0)\times\CC^3\to\CC$ by
\be\label{eq:Fslambdaz1z2}F(s,\lambda,z_1,z_2)=-\left[z_1-\lambda z_2(1+ab^{-1})e^{-\alpha s}\right]\frac{\lambda-\vr_1}{1-\vr_2}e^{-\alpha(q_2^\star+1)}.\ee
Also, define the continuous function $Q:\CC\to[2,\infty)$ by
\be\label{eq:Glambda}Q(\lambda)=\lsb1+\abs{\lambda}(1+ab^{-1})\rsb\lsb1+\abs{\lambda}(1+a^{-1}b)\rsb.\ee
Observe that $F$ and $Q$ also depend on $\alpha\ge0$ and $a,b>0$, which are fixed throughout this section.

\begin{lem}\label{lem:Uspectrum}
	For $\lambda\in\CC$, $s\in(-q_1^\star,0)$ and $\vf\in\C(\CC)$, the solution $y^\star(\lambda,s,\vf)$ of the limiting extended variational equation \eqref{eq:variationalbar} with initial condition $\vf$ is continuously differentiable on $(0,q_1^\star+1)\cup(q_1^\star+1,\period^\star)$,
		\be\label{eq:yastbound}\sup_{t\in[s-1,\period^\star-\ve)}\abs{y^\star(t;\lambda,s,\vf)}\le\norm{\psi}_{[-1,0]}Q(\lambda)\quad\text{for}\quad \ve\in(0,\period^\star).\ee
	and its derivative satisfies
		\be\abs{\frac{dy^\star(t;\lambda,s,\psi)}{dt}}\le\norm{\vf}_{[-1,0]}Q(\lambda)\max\{\alpha,1\},\qquad t\in(0,q_1^\star+1)\cup(q_1^\star+1,\period^\star).\ee 
	In addition, the associated limiting extended monodromy operator is explicitly given by
		\be\label{eq:Uastexplicit}\lsb \monoU_\lambda^\star(s)\vf\rsb(\theta)=F(s,\lambda,\psi(0),\psi(-1-s))e^{-\alpha\theta},\qquad\theta\in[-1,0].\ee
	Consequently, the restriction of $\monoU_\lambda^\star(s)$ to $\C(\CC)$ is a compact linear operator from $\C(\CC)$ to $\C(\CC)$, i.e., $\monoU_\lambda^\star(s)|_{\C(\CC)}\in B_0(\C(\CC))$, and the spectrum of $\monoU_\lambda^\star(s)|_{\C(\CC)}$ is equal to $\{0,\nu_\star(\lambda)\}$.
	Furthermore, if $\lambda\not\in\{\vr_1,\vr_2\}$, then $\nu_\star(\lambda)$ is the unique nonzero eigenvalue of $\monoU_\lambda^\star(s)|_{\C(\CC)}$, is a simple eigenvalue of $\monoU_\lambda^\star(s)|_{\C(\CC)}$, and has corresponding unit eigenfunction $\vf_0\in\C(\CC)$ given by 
		\be\label{eq:psi0}\vf_0(\theta)=e^{-\alpha(\theta+1)},\qquad\theta\in[-1,0].\ee
\end{lem} 

\begin{proof}
	First observe that by the definitions of $\vr_1,\vr_2$ in \eqref{eq:rho} and the definitions of $q_1^\star,q_2^\star$ in \eqref{eq:q1bar_q2bar} and \eqref{eq:q1bar_q2baralpha0}, we have
	\begin{align}
	\frac{\lambda-\vr_2}{1-\vr_1}e^{\alpha(q_2^\star+1)}&=\lambda e^\alpha\lb1+\frac{a}{b}\rb-1,\label{eq:lambdavr2}\\
	\frac{\lambda-\vr_1}{1-\vr_2}e^{\alpha(q_1^\star+1)}&=\lambda e^\alpha\lb1+\frac{b}{a}\rb-1.\label{eq:lambdavr1}
	\end{align}
	Let $\lambda\in\CC$, $s\in(-q_1^\star,0)$, $\vf\in\C(\CC)$ and $y^\star=y^\star(\lambda,s,\vf)$.
	By the equation for $y^\star$ in \eqref{eq:ybar} and Remark \ref{rem:dh} we see that $y^\star(t)=\vf(0)e^{-\alpha(t-s)}$ for $t\in[s,0)$ 
	\begin{align}\label{eq:yast0}
	   y^\star(0)=\vf(0)e^{\alpha s}+\lambda y^\star(-1-s)\mu^\star(\{-1\})=\vf(0)e^{\alpha s}-\lambda\vf(-1-s)(1+ab^{-1}),
	\end{align}
	where we have used the fact that $q_1^\star\le 1$, $-1-s\in[-1,0]$ and the definition of $h^\star$ in \eqref{eq:hbar}. 
	Continuing, we have $y^\star(t)=y^\star(0)e^{-\alpha t}$ for $t\in[0,q_1^\star+1)$ and
	\begin{align}\label{eq:yast0q11}
	   y^\star(q_1^\star+1)&=y^\star(0)e^{-\alpha(q_1^\star+1)}-\lambda y^\star(0)e^{-\alpha q_1^\star}(1+a^{-1}b)=-y^\star(0)\frac{\lambda-\vr_1}{1-\vr_2},
	\end{align}
	where we have used the definition of $h^\star$ in \eqref{eq:hbar} and the identity in \eqref{eq:lambdavr1}. 
	Next,
	\be\label{eq:yastbarq11}y^\star(t)=-y^\star(0)\frac{\lambda-\vr_1}{1-\vr_2}e^{-\alpha (t-q_1^\star-1)},\qquad t\in[q_1^\star+1,\period^\star).\ee
	Along with the expression for $y^\star(0)$ in equation \eqref{eq:yast0} this proves that equation \eqref{eq:Uastexplicit} for $\monoU_\lambda^\star(s)\vf$ holds, which immediately implies that $\monoU_\lambda^\star(s)|_{\C(\CC)}\in B_0(\C(\CC))$.
	In addition, it follows that $y^\star$ is continuously differentiable on $(0,q_1^\star+1)\cup(q_1^\star+1,\period^\star)$ with derivative bounded by $\norm{\vf}_{[-1,0]}Q(\lambda)\max\{\alpha,1\}$, and the bound in \eqref{eq:yastbound} holds.
	Now suppose $\mu$ is a nonzero eigenvalue of $\monoU_\lambda^\star(s)|_{\C(\CC)}$.
	Let $\vf\in\C(\CC)$ be a corresponding eigenfunction.
	Using the equation for $\monoU_\lambda^\star(s)\vf$ in \eqref{eq:Uastexplicit} we obtain
		\be\label{eq:muvf}\mu\vf(\theta)=F(s,\lambda,\vf(0),\vf(-1-s))e^{-\alpha\theta},\qquad\theta\in[-1,0].\ee
	Therefore, $\vf(\theta)=ze^{-\alpha\theta}$, $\theta\in[-1,0]$, for some constant $z\in\CC$.
	First, substituting this form of $\vf$ into equation \eqref{eq:muvf}, then using the definition of $F$ in \eqref{eq:Fslambdaz1z2}, the identity in \eqref{eq:lambdavr2} and the definition of $\nu_\star$ in \eqref{eq:nuast} yields $\mu=\nu_\star(\lambda)$.
	This proves that $\nu_\star(\lambda)$ is the only nonzero eigenvalue (provided $\lambda\not\in\{\vr_1,\vr_2\}$) and its associated generalized eigenspace is equal to the span of $\vf_0$.
	Thus, $\nu_\star(\lambda)$ is a simple eigenvalue of $\monoU_\lambda^\star(s)$.
\end{proof}

\begin{remark}
	Given $\lambda\in\CC$, $s\in\R$ and $\vf\in\C(\CC)$, we let $y^\beta=y^\beta(\lambda,s,\vf)\in C([s-1,\infty),\CC)$ denote the unique solution of the extended variational equation \eqref{eq:eve} that satisfies $y_s^\beta=\vf$.	
	Then $y^\beta$ satisfies
	\begin{equation}\label{eq:ylambdasolution}
		y^\beta(t)=\vf(0)e^{-\alpha(t-s)}+\lambda\beta\int_{s-1}^{t-1}e^{-\alpha(t-u-1)}f'(\sops^\beta(u))y^\beta(u)du,\qquad t>s.
	\end{equation}
\end{remark}

\begin{lem}\label{lem:yuniformbound}
	Suppose $s,t\in\R\setminus\Iset$ satisfy $s\le t$ and $K$ is a compact subset of $\CC$.
	There exists $R_0=R_0(s,t,K)<\infty$ such that for any $\lambda\in K$ and $\beta>\buniq$,
		$$\sup_{u\in[s-1,t]}\abs{y^\beta(u;\lambda,s,\vf)}\le R_0\norm{\vf}_{[-1,0]}.$$
\end{lem}

\begin{proof}
	Let $\Lambda_K=\sup\{|\lambda|:\lambda\in K\}<\infty$.
	It follows from equation \eqref{eq:ylambdasolution} that for $\lambda\in K$,
        $$\sup_{u\in[t-1,s]}|y^\beta(u;\lambda,s,\vf)|\le \norm{\vf}_{[-1,0]}\exp\lsb \Lambda_K\beta\int_{s-1}^{t-1}\abs{f'(\sops^\beta(r))}dr \rsb.$$
	Then by inequality \eqref{eq:supbetafprime} of Proposition \ref{prop:mubetaconvergence}, the lemma holds with
		$$R_0(s,t,K)=\exp\lsb\Lambda_K\sup_{\beta>\buniq}\beta\int_{s-1}^{t-1}\abs{f'(\sops^\beta(r))}dr\rsb<\infty.$$
\end{proof}

\begin{lem}\label{lem:y_convergence}
	Let $\ve\in(0,q_1^\star/4)$ and $K$ be a compact set in $\CC$. 
	Then for all $s\in(-q_1^\star+2\ve,-2\ve)$,
  	\begin{align*}
  		\lim_{\beta\to\infty}\sup\lcb\abs{y^\beta(t;\lambda,s,\vf)-y^\star(t;\lambda,s,\vf)}:t\in\mathcal{I}_{s,\ve},\lambda\in K,\vf\in\C(\CC)\text{ s.t.\ }\norm{\vf}_{[-1,0]}\le 1\rcb=0,
  	\end{align*}
  	where $\mathcal{I}_{s,\ve}=[s-1,s+\bar \period+2\ve]\setminus\Jset_\ve=[s-1,-\ve]\cup[\ve,q_1^\star+1-\ve]\cup[q_1^\star+1+\ve,s+\period^\star+2\ve]$.
\end{lem}

\begin{proof}
	Let $s\in(-q_2^\star+2\ve,-2\ve)$.
    Define the positive integer $m$ and the finite sequence $\{r_k\}_{k=1,\dots,m}$ in $[s,s+\period^\star+2\ve]$ as follows:
    set $r_1=s$ and for $k\ge1$ such that $r_k<s+\period^\star+2\ve$, recursively define
    	$$r_{k+1}=\max\lcb[r_k+1/2,r_k+1]\cap\mathcal{I}_{s,\ve}\rcb.$$
    If $r_{k+1}=t+\period^\star+2\ve$ for some $k\in\N$, set $m=k+1$ and end the sequence. 
    Since $\ve<q_1^\star/4\le 1/4$ and $\mathcal{I}_{s,\ve}$ is a closed set, it follows that $[r_k+1/2,r_k+1]\cap\mathcal{I}_{s,\ve}$ is nonempty for each $k=1,\dots,m-1$, and
    	\be\label{eq:skmathcalItve} r_k\in\mathcal{I}_{s,\ve},\qquad k=1,\dots,m.\ee
    Let $R_0=R_0(t,t+\period^\star+2\ve,K)$ be as in Lemma \ref{lem:yuniformbound} and set
    	$$\Lambda_K=\sup_{\lambda\in K}Q(\lambda)<\infty,$$
    where $Q(\lambda)$ is defined in \eqref{eq:Glambda} and clearly satisfies $Q(\lambda)\ge\abs{\lambda}$ for all $\lambda\in K$.
    Define
    	\be\label{eq:Cdef}C_1=2+R_0+\Lambda_K+\Lambda_K\sup_{\beta>\buniq}\beta\int_{s-1}^{s+\period^\star+2\ve}\abs{f'(\sops^\beta(u))}du,\ee
    where the last term on the right hand side is finite due to the bound in \eqref{eq:supbetafprime} of Proposition \ref{prop:mubetaconvergence}.
    Let $\eta>0$ and define
        \be\label{eq:deltak}\eta_k=\eta C_1^{-(m-k)},\qquad k=1,\dots,m.\ee
	By Proposition \ref{prop:xbar_convergence}, there is a $\beta'>\buniq$ such that for each $\beta>\beta'$,
    	\be\label{eq:periodbetabarperiod}|{\period^\beta-\period^\star}|<2\ve.\ee
    From Lemma \ref{lem:Uspectrum} and the fact that $s+\period^\star+2\ve<\period^\star$ we see that for any $\lambda\in K$ and $\vf\in\C(\CC)$ satisfying $\norm{\vf}_{[-1,0]}\le1$, $y^\star(\lambda,s,\vf)$ is continuously differentiable on $(0,q_1^\star+1)\cup(q_1^\star+1,\period^\star)$ with derivative bounded by $\Lambda_K\max\{\alpha,1\}$ and $y^\star(\lambda,s,\vf)$ satisfies the uniform bound
    	$$\sup_{t\in[s-1,s+\period^\star+2\ve]}\abs{y^\star(t;\lambda,s,\vf)}\le\Lambda_K.$$
    Consequently, $y^\star(\lambda,s,\vf)$ is Lipschitz continuous on $\mathcal{I}_{s,\ve}$ with Lipschitz constant depending only on $\alpha$, $K$ and $\ve$.
	Therefore, by Proposition \ref{prop:mubetaconvergence} and Lemma \ref{lem:yuniformbound}, and by choosing $\beta'$ possibly larger, we have, for each $\beta>\beta'$, $\lambda\in K$, $\vf\in\C(\CC)$ satisfying $\norm{\vf}_{[-1,0]}\le1$, $k=2,\dots,m$ and $t\in[r_k,r_{k+1}]\cap\mathcal{I}_{s,\ve}$,
    \begin{align}
    	\label{eq:intAybarmubetamu}
    	&\abs{\int_{r_k-1}^{t-1} e^{-\alpha(t-u) }y^\star(u;\lambda,s,\vf)\beta f'(p^\beta(u))du-\int_{r_k-1}^{t-1} e^{-\alpha(t-u) }y^\star(u;\lambda,s,\vf)(\dot{\bar\sops}^\star(u))^{-1}dh^\star(u)}\\
    	\notag
    	&\qquad<\frac{\eta_k}{\Lambda_K}.
    \end{align}
    By Lemma \ref{lem:limbetaintA0} and Proposition \ref{prop:mubetaconvergence}, and again choosing $\beta'$ possibly larger, for each $\beta>\beta'$, we have
    	\be\label{eq:hbetaJve}\int_{r_k-1}^{r_{k+1}-1}1_{[r_k-1,r_{k+1}-1]\setminus\Jset_\ve}(u)\beta|f'(p^\beta(u))|du<\frac{\eta_k}{\Lambda_K}.\ee
	Suppose $\beta>\beta'$. 
    Let $\lambda\in K$ and $\vf\in\C(\CC)$ be such that $\norm{\vf}_{[-1,0]}\le1$.
    We use finite induction to prove that for each $k=1,\dots,m$,
    	\be\label{eq:supybarysk}\sup\lcb\abs{y^\beta(t;\lambda,s,\vf)-y^\star(t;\lambda,s,\vf)}:t\in[s,r_k]\cap\mathcal{I}_{s,\ve}\rcb<\eta_k.\ee
    Since $r_m=s+\period^\star+2\ve$ and $\eta>0$ was arbitrary, this will complete the proof.
    For notational convenience, throughout the remainder of the proof we write $y^\beta$ and $y^\star$ for $y^\beta(\lambda,s,\vf)$ and $y^\star(\lambda,s,\vf)$, respectively.
    The base case ($k=1$) holds because $r_1=s$ and $y_s^\beta=y^\star_s=\vf$ by definition.
    Next, we prove the induction step.
    Suppose the bound in \eqref{eq:supybarysk} holds for some $k\in\{1,\dots,m-1\}$.
    Let $t\in[r_k,r_{k+1}]\cap\mathcal{I}_{s,\ve}$.
    The respective equations for $y^\beta$ and $y^\star$ in \eqref{eq:ylambdasolution} and \eqref{eq:ybar} imply the following equality
    \begin{align}
        \notag
    	y^\beta(t)-y^\star(t)&=y^\beta(r_k)-y^\star(r_k)\\
    	\notag
    	&\qquad+\lambda\int_{r_k-1}^{t-1}1_{[r_k-1,t-1]\setminus\Jset_\ve}(u)e^{-\alpha(t-u)}(y^\beta(u)-y^\star(u))\beta f'(p^\beta(u))du\\
    	\notag
    	&\qquad+\lambda\int_{r_k-1}^{t-1}1_{\Jset_\ve}(u)e^{-\alpha(t-u)}(y^\beta(u)-y^\star(u))\beta f'(p^\beta(u))du\\	
    	\notag
    	&\qquad+\lambda\int_{r_k-1}^{t-1}e^{-\alpha(t-u)}y^\star(u)\beta f'(p^\beta(u))du\\
    	&\qquad-\lambda\int_{r_k-1}^{t-1}e^{-\alpha(t-u)}(\dot{\bar\sops}^\star(u))^{-1}y^\star(u)dh^\star(u) \label{eq:ybeta_star1}.
    \end{align}
    In view of the inclusion stated in \eqref{eq:skmathcalItve} and the induction hypothesis \eqref{eq:supybarysk}, we have
    	\be\label{eq:ybeta_star2} \abs{y^\beta(r_k)-y^\star(r_k)}<\eta_k.\ee
    Due to the bound $\abs{\lambda}\le\Lambda_K$, the uniform bound on $y^\beta$ established in Lemma \ref{lem:yuniformbound}, the uniform bound for $y^\star$ in \eqref{eq:yastbound}, and the bound in \eqref{eq:hbetaJve}, it follows that
    \begin{align}
    	&\abs{\lambda\int_{r_k-1}^{t-1}1_{[r_k-1,t-1]\setminus\Jset_\ve}(u)e^{-\alpha(t-u)}(y^\beta(u)-y^\star(u))\beta f'(p^\beta(u))du}\notag\\
    	&\qquad\le\Lambda_K\sup_{u\in[r_k-1,r_{k+1}-1]}\lb\abs{y^\beta(u)}+\abs{y^\star(u)}\rb\int_{r_k-1}^{r_{k+1}-1}1_{[r_k-1,r_{k+1}-1]\setminus\Jset_\ve}\beta |f'(p^\beta(u))|du\notag\\
    	&\qquad\le\eta_k\lb R_0+\Lambda_K\rb \label{eq:ybeta_star3}. 
    \end{align}
    Continuing, by the bound $\abs{\lambda}\le\Lambda_K$ and our induction hypothesis \eqref{eq:supybarysk}, we have
    \begin{align}
     	&\abs{\lambda\int_{r_k-1}^{t-1}1_{\Jset_\ve}(u)e^{-\alpha(t-u)}(y^\beta(u)-y^\star(u))\beta f'(p^\beta(u))du}\le\eta_k\Lambda_K\sup_{\beta>\beta'}\beta\int_{r_k-1}^{r_{k+1}-1}\abs{f'(p^\beta(u))}du\label{eq:ybeta_star4}.
    \end{align}
    Then, by the bound in \eqref{eq:intAybarmubetamu} and the bound $\abs{\lambda}\le\Lambda_K$, we see that
    	\be\label{eq:ybeta_star5}\abs{\lambda\int_{r_k-1}^{t-1}e^{-\alpha(t-u)}y^\star(u)\beta f'(p^\beta(u))du-\lambda\int_{r_k-1}^{t-1}e^{-\alpha(t-u)}(\dot{\bar\sops}^\star(u))^{-1}y^\star(u)dh^\star(u)}<\eta_k.\ee
	Using the expression for $y^\beta(t)-y^\star(t)$ in equation \eqref{eq:ybeta_star1} along with the inequalities \eqref{eq:ybeta_star1}--\eqref{eq:ybeta_star5} yields
	\begin{align*}
		\abs{y^\beta(t)-y^\star(t)}&\le\eta_k\lb2+R_0+\Lambda_K+\Lambda_K\sup_{\beta>\beta'}\beta\int_{r_k-1}^{r_{k+1}-1}\abs{f'(\sops^\beta(u))}du\rb.
	\end{align*}
	Substituting in with the respective definitions for $C_1$ and $\eta_{k+1}$ in \eqref{eq:Cdef} and \eqref{eq:deltak}, we obtain
		\be\label{eq:ybeta_star6}\abs{y^\beta(t)-y^\star(t)}\le\eta_kC_1\le\eta_{k+1}.\ee
	Since the inequality \eqref{eq:ybeta_star6} holds for all $t\in[r_k,r_{k+1}]$, it follows that the induction hypothesis \eqref{eq:supybarysk} holds with $k+1$ in place of $k$.
	Then by the principle of mathematical induction, the bound in \eqref{eq:supybarysk} holds for all $k=1,\dots,m$, thus completing the proof.
\end{proof}

We can now prove the main result of this section.

\begin{proof}[Proof of Proposition \ref{prop:mono_convergence}]
	Fix $\eta>0$.
	Choose $\ve\in(0,q_1^\star/2)$ sufficiently small so that $s\in(-q_1^\star+\ve,-\ve)$.
    By Proposition \ref{prop:xbar_convergence} and Lemma \ref{lem:y_convergence} there exists $\beta^\dagger>0$ such that for all $\beta>\beta^\dagger$, the following hold:
    	\be\label{eq:pbetabarpve}|\period^\beta-\period^\star|<\min\lb\ve,\frac{\eta}{2\max\{\alpha,1\}\sup_{\lambda\in K}Q(\lambda)}\rb,\ee
    where $Q(\lambda)$ is the continuous function defined in \eqref{eq:Glambda}, and, for all $\lambda\in K$ and $\vf\in\C(\CC)$ satisfying $\norm{\vf}_{[-1,0]}\le1$,
    	\be\label{eq:pbetabarpve1}\sup\lcb\abs{y^\beta(t;\lambda,s,\vf)-y^\star(t;\lambda,s,\vf)}:t\in[s+\period^\star-1-\ve,s+\period^\star+\ve]\rcb<\frac{\eta}{2},\ee
    where we have used the fact that $\period^\star=q_1^\star+q_2^\star+2$ and so $[s+\period^\star-1-\ve,s+\period^\star+\ve]\subset[q_1^\star+1+\ve,s+\period^\star+\ve]$.
    Suppose $\beta>\beta^\dagger$. 
    Let $\lambda\in K$ and $\vf\in\C(\CC)$ be such that $\norm{\vf}_{[-1,0]}\le1$.
    By the Lipschitz continuity of $y^\star(\lambda,s,\vf)$ (see Lemma \ref{lem:Uspectrum}), the fact that $\norm{\vf}_{[-1,0]}\le1$ and the bound in \eqref{eq:pbetabarpve},
    \begin{align}\label{eq:barypbetabarybarp}
    	|y^\star(s+\period^\beta+\theta;\lambda,s,\vf)-y^\star(s+\period^\star+\theta;\lambda,s,\vf)|<\frac{\eta}{2}.
    \end{align}
    Therefore, by the triangle inequality and the last two inequalities \eqref{eq:pbetabarpve1} and \eqref{eq:barypbetabarybarp},
    \begin{align*}
    	\lVert \monoU_\lambda^\beta(s)\vf-\monoU_\lambda^\star(s)\vf\rVert_{[-1,0]}&=\sup_{\theta\in[-1,0]}\abs{y^\beta(s+\period^\beta+\theta;\lambda,s,\vf)-y^\star(s+\period^\star+\theta;\lambda,s,\vf)}<\eta.
    \end{align*}
    Since this holds for all $\lambda\in K$ and $\vf\in\C(\CC)$ satisfying $\norm{\vf}_{[-1,0]}\le 1$, and $\eta>0$ was arbitrary, this proves the proposition.
\end{proof}

\subsection{Proof of Theorem \ref{thm:1a}}
\label{sec:proofmonodromy}

\begin{proof}[Proof of Theorem \ref{thm:1a}]
	Suppose $\ve=q_1^\star/2$ and $s\in (-q_1^\star+\ve,-\ve)$.
	Throughout this proof we write $\monoU_\lambda^\star(s)$ to mean its restriction to $\C(\CC)$.
	Let $\{\beta_k\}_{k=1}^\infty$ be a sequence in $(\buniq,\infty)$ with $\lim_{k\to\infty}\beta_k=\infty$ and $\{\lambda_k\}_{k=1}^\infty$ be a sequence in $K$.
	By possibly taking a subsequence, we can assume there exists $\lambda\in K$ such that $\lim_{k\to\infty}\lambda_k=\lambda$.
	Then by Proposition \ref{prop:mono_convergence}, the explicit expression for $\monoU_\lambda^\star(s)$ in \eqref{eq:Uastexplicit} and the continuity of $\lambda\mapsto F(s,\lambda,\vf(0),\vf(-1-s))$ that follows from its definition in \eqref{eq:Fslambdaz1z2},
		$$\lim_{k\to\infty}\lVert \monoU_{\lambda_k}^{\beta_k}(s)-\monoU_\lambda^\star(s)\rVert\le\lim_{k\to\infty}\lVert \monoU_{\lambda_k}^{\beta_k}(s)-\monoU_{\lambda_k}^\star(s)\rVert+\lim_{k\to\infty}\lVert \monoU_{\lambda_k}^\star(s)-\monoU_\lambda^\star(s)\rVert=0.$$
	Therefore, by the continuity of $\sigma(\cdot)$ (see Section \ref{sec:notation_linear_operators}), Lemma \ref{lem:Uspectrum} and the continuity of $\nu_\star(\cdot)$ that follows from its definition in \eqref{eq:nuast},
	\begin{align*}
		\lim_{k\to\infty}d_H\lb\sigma(\monoU_{\lambda_k}^{\beta_k}(s)),\sigma(\monoU_\lambda^\star(s))\rb&\le \lim_{k\to\infty} d_H\lb\sigma(\monoU_{\lambda_k}^{\beta_k}(s)),\{0,\nu_\star(\lambda)\}\rb+\lim_{k\to\infty}|\nu_\star(\lambda)-\nu_\star(\lambda_k)|\\
		&=0.
	\end{align*}
	Since this holds for every such pair of sequences $\{\beta_k\}_{k=1}^\infty$ and $\{\lambda_k\}_{k=1}^\infty$, we have
		$$\lim_{\beta\to\infty}\sup_{\lambda\in K}d_H(\sigma(\monoU_\lambda^\beta(s)),\{0,\nu_\star(\lambda)\})=0.$$
	The theorem then follows because $\sigma(\mono_\lambda)=\sigma(U_\lambda^\beta(s))$ for fixed $\beta$.
\end{proof}

\section{Limits of the derivatives of the extended characteristic multipliers}
\label{sec:boundarycases}

In this section we show that the derivative (with respect to $\lambda$) of the dominant multiplier of the extended monodromy operator $\mono_\lambda$ exists for $\beta$ sufficiently large and converges to $\pd\nu_\star(\lambda)/\pd\lambda$, as $\beta\to\infty$.
This is used in the next section to characterize the sets $\stable$ and $\unstable$, defined in \eqref{eq:stablesets}, near $\lambda=1$ and $\lambda=0$, and on the interval $[0,1)$.
The following is the main result of this section.
Recall the definition of $\monoU_\lambda^\beta(s)$ in \eqref{eq:Ulambdabetavf}, and that we have fixed $\alpha\ge0$ and $f$ satisfying Assumption \ref{ass:main}.

\begin{prop}
\label{prop:nu_lambda_simple}
	Let $s\in(-q_1^\star,0)$.
	Suppose $K$ is a compact set in $\CC\setminus\{\vr_1,\vr_2\}$.
    	For all $\beta>\buniq$ sufficiently large, there exists a unique function $\nu_\beta:K\to\CC$ such that $\nu_\beta(\cdot)$ is holomorphic on $\interior(K)$, and
    	\be\label{eq:rhoUbetanubeta}\nu_\beta(\lambda)\in\sigma(\monoU_\lambda^\beta(s))\qquad\text{and}\qquad\rho(\monoU_\lambda^\beta(s))=|\nu_\beta(\lambda)|\qquad\text{for all }\lambda\in K.\ee
    	Furthermore, $\nu_\beta(r)\in\R$ and $\pd\nu_\beta(r)/\pd\lambda\in\R$ for all $r\in K\cap\R$; if $1\in K$, then $\nu_\beta(1)=1$; and if $\alpha=0$ and $0\in K$, then $\nu_\beta(0)=1$.
    	In addition,
    	\begin{align}
	   	\label{eq:mubetanu} \lim_{\beta\to\infty}\sup_{\lambda\in K}\abs{\nu_\beta(\lambda)-\nu_\star(\lambda)}=0\qquad\text{and}\qquad\lim_{\beta\to\infty}\sup_{\lambda\in K}\abs{\frac{\pd\nu_\beta(\lambda)}{\pd\lambda}-\frac{\pd\nu_\star(\lambda)}{\pd\lambda}}&=0.
   	\end{align}
\end{prop}

The remainder of this section is devoted to the proof of Proposition \ref{prop:nu_lambda_simple}.

\subsection{Regularity of the extended monodromy operators}
\label{sec:monoregular}

In this section we show that for certain fixed $s\in\R$ the operator $\monoU_\lambda^\beta(s)$ is continuous in $(\beta,\lambda)$ and holomorphic in $\lambda$.
These results will be used in the next section to prove convergence of the derivatives of the extended monodromy operators.

To this end, we first need some definitions.
Define the solution maps $S:(0,\infty)\times\C(\R)\times\R_+\to\C(\R)$ and $T:(0,\infty)\times\C(\R)\times\R_+\to\R$, for $(\beta,\phi,t)\in\C(\R)\times\R_+$, by
\begin{align*}
	S(\beta,\phi,t)&=x_t^\beta(\phi),\\
	T(\beta,\phi,t)&=x^\beta(t;\phi).
\end{align*}
Then $S$ is continuous on its domain and continuously differentiable on $(0,\infty)\times\C(\R)\times(1,\infty)$ (see, e.g., \cite[Chapter 2, Theorem 4.1]{Hale1993} \& \cite[Chapter VII, Theorem 6.2]{Diekmann1991}).
Since the projection mapping $\vf\mapsto\vf(-1)$ from $\C(\R)$ to $\R$ is continuously differentiable, the function $T$ is continuously differentiable on $(0,\infty)\times\C(\R)\times(0,\infty)$.
Define the hyperplane
	$$\hyper=\{\phi\in\C(\R):\phi(-1)=0\}$$
and note that by the time-translations performed at the beginning of Section \ref{sec:normalizedSOPS}, $\sops_0^\beta\in\hyper$ for all $\beta>\buniq$.
The following lemma, which is a consequence of the implicit function theorem (see, e.g., \cite[Chapter 2, Theorem 2.3]{Chow1982}), is a version of \cite[Lemma 1]{Xie1992} stated for our setting.

\begin{lem}\label{lem:qbeta}
	For each $\beta>\buniq$ there exists a neighborhood $W^\beta$ of $(\beta,\sops_0^\beta)$ in $(\buniq,\infty)\times\hyper$ and a unique function $q^\beta:W^\beta\to(1,\infty)$ such that $q^\beta(\beta,\sops_0^\beta)=\period^\beta-1$ and $T(\tilde\beta,\phi,q^{\beta}({\tilde\beta},\phi))=0$ for all $(\tilde\beta,\phi)\in W^\beta$.
	Furthermore, $q^\beta$ is continuously differentiable on $W^\beta$.
\end{lem}

\begin{proof}
	Let $\beta>\buniq$.
	By the periodicity of $\sops^\beta$ and the fact that $\sops_0^\beta\in\hyper$, we see that $T(\beta,\sops_0^\beta,\period^\beta-1)=\sops^\beta(-1)=0$ and $D_tT(\beta,\sops_0^\beta,\period^\beta-1)=\dot\sops^\beta(-1)>0$.
	The lemma then follows immediately from the implicit function theorem and the fact that $\period^\beta>2$.
\end{proof}

Recall that for $\phi\in\C(\R)$ we let $x(\phi)$ denote the solution of the DDE starting at $\phi$.
For $t\ge-1$, we let $x(t;\phi)$ denote evaluation of $x(\phi)$ at time $t$.
Given $\beta>\buniq$ define the shift map $\Phi^\beta:W^\beta\to\hyper$ by
	\be\label{eq:Phibeta}\Phi^\beta(\tilde\beta,\phi)=S({\tilde\beta},\phi,q^\beta({\tilde\beta},\phi)+1)=x^{\tilde\beta}_{q^\beta({\tilde\beta},\phi)+1}(\phi),\qquad({\tilde\beta},\phi)\in W^\beta.\ee
Observe that $\Phi^\beta(\beta,\sops_0^\beta)=\sops_0^\beta$ and due to the regularity properties of $S$ and $q^\beta$, it follows that $\Phi^\beta$ is continuously differentiable on $W^\beta$.
The following lemma is a version of \cite[Theorem 9]{Xie1992} stated for our setting.
	
\begin{lem}\label{lem:DPhi}
	Let $\beta>\buniq$.
	Then $D_\phi\Phi^\beta(\beta,\sops_0^\beta)\in B_0(\C(\R))$ and $1\not\in\sigma( D_\phi\Phi^\beta(\beta,\sops_0^\beta))$.
\end{lem}

Lemma \ref{lem:DPhi} follows from the main functional analytic result in \cite{Xie1992}.
In Appendix \ref{apdx:meansfieldlemma} we state the main result from \cite{Xie1992} and prove Lemma \ref{lem:DPhi}.
We can now show that the family of SOPS are continuously differentiable in $\beta$.

\begin{lem}\label{lem:sopsdiffbeta}
	For each $t\in\R$ the function $\beta\mapsto\sops_t^\beta$ from $(\buniq,\infty)$ to $\C(\R)$ is continuously differentiable.
\end{lem}

\begin{proof}
	Let $\beta>\buniq$.
	Let $W^\beta$ and $q^\beta$ be as in Lemma \ref{lem:qbeta}, and define $\Phi^\beta:W^\beta\to\hyper$ as in \eqref{eq:Phibeta}.
	Define $\Psi^\beta:W^\beta\to\C(\R)$ by
		$$\Psi^\beta(\tilde\beta,\phi)=\Phi^\beta(\tilde\beta,\phi)-\phi,\qquad(\tilde\beta,\phi)\in W^\beta.$$
	Since $(\beta,\sops_0^\beta)$ is a fixed point of $\Phi^\beta$, we have $\Psi^\beta(\beta,\sops_0^\beta)=0$.
	By Lemma \ref{lem:DPhi}, $\Psi^\beta$ is continuously differentiable and its Frech\'et derivative with respect to $\phi$ satisfies $D_\phi\Psi^\beta(\beta,\sops_0^\beta)=D_\phi \Phi^\beta(\beta,\sops_0^\beta)-I_{\C(\R)}$.
	Since $1\not\in\sigma(D_\phi \Phi^\beta(\beta,\sops_0^\beta))$ by Lemma \ref{lem:DPhi}, it follows that $\Null D_\phi\Psi^\beta(\beta,\sops_0^\beta)=\{0\}$.
	Thus, by the implicit function theorem, there is a neighborhood $A^\beta$ of $\beta$ in $(\buniq,\infty)$ and a continuously differentiable function $P^\beta:A^\beta\to\hyper$ such that $P^\beta(\beta)=\sops_0^\beta$ and $\Psi^\beta(\tilde\beta,P^\beta(\tilde\beta))=0$ for all $\tilde\beta\in A^\beta$.
	Thus $\Phi^\beta(\tilde\beta,P^\beta(\tilde\beta))=P^\beta(\tilde\beta)$ for all $\tilde\beta\in A^\beta$.
	It follows that, for each $t>1$, the function $\tilde\beta\mapsto S(\tilde\beta,P^\beta(\tilde\beta),t)=x_t^{\tilde\beta}(P^\beta(\tilde\beta))$ is continuously differentiable on $A^\beta$ and, for each $\tilde\beta\in A^\beta$, $x^{\tilde\beta}(P^\beta(\tilde\beta))$ is a periodic solution of the scalar DDE \eqref{eq:dde}.
	We now argue that there is a possibly smaller neighborhood $\wt A^\beta$ of $\beta$ such that, for each $\tilde\beta\in \wt A^\beta$, $x^{\tilde\beta}(P^\beta(\tilde\beta))$ is an SOPS.
	To see this, recall that $z_1^\beta>0$ and $\sops^\beta(s)>0$ for all $s\in(-1,z_1^\beta)$.
	Let $0<\ve<\min\{z_1^\beta,1\}$ so that $S(\beta,P^\beta(\beta),\ve)(\theta)=\sops_\ve^\beta(\theta)>0$ for all $\theta\in[-1,0]$.
	Since $\tilde\beta\mapsto S(\tilde\beta,P^\beta(\tilde\beta),\ve)=x_\ve^{\tilde\beta}(P^\beta(\tilde\beta))$ is continuous on $A^\beta$, there is a neighborhood $\wt A^\beta\subset A^\beta$ of $\beta$ such that $x_\ve^{\tilde\beta}(P^\beta(\tilde\beta))(\theta)>0$ for all $\theta\in[-1,0]$.
	By a standard argument, this implies that $x^{\tilde{\beta}}(P^\beta(\tilde{\beta}))$ is slowly oscillating (see, e.g., \cite[Chapter XV, Lemma 3.2]{Diekmann1991}), and therefore an SOPS.
	In particular, $x^{\tilde\beta}(P^\beta(\tilde\beta))$ is an SOPS that satisfies $x_0^{\tilde\beta}(P^\beta(\tilde\beta))\in\hyper$ and $x_\ve^{\tilde\beta}(s;P^\beta(\tilde\beta))>0$ for all $s\in[\ve-1,\ve]$. 
	By the uniqueness of SOPS stated in Theorem \ref{thm:xie} and the time translations performed at the beginning of Section \ref{sec:normalizedSOPS}, it must hold that $P^\beta(\tilde\beta)=\sops_0^{\tilde\beta}$ and $x^{\tilde\beta}(P^\beta(\tilde\beta))=\sops^{\tilde\beta}$ for all $\tilde\beta\in\wt A^\beta$.
	Since $\beta>\buniq$ was arbitrary, this proves that the function $\beta\mapsto\sops_0^\beta$ from $(\buniq,\infty)$ to $\C(\R)$ is continuously differentiable.
	Therefore, $\beta\mapsto S(\beta,\sops_0^\beta,t)$ is continuously differentiable on $(\buniq,\infty)$, for each $t>1$.
	The periodicity of $\sops^\beta$ then implies this holds for all $t\in\R$.
\end{proof}

\begin{prop}
	\label{lem:mono_continuity}
	Given $s\in\R$ the function $(\beta,\lambda)\mapsto\monoU_\lambda^\beta(s)$ from $(\buniq,\infty)\times\CC$ to $B_0(\C(\CC))$ is continuous in $\beta$ and holomorphic in $\lambda$.
\end{prop}

\begin{proof}
	Fix $s\in\R$ and $t\ge s$.
	Let $m=\min\{j\in\N:j\ge t-s\}$.
	Set $s_j=s+j$ for $j=0,\dots,m-1$ and set $s_m=t$.
	Let $\vf\in\C(\CC)$ and write $y^\beta(\lambda)$ in place of $y^\beta(\lambda,s,\vf)$.
	For $t\ge s-1$ we let $y^\beta(t;\lambda)$ denote evaluation of $y^\beta(\lambda)$ at time $t$.
	We show that for each $j=0,\dots,m$, the following hold:
	\begin{itemize}
		\item[(i)] $(\beta,\lambda)\mapsto y_{s_j}^\beta(\lambda)$ is continuous on $(\buniq,\infty)\times\CC$; and
		\item[(ii)] For each $(\beta,\lambda)\in(\buniq,\infty)\times\CC$, $\pd y^\beta(r;\lambda)/\pd\lambda$ exists for all $r\in[s-1,s_j]$, $\pd y^\beta(\cdot;\lambda)/\pd\lambda$ is continuous on $[s-1,s_j]$ and
			\be\label{eq:limhdydlambda}\lim_{z\to0}\sup_{r\in[s-1,s_j]}\abs{\frac{\partial y^\beta(r;\lambda)}{\partial\lambda}-\frac{y^\beta(r;\lambda+z)-y^\beta(r;\lambda)}{z}}=0.\ee
	\end{itemize}
	Since $t=s_m$, this will prove that $(\beta,\lambda)\mapsto y_t^\beta(\lambda)$ is continuous and holomorphic in $\lambda$.
	The fact that $(\beta,\lambda)\mapsto\monoU_\lambda^\beta(s)$ is continuously differentiable then follows by taking $t=s+\period^\beta$ and using the linearity of the operator $\monoU_\lambda^\beta(s)$.
	
	Fix $(\beta,\lambda)\in(\buniq,\infty)\times\CC$.
	The proof proceeds by induction.
	The base case $j=0$ follows from the fact that $y_{s_0}^\beta(\lambda)=\vf$ for all $(\beta,\lambda)\in(\buniq,\infty)\times\CC$.
	Now suppose that (i) and (ii) hold for some $j\in\{0,\dots,m-1\}$.
	We first prove (i) holds with $j+1$ in place of $j$.
	By equation \eqref{eq:ylambdasolution} for $y^\beta$ and the induction hypothesis, for $(\beta,\lambda),(\tilde\beta,\tilde\lambda)\in(\buniq,\infty)\times\CC$,
	\begin{align}
		\sup_{r\in[s-1,s_{j+1}]}|{y^{\beta}(r;\lambda)-y^{\tilde\beta}(r;\tilde\lambda)}|&\le|{\lambda\beta-\tilde\lambda\tilde\beta}|\int_{s-1}^{s_j}|f'(\sops^{\beta}(u))y^{\beta}(u;\lambda)|du \notag\\
		&\qquad+|{\tilde\lambda\tilde\beta}|\int_{s-1}^{s_j}|{f'(\sops^{\beta}(u))y^{\beta}(u;\lambda)-f'(\sops^{\tilde\beta}(u))y^{\tilde\beta}(u;\tilde\lambda)}|du \label{eq:limhdydlambda1}.
	\end{align}
	It follows from the bounded convergence theorem, the induction hypothesis and Lemma \ref{lem:sopsdiffbeta} that the right hand side of the inequality \eqref{eq:limhdydlambda1} converges to zero as $(\tilde\beta,\tilde\lambda)\to(\beta,\lambda)$.
	Thus, $(\lambda,\beta)\mapsto y_{s_j}^\beta(\lambda)$ is continuous.
	This completes the proof of (i).
	
	Next we prove (ii) holds with $j+1$ in place of $j$.
		By equation \eqref{eq:ylambdasolution} for $y^\beta$, the induction hypothesis and the dominated convergence theorem, we have, for $r\in[s_j,s_{j+1}]$,
	\begin{align}
		\frac{\pd y^\beta(r;\lambda)}{\pd\lambda}&=\lim_{z\to0}\frac{y^\beta(r;\lambda+z)-y^\beta(r;\lambda)}{z} \notag\\
		&=\beta\lim_{z\to0}\int_{s-1}^{r-1}e^{-\alpha(r-u-1)}f'(\sops^\beta(u))y^\beta(u;\lambda+z)du \notag\\
		&\qquad+\lambda\beta\lim_{z\to0}\int_{s-1}^{r-1}e^{-\alpha(r-u-1)}f'(\sops^\beta(u))\frac{y^\beta(u;\lambda+z)-y^\beta(u;\lambda)}{z}du \notag\\
		&=\beta\int_{s-1}^{r-1}e^{-\alpha(r-u-1)}f'(\sops^\beta(u))y^\beta(u;\lambda)du \notag\\
		&\qquad+\lambda\beta\int_{s-1}^{r-1}e^{-\alpha(r-u-1)}f'(\sops^\beta(u))\frac{\pd y^\beta(u;\lambda)}{\pd\lambda}du \label{eq:eq:limhdydlambda2}.
	\end{align}
	It follows that $\pd y^\beta(r;\lambda)/\pd\lambda$ exists for all $r\in[-1,s_{j+1}]$ and $\pd y^\beta(\cdot;\lambda)/\pd\lambda$ is continuous on $[-1,s_{j+1}]$.
	In addition, by equation \eqref{eq:ylambdasolution} for $y^\beta$, equation \eqref{eq:eq:limhdydlambda2}, and the induction hypothesis, for nonzero $z\in\CC$,
	\begin{align*}\label{eq:suprDeltaypdylambda}
		&\sup_{r\in[s-1,s_{j+1}]}\abs{\frac{y^\beta(r;\lambda+z)-y^\beta(r;\lambda)}{z}-\frac{\pd y^\beta(r;\lambda)}{\pd\beta}}\\
		\notag&\qquad\le\abs{\lambda}\sup_{r\in[s-1,s_j]}\int_{s-1}^{s_j}\abs{f'(\sops^{\beta}(u))y^\beta(u;\lambda+z)-f'(\sops^\beta(u))y^\beta(u;\lambda)}du\\
		\notag&\qquad\qquad+\abs{\lambda}\beta\sup_{r\in[s-1,s_j]}\abs{\frac{y^{\beta}(r;\lambda+z)-y^\beta(r;\lambda)}{z}-\frac{\pd y(r;\lambda)}{\pd\lambda}}\int_{s-1}^{s_j}\abs{f'(\sops^\beta(u))}du.
	\end{align*}
	By the induction hypothesis, taking limits as $z\to0$ in the last inequality yields the limit in \eqref{eq:limhdydlambda} with $s_{j+1}$ in place of $s_j$.
	This completes the proof of (ii).
	This proves the induction step.
	Thus, by the principle of mathematical induction, (i) and (ii) hold with $j=m$.
\end{proof}

To complete the proof of the continuity of $\nu_\beta$ at $\beta=\infty$, it is useful work in a neighborhood of 0 as opposed to infinity. To this end set $\zeta_0=\buniq^{-1}$ and, for $s\in(-q_1^\star,0)$, define the function $V_s:(-\zeta_0,\zeta_0)\times\CC\to B_0(\C(\CC))$, for $(\zeta,\lambda)\in(-\zeta_0,\zeta_0)\times\CC$, by
\be\label{eq:V}
	V_s(\zeta,\lambda)=
	\begin{cases}
		\monoU_\lambda^{|\zeta|^{-1}}(s),&\text{if }\zeta\neq0,\\
		\monoU_\lambda^\star(s),&\text{if }\zeta=0.
	\end{cases}
\ee
We have the following corollary.

\begin{cor}\label{cor:mono_Continuity}
	Let $s\in(-q_1^\star,0)$.
	The function $V_s(\cdot,\cdot)$ is continuous on $(-\zeta_0,\zeta_0)\times\CC$ and for each $\zeta\in(-\zeta_0,\zeta_0)$, the function $V_s(\zeta,\cdot)$ is holomorphic on $\CC$.
	Finally, if $\lambda\not\in\{\vf_1,\vf_2\}$ then $\nu_\star(\lambda)$ is the unique nonzero eigenvalue of $V_s(0,\lambda)$ and $\nu_\star(\lambda)$ is a simple eigenvalue of $V_s(0,\lambda)$.
\end{cor}

\begin{proof}
	The continuity of $V_s$ follows from Propositions \ref{lem:mono_continuity} and \ref{prop:mono_convergence}, and the definition of $\monoU_\lambda^\star(s)$ in \eqref{eq:Uastexplicit}.
	The fact that, for each $\zeta\in(-\zeta_0,\zeta_0)$, the function $V_s(\zeta,\cdot)$ is holomorphic on $\CC$ follows from Proposition \ref{lem:mono_continuity} and the definition of $\monoU_\lambda^\star(s)$.
	The final line follows from Lemma \ref{lem:Uspectrum}.
\end{proof}

\subsection{Useful functional analysis results}
\label{sec:useful}

In preparation for proving Proposition \ref{prop:nu_lambda_simple}, we state some useful results.

\begin{theorem}[{\cite[Chapter 14, Corollary 3.2]{Chow1982}}]
	\label{thm:mu_simple} 
	Let $\banach$ and $\banachY$ be Banach spaces, $O\subset\banach$ be an open set, $V:O\to B(\banachY)$ be a continuous function and $\chi_0\in O$.
	Suppose $\mu_0\in\CC$ is a simple eigenvalue of $V(\chi_0)$ with associated unit eigenfunction $\psi_0\in\banachY$.
	Then there is a $\delta\in(0,1)$ and continuous functions $\mu:\ball_\banach(\chi_0,\delta)\to\ball(\mu_0,\delta)$ and $\psi:\ball_\banach(\chi_0,\delta)\to\banachY$ such that $\mu(\chi_0)=\mu_0$, $\psi(\chi_0)=\psi_0$ and for each $\chi\in\ball_\banach(\chi_0,\delta)$, $\mu(\chi)$ is the unique simple eigenvalue of $V(\chi)$ and $\psi(\chi)$ is a unit eigenfunction of $V(\chi)$ associated with $\mu(\chi)$.	
	In addition, if $V$ is holomorphic, then $\mu$ and $\psi$ are also holomorphic.
\end{theorem}

We will need the following corollary, whose proof is deferred to Appendix \ref{apdx:meansfieldlemma}.
For a Banach space $\banachY$ and an operator $A\in B_0(\banachY)$, recall from Section \ref{sec:notation_linear_operators} that $\banachY^\ast=B_0(\banachY,\CC)$ denotes the dual space of $\banachY$, $A^\ast\in B_0(\banachY^\ast)$ denotes the adjoint of $A$, and $\sigma(A)=\sigma(A^\ast)$.

\begin{cor}\label{cor:mu_eigen} 
    Let $\banach$ and $\banachY$ be Banach spaces, $O\subset\banach$ be an open set, $V:O\to B_0(\banachY)$ be a continuous function and $\chi_0\in O$.
    Suppose $\mu_0\in\CC\setminus\{0\}$ is a simple eigenvalue of $V(\chi_0)$ and $\psi_0\in\banachY$ is an associated unit eigenfunction.
    Let $\delta\in(0,1)$, $\mu:\ball_\banach(\chi_0,\delta)\to\ball(\mu_0,\delta)$ and $\psi:\ball_\banach(\chi_0,\delta)\to\banachY$ be as in Theorem \ref{thm:mu_simple}.
    Suppose $\psi^\ast_0\in\banachY^\ast$ is a unit eigenfunction of $V^\ast(\chi_0)$ associated with $\mu_0$ such that $\psi^\ast_0(\psi_0)\neq0$.
    Then there is a $\delta_0\in(0,\delta]$ such that for all $\chi\in\ball_\banach(\chi_0,\delta_0)$, $\mu(\chi)$ is the unique eigenvalue of $V(\chi)$ in $\ball(\mu_0,\delta_0)$.
\end{cor}

For $s\in(-q_1^\star,0)$ and $\lambda\in\CC$ let $(\monoU_\lambda^\star(s)|_{\C(\CC)})^\ast:B(\C(\CC),\CC)\mapsto B(\C(\CC),\CC)$ denote the adjoint of $\monoU_\lambda^\star(s)|_{\C(\CC)}$.
Recall the definition for $F(s,\lambda,z_1,z_2)$ in \eqref{eq:Fslambdaz1z2}.

\begin{lem}\label{lem:psiast}
	Let $s\in(-q_1^\star,0)$.
	The spectrum of $(\monoU_\lambda^\star(s)|_{\C(\CC)})^\ast$ is equal to $\{0,\nu_\star(\lambda)\}$ and if $\lambda\not\in\{\vr_1,\vr_2\}$, then $\nu_\star(\lambda)$ is the unique simple eigenvalue of $(\monoU_\lambda^\star(s)|_{\C(\CC)})^\ast$ and has corresponding eigenfunction $\vf_\lambda^\ast\in B(\C(\CC),\CC)$ given by 
		$$\vf_\lambda^\ast(\vf)=F(s,\lambda,\vf(0),\vf(-1-s))e^\alpha,\qquad\vf\in\C(\CC),$$
	which satisfies $\vf_\lambda^\ast(\vf_0)=\nu_\star(\lambda)$, where $\vf_0$ is defined in \eqref{eq:psi0}.
\end{lem}

\begin{proof}
	Let $\lambda\not\in\{\vr_1,\vr_2\}$.
	By the respective definitions of $F(s,\lambda,z_1,z_2)$ and $\psi_0$ in \eqref{eq:Fslambdaz1z2} and \eqref{eq:psi0},
	\begin{align}
	   \vf_\lambda^\ast(\vf_0)&=F(s,\lambda,\vf_0(0),\vf_0(-1-s))e^\alpha \notag\\
	   &=-\lsb 1-\lambda e^\alpha(1+ab^{-1})\rsb\frac{\lambda-\vr_1}{1-\vr_2}e^{-\alpha(q_2^\star+1)} \notag\\
	   &=\nu_\star(\lambda)\ne0. \label{eq:psilambdapsi0}
	\end{align}
	Suppose $\vf\in\C(\CC)$.
	Then by equation \eqref{eq:Uastexplicit} for $\monoU_\lambda^\star(s)\psi$ and equation \eqref{eq:psilambdapsi0},
	\begin{align*}
	\lsb(\monoU_\lambda^\star(s))^\ast-\nu_\star(\lambda)\Id_{B(\C(\CC),\CC)}\rsb\circ\vf_\lambda^\ast(\vf) &= \vf_\lambda^\ast\circ\lsb\monoU_\lambda^\star(s)-\nu_\star(\lambda)\Id_{B(\C(\CC),\CC)}\rsb(\vf) \\
	&= F(s,\lambda,\vf(0),\vf(-1-s))e^\alpha\vf_\lambda^\ast(\vf_0)-\nu_\star(\lambda)\vf_\lambda^\ast(\vf)\\
	&= 0
	\end{align*}
	It follows that $\nu_\star(\lambda)$ is an eigenvalue of $(\monoU_\lambda^\star(s))^\ast$ with associated eigenfunction $\vf_\lambda^\ast$.
\end{proof}
    
\subsection{Proof of Proposition \ref{prop:nu_lambda_simple}}
\label{sec:boundary}

\begin{proof}[Proof of Proposition \ref{prop:nu_lambda_simple}]
	Let $\delta\in(0,1)$ and $K\subset\CC\setminus\{\vr_1,\vr_2\}$ be compact.
    Since $\nu_\star$ is continuous and nonzero on $K$, we have 
    	$$\ve_0=\frac{1}{2}\inf\{\abs{\nu_\star(\lambda)}:\lambda\in K\}>0.$$
	Set $\zeta_0=\buniq^{-1}$ and define $V_s$ as in \eqref{eq:V}.	
	By Corollary \ref{cor:mono_Continuity}, $V_s(\cdot,\cdot)$ is continuous on $(-\zeta_0,\zeta_0)\times\CC$ and $\nu_\star(\lambda)$ is a simple eigenvalue of $V_s(0,\lambda)$ for each $\lambda\in\CC$.
	Let $\lambda_1\in K$ and define $\vf_{\lambda_1}^\ast\in B_0(\C(\CC),\CC)$ as in Lemma \ref{lem:psiast} so that $\vf_{\lambda_1}^\ast$ is an eigenfunction of $(\monoU_{\lambda_1}^\star(s))^\ast$ associated with $\nu_\star(\lambda_1)$ that satisfies $\vf_{\lambda_1}^\ast(\vf_0)\ne0$.
	By renormalizing, we can assume that $\vf_{\lambda_1}^\ast$ is a unit eigenfunction associated with $\nu_\star(\lambda_1)$ that satisfies $\vf_{\lambda_1}^\ast(\vf_0)\ne0$.
	We now apply Theorem \ref{thm:mu_simple} and Corollary \ref{cor:mu_eigen} with $\banach=\R\times\CC$ and norm $\norm{(\xi,z)}_\banach=\max\{\abs{\xi},\abs{z}\}$ for $(\xi,z)\in\banach$, $O=(-\zeta_0,\zeta_0)\times\CC$, $\banachY=\C(\CC)$, $\banachY^\ast=B_0(\C(\CC),\CC)$, $V=V_s$, and $\chi_0=(0,\lambda_1)$.
	Together they imply there is an $0<\ve_1<\min(\zeta_0,\ve_0)$ such that for each $\zeta\in(-\ve_1,\ve_1)$ and $\lambda\in\ball(\lambda_1,\ve_1)$, there is a unique eigenvalue $\mu_1(\zeta,\lambda)$ of $V_s(\zeta,\lambda)$ in $\ball(\nu_\star(\lambda_1),\ve_1)$.
    Furthermore, $\mu_1(\cdot,\cdot)$ is continuous on $(-\ve_1,\ve_1)\times\ball(\lambda_1,\ve_1)$.
    Since $K$ is compact, there are finitely many $\lambda_1,\dots,\lambda_m\in K$ and $0<\ve_1,\dots,\ve_m<\min\{\zeta_0,\ve_0\}$ such that $K\subset\ball(\lambda_1,\ve_1)\cup\cdots\cup\ball(\lambda_m,\ve_m)$ and for each $k=1,\dots,m$, there exists $\mu_k:(-\ve_k,\ve_k)\times\ball(\lambda_k,\ve_k)\to\CC$ satisfying the aforementioned properties (with $k$ in place of $1$).    
    Let $\ve=\min(\delta,\ve_1,\dots,\ve_m)$ and define the continuous function $\mu:(-\ve,\ve)\times K\to\CC$ by
    	$$\mu(\zeta,\lambda)=\mu_k(\zeta,\lambda)\qquad\text{if }\lambda\in\ball(\lambda_k,\ve_k),\qquad k=1,\dots,m.$$
    To see that $\mu(\cdot,\cdot)$ is well defined on $(-\ve,\ve)\times K$, suppose $\eta\in(-\ve,\ve)$ and $\lambda\in\ball(\lambda_j,\ve_j)\cap\ball(\lambda_k,\ve_k)$ for some $j\neq k$.
    Then $\mu_j(\eta,\lambda)$ and $\mu_k(\eta,\lambda)$ are both equal to the unique eigenvalue of $V_s(\eta,\lambda)$ in $\ball(\nu_\star(\lambda),\ve_j)\cap\ball(\nu_\star(\lambda),\ve_k)$.
    Thus, $\mu$ is well defined and continuous on $(-\ve,\ve)\times K$, and for each $(\zeta,\lambda)\in(-\ve,\ve)\times K$, $\mu(\zeta,\lambda)$ is the unique eigenvalue of $V_s(\zeta,\lambda)$ in $\ball(\nu_\star(\lambda),\ve)$.
    Next we show that $\mu$ is holomorphic in $\lambda$.
    Let $\zeta\in(-\ve,\ve)$.
    By Corollary \ref{cor:mono_Continuity}, $V_s(\zeta,\cdot)$ is holomorphic on $\CC$.
    Let $\lambda\in \interior(K)$.
    By a second application of Theorem \ref{thm:mu_simple}, this time with $O=\banach=\CC$, $\banachY=B_0(\C(\CC))$, $V(\cdot)=V_s(\zeta,\cdot)$ and $\chi_0=\lambda$, there exists $\wt\ve>0$ such that $\mu(\zeta,\cdot)$ is holomorphic on $\ball(\lambda,\wt\ve)$.
    Since this holds for all $\lambda\in \interior(K)$, it follows that $\mu(\zeta,\cdot)$ is holomorphic on all of $\interior(K)$.
    
    We can now define the functions $\nu_\beta:K\to\CC$.
    For each $\beta>\ve^{-1}$ and $\lambda\in K$, define $\nu_\beta(\lambda)=\mu(\beta^{-1},\lambda)$.
    Then $\nu_\beta(\cdot)$ is holomorphic on $\interior(K)$ and, for each $\lambda\in K$, $\nu_\beta(\lambda)$ is the unique eigenvalue of $\monoU_\lambda^\beta(s)=V_s(\beta^{-1},\lambda)$ in $\ball(\nu_\star(\lambda),\ve)$.
    By the definition of $\nu_\beta$ and the regularity of $\mu$,
    \begin{align*}
    	\lim_{\beta\to\infty}\sup_{\lambda\in K}\abs{\nu_\beta(\lambda)-\nu_\star(\lambda)}&=\lim_{\zeta\to0}\sup_{\lambda\in K}|\mu(\zeta,\lambda)-\mu(0,\lambda)|=0.
\end{align*}
Since $\nu_\beta(\cdot)$ is holomorphic and $\nu_\beta$ converges to $\nu_\star$ uniformly on compacts, as $\beta\to\infty$, it follows that the derivative $\pd\nu_\beta$ converges to $\pd\nu_\star$ uniformly on compact sets, as $\beta\to\infty$ (see, e.g., \cite[Theorem 1.2]{Lang1999}); that is,
\begin{align*}
    	\lim_{\beta\to\infty}\sup_{\lambda\in K}\abs{\frac{\pd\nu_\beta(\lambda)}{\pd\lambda}-\frac{\pd\nu_\star(\lambda)}{\pd\lambda}}&=\lim_{\zeta\to0}\sup_{\lambda\in K}\abs{\frac{\pd\mu(\zeta,\lambda)}{\pd\lambda}-\frac{\pd\mu(0,\lambda)}{\pd\lambda}}=0.
    \end{align*}
    Then along with Theorem \ref{thm:1a} this implies there exists $\beta'\ge\buniq$ such that for all $\lambda\in K$, the limits in \eqref{eq:mubetanu} hold and $d_H(\sigma(\monoU_\lambda^\beta(s)),\{0,\nu_\star(\lambda)\})<\ve$ for all $\beta>\beta'$.
    Since $\nu_\beta(\lambda)$ is the unique eigenvalue of $\monoU_\lambda^\beta(s)$ in $\ball(\nu_\star(\lambda),\ve)$ and $\ve\le\frac{1}{2}\abs{\nu_\star(\lambda)}$, it follows that \eqref{eq:rhoUbetanubeta} holds.

    Suppose $\beta>\beta'$.
    Let $r\in K\cap\R$.
    Note that solutions of the extended variational equation \eqref{eq:eve} with real-valued initial conditions are real-valued.
    Thus, $\monoU_r^\beta(s)|_{\C(\R)}\in B_0(\C(\R))$ and so the elements of $\sigma(\monoU_r^\beta(s)|_{\C(\R)})$ are real or in conjugate pairs.
    Since $\nu_\star(r)\in\R$, $\nu_\beta(r)$ is the unique simple eigenvalue of $\monoU_r^\beta(s)$ in $\ball(\nu_\star(r),\ve)$ and the non-real-valued eigenvalues of $\monoU_r^\beta(s)$ are in conjugate pairs, it follows that $\nu_\beta(r)$ must be real-valued.
    Using the facts that $\nu_\beta(\cdot)$ maps $K\cap\R$ to $\R$ and is holomorphic, we conclude that $\pd\nu_\beta(r)/\pd\lambda$ is real-valued for all $r\in K\cap\R$.
	To prove the final statement, suppose $1\in K$.
	It follows from Theorem \ref{thm:xie} that 1 is a simple eigenvalue of $\monoU_1^\beta(0)$. 
	Since $\nu_\star(1)=1$ and $\nu_\beta(1)$ is the unique eigenvalue of $\monoU_1^\beta(0)$ in $\ball(\nu_\star(1),\ve)$, we see that $\nu_\beta(1)=1$.
	Now suppose $0\in K$.
	Note that when $\alpha=\lambda=0$, the extended variational equation \eqref{eq:eve} reduces to $\dot{y}(t)=0$.
	It follows that $\sigma(\monoU_0^\beta(0))=\{0,1\}$ and 1 is a simple eigenvalue of $\monoU_0^\beta(0)$ with corresponding eigenspace equal to the span of $\tilde\vf$,
    where $\tilde\vf(\cdot)\equiv1$ is the constant function identically equal to 1.
	Hence, $\nu_\beta(0)=1$. 
\end{proof}

\section{Limits of the extended characteristic multipliers for $\lambda$ near 1, 0 and in $[0,1)$}
\label{sec:description}

We now prove our remaining results on the characterization of the extended characteristic multipliers when $\beta$ is large.

\subsection{Proof of Theorem \ref{thm:2a}}
\label{sec:proofsnbhd}

The next lemma, which follows from Laurent's Theorem for holomorphic functions, will be useful in the proof of Theorem \ref{thm:2a}, as well as in the proof of Theorem \ref{thm:nbhdof0} in the next section, Section \ref{sec:proofnbhd0}.

\begin{lem}\label{lem:taylor}
	Suppose $z\in\CC$, $r>0$ and $g$ is holomorphic on an open set that contains $\ball(z,r)$. 
	Then
	$$\abs{g(\tilde{z})-g(z)-g'(z)(\tilde{z}-z)}\leq \frac{N\abs{\tilde{z}-z}^2}{r(r-|\tilde{z}-z|)},\qquad\tilde{z}\in\ball(z,r),$$
	where $N=\sup\{\abs{g(\tilde{z})}:\tilde{z}\in \pd \ball(z,r)\}$.
\end{lem}

\begin{proof}
	By Laurent's theorem (see, e.g., \cite[Chapter III, Theorem 7.3]{Lang1999}),
	\be\label{eq:taylor}g(\tilde{z})=g(z)+g'(z)(\tilde{z}-z)+\sum_{k=2}^\infty\frac{g^{(k)}(z)}{k!}(\tilde{z}-z)^k,\qquad\tilde{z}\in\ball(z,r),\ee
	and
		\be\label{eq:taylor1}\abs{\sum_{k=2}^\infty\frac{g^{(k)}(z)}{k!}(\tilde{z}-z)^k}\le N\sum_{k=2}^\infty\frac{|\tilde{z}-z|^k}{r^k}=\frac{N|\tilde{z}-z|^2}{r(r-|\tilde{z}-z|)}.\ee
	The lemma then follows after rearranging equation \eqref{eq:taylor} and using inequality \eqref{eq:taylor1}.
\end{proof}

\begin{proof}[Proof of Theorem \ref{thm:2a}] 
    By the definition of $\nu_\star$ in \eqref{eq:nuast}, we see that $\nu_\star(1)=1$ and 
		\be\label{eq:nuprime1}\nu_\star'(1)=\frac{1}{1-\vr_1}+\frac{1}{1-\vr_2}\ge2.\ee
	By definitions of $\vr_1,\vr_2$ in \eqref{eq:rho} and our assumption that $a\ge b>0$, we have $0<\vr_2\le\vr_1<1$.
    Let $\ve=(1-\vr_1)/2>0$ and define the compact set $K=\clo(\ball(1,\ve))\subset\CC\setminus\{\vr_1,\vr_2\}$ so that, by the definition of $\nu_\star(\lambda)$,
    	$$m=\inf\{\abs{\nu_\star(\lambda)}:\lambda\in K\}>0.$$
    Define the positive constants
    	\be\label{eq:consts}d=\frac{\nu_\star'(1)}{2},\qquad\delta=\min(m,d),\qquad N=\sup\{\abs{\nu_\star(\lambda)}:\lambda\in K\}+\delta,\qquad c=\frac{\ve^2}{4N}.\ee 
    By Proposition \ref{prop:nu_lambda_simple}, there exists $\bnbhdone=\bnbhdone(\alpha,f)\ge\buniq$ such that for each $\beta>\bnbhdone$, there exists a function $\nu_\beta:K\to\CC$, holomorphic on $\interior(K)$, such that $\nu_\beta(1)=1$ and the following hold:
    	\be\label{eq:rhoUlambdanulambdabeta}\rho(\monoU_\lambda^\beta(s))=|\nu_\beta(\lambda)|,\qquad\lambda\in K,\ee
    and
    	\be\label{eq:nubetaM}\sup\lcb\abs{\nu_\beta(\lambda)}:\lambda\in K\rcb\le N,\ee
    and
    	\be\label{eq:nubetaprimebound} d\le\nu_\beta'(1)\le 3d.\ee  
    
    {\bf Definitions of $A_{<1}$ and $A_{>1}$}:
    Define the continuous functions $g_1,g_2:\CC\to\R_+$ by
    \begin{align*}
    	g_1(\lambda)&=\abs{1+d(\Re\lambda-1)+i3d\Im\lambda}+\frac{|\lambda-1|^2}{c},\qquad\lambda\in\CC,\\
    	g_2(\lambda)&=\abs{1+d(\Re\lambda-1)+id\Im\lambda}-\frac{|\lambda-1|^2}{c},\qquad\lambda\in\CC.
    \end{align*}
    
    Let
    \begin{align}
        \label{eq:A<1}A_{<1}&=\lcb\lambda\in\CC_{<1}\cap \ball(1,\ve/2):g_1(\lambda)<1\rcb,\\
        \label{eq:A>1}A_{>1}&=\lcb\lambda\in\CC_{>1}\cap \ball(1,\ve/2):g_2(\lambda)>1\rcb.
    \end{align}
    Since $g_1$ and $g_2$ are continuous, $g_1(1)=g_2(1)=1$, $g_1(r)\in(0,1)$ for all $r\in(1-cd,1)$ and $g_2(r)\in(1,\infty)$ for all $r\in(1,1+cd)$, it follows that $A_{<1}$ and $A_{>1}$ are nonempty open sets with $1\in\partial A_{<1}\cap\partial A_{>1}$.
    
    {\bf Proof of parts (i) and (ii)}:
    Let $\lambda\in\CC_{<0}$ and let $r>0$ and $\theta\in(-\pi/2,\pi/2)$ be such that $\lambda=-re^{i\theta}$.
    Then for $\delta\in(0,1)$,
	    $$g_1(1+\delta\lambda)=\sqrt{1-2d\delta r\cos\theta+(d\delta r\cos\theta)^2+(3d\delta r\sin\theta)^2}+\frac{\delta^2r^2}{c},$$
    and
    \begin{align*}
	    \lim_{\delta\downarrow0}\frac{g_1(1+\delta\lambda)-g_1(1)}{\delta}&=-dr\cos\theta<0.
    \end{align*}
    Therefore, since $g_1(1)=1$, there exists $\delta_\lambda>0$ such that $1+\delta\lambda\in \ball(1,\ve/2)$ and $g_1(1+\delta(\lambda))<1$ for all $\delta\in(0,\delta_\lambda)$, so $1+\delta\lambda\in A_{<1}$.
    Thus, part (i) holds.
    Now let $\lambda\in\CC_{>0}$ and let $r>0$ and $\theta\in(-\pi/2,\pi/2)$ be such that $\lambda=re^{i\theta}$.
    An analogous argument (using $g_2$ instead of $g_1$) shows that there exists $\delta_\lambda>0$ such that $1+\delta\lambda\in A_{>1}$ for all $\delta\in(0,\delta_\lambda)$, so (ii) holds.
    To avoid repetition, we omit the details.
    
    {\bf Proofs of parts (iii) and (iv)}:
    Fix $\beta>\bnbhdone$.
    Recall that $\nu_\beta(1)=1$.
    Suppose $\lambda\in\ball(1,\ve/2)$ so that $\ve-|\lambda-1|>\ve/2$.
    By Lemma \ref{lem:taylor} with $z=1$ and $r=\ve$, the upper bound for $\nu_\beta$ in \eqref{eq:nubetaM} and definitions of $N$ and $c$ in \eqref{eq:consts}, we have
    \begin{equation}\label{eq:taylor1}
        \abs{\nu_\beta(\lambda)-1-\nu'_\beta(1)(\lambda-1)}<\frac{N \abs{\lambda-1}^2}{\ve(\ve-|\lambda-1|)}\le\frac{|\lambda-1|^2}{c}.
    \end{equation}
    Let $\lambda\in A_{<1}\subset \ball(1,\ve/2)$.
    Inequalities \eqref{eq:nubetaprimebound} and \eqref{eq:taylor1} and the definition of $A_{<1}$ imply
    \begin{align*}
        \abs{\nu_\beta(\lambda)}&\le\abs{1+\nu'_\beta(1)(\Re\lambda-1)+i\nu_\beta'(1)\Im\lambda}+\abs{\nu_\beta(\lambda)- 1-\nu'_\beta(1)(\lambda-1)} \\
        &\le\abs{1+d(\Re\lambda-1)+i3d\Im\lambda}+\frac{\abs{\lambda-1}^2}{c} \\
        &<1.
    \end{align*}    
    Along with equation \eqref{eq:rhoUlambdanulambdabeta} for $\rho(\monoU_\lambda^\beta(s))$, this proves that $\rho(\monoU_\lambda^\beta(s))<1$ for all $\lambda\in A_{<1}$, so part (iii) holds.
    Alternatively, let $\lambda\in A_{>1}$.
    An analogous argument, which we omit to avoid repetition, shows that $\abs{\nu_\beta(\lambda)}>1$ for all $\lambda\in A_{>1}$. 
    Along with equation \eqref{eq:rhoUlambdanulambdabeta} this proves that $\rho(\monoU_\lambda^\beta(s))>1$, so part (iv) holds.
\end{proof}

\subsection{Proof of Theorem \ref{thm:nbhdof0}}
\label{sec:proofnbhd0}

The proof of Theorem \ref{thm:nbhdof0} follows an exactly analogous structure to the proof of Theorem \ref{thm:2a}.
For completeness, we include the details of the proof. 

\begin{proof}[Proof of Theorem \ref{thm:nbhdof0}]
	By the definition of $\nu_\star$ in \eqref{eq:nuast}, the definitions of $\vr_1,\vr_2$ in \eqref{eq:rho} and the fact that $\alpha=0$, we have $\nu_\star(0)=1$, $\vr_1+\vr_2=1$ and
		$$\nu_\star'(0)=-\frac{1}{(1-\vr_1)(1-\vr_2)}\le-4.$$
	By our assumption $a\ge b>0$, we have $0<\vr_2\le\vr_1<1$.
	Let $\ve=\vr_2/2$ and define the compact set $K=\clo(\ball(0,\ve))\subset\CC\setminus\{\vr_1,\vr_2\}$ so that, by the definition of $\nu_\star$ in \eqref{eq:nuast},
		$$m=\inf\{\abs{\nu_\star(\lambda)}:\lambda\in K\}>0.$$
	Define the positive constants
		\be\label{eq:consts1}d=-\frac{\nu_\star'(0)}{2},\qquad \delta=\min(m,d),\qquad N=\sup\{\abs{\nu_\star(\lambda)}:\lambda\in K\}+\delta,\qquad c=\frac{\ve^2}{4N}.\ee
	By Proposition \ref{prop:nu_lambda_simple}, there $\bnbhdzero^\dag\ge\buniq$ such that for each $\beta>\bnbhdzero^\dag$ there exists a function $\nu_\beta:K\to\CC$, holomorphic on $\interior(K)$, such that $\nu_\beta(0)=1$ and the following hold:
		\be\label{eq:rhoUnubeta0}\rho(\monoU_\lambda^\beta(s))=|\nu_\beta(\lambda)|,\qquad \lambda\in K,\ee
		and
		\be\label{eq:supnubetaK0}\sup\{\abs{\nu_\beta(\lambda)}:\lambda\in K\}\le N,\ee
		and
		\be\label{eq:dnuprime0}-3d\le\nu_\beta'(1)\le-d.\ee
	Let
    $$g_1(\lambda)=\abs{1+d(\Re\lambda-1)+i3d\Im\lambda}+\frac{|\lambda-1|^2}{c},\qquad\lambda\in\CC$$
    	
	{\bf Definitions of $A_{>0}$ and $A_{<0}$}:
	Define the continuous functions $g_1,g_2:\CC\to\R_+$ by
	\begin{align*}
		g_1(\lambda)=\abs{1-d\Re\lambda+i3d\Im\lambda}+\frac{|\lambda|^2}{c},\qquad \lambda\in\CC,\\
		g_2(\lambda)=\abs{1-d\Re\lambda+id\Im\lambda}-\frac{|\lambda|^2}{c},\qquad \lambda\in\CC.
	\end{align*}
	Set
	\begin{align}
		\label{eq:A>0}A_{>0} &= \lcb \lambda\in \CC_{>0}\cap \ball(0,\ve/2):g_1(\lambda)<1 \rcb,\\
		\label{eq:A<0}A_{<0} &= \lcb \lambda\in \CC_{<0}\cap \ball(0,\ve/2):g_2(\lambda)>1 \rcb.
	\end{align}
	Since $g_1$ and $g_2$ are continuous, $g_1(0)=g_2(0)=1$, $g_1(r)\in(0,1)$ for all $r\in(0,cd)$ and $g_2(r)\in(1,\infty)$ for all  $r\in(-cd,0)$, it follows that $A_{>0}$ and $A_{<0}$ are nonempty open sets with $0\in\partial A_{>0}\cap\partial A_{<0}$.
	
	{\bf Proofs of parts (i) and (ii)}:
	Let $\lambda\in\CC_{>0}$ and let $r>0$ and $\theta\in(-\pi/2,\pi/2)$ be such that $\lambda=re^{i\theta}$.
	Then for $\delta\in(0,1)$,
	\begin{align*}
		g_1(\delta\lambda)=\sqrt{1-2d\delta r\cos\theta+(d\delta r\cos\theta)^2+(3d\delta r\sin\theta)^2}+\frac{\delta^2r^2}{c},
	\end{align*}
	and
	\begin{align*}
		\lim_{\delta\downarrow0}\frac{g_1(\delta\lambda)-g_1(0)}{\delta}&=-dr\cos\theta<0.
	\end{align*}
	Therefore, since $g_1(0)=1$, there exists $\delta_\lambda>0$ such that $\delta\lambda\in\ball(0,\ve/2)$ and $g_1(\delta\lambda)<1$ for all $\delta\in(0,\delta_\lambda)$, so $\delta\lambda\in A_{>0}$.
	Thus part (i) holds.
	Now, let $\lambda\in\CC_{<0}$ and $r>0$ and $\theta\in(-\pi/2,\pi/2)$ be such that $\lambda=-re^{i\theta}$.
	An analogous argument (using $g_2$ instead of $g_1$) shows that there exists $\delta_\lambda>0$ such that $\delta\lambda\in A_{<0}$ for all $\delta\in(0,\delta_\lambda)$, so (ii) holds.
	To avoid repetition, we omit the details.
	
	{\bf Proofs of parts (iii) and (iv)}:
	Fix $\beta>\bnbhdzero^\dag$.
	Recall that $\nu_\beta(0)=1$.
	Suppose $\lambda\in\ball(0,\ve/2)$ so that $\ve-|\lambda|>\ve/2$.
	By Lemma \ref{lem:taylor} with $z=0$ and $r=\ve$, the bound in \eqref{eq:supnubetaK0} and the definitions of $N$ and $c$ in \eqref{eq:consts1}, we have
	\begin{equation}\label{eq:taylor0}
		\abs{\nu_\beta(\lambda)-1-\nu'_\beta(0)\lambda}<\frac{2 N \abs{\lambda}^2}{\ve(\ve-|\lambda|)}\le\frac{|\lambda|^2}{c}.
	\end{equation}
	Let $\lambda\in A_{>0}\subset\ball(0,\ve/2)$.
	The triangle inequality, the inequalities \eqref{eq:dnuprime0} and \eqref{eq:taylor0}, and the definition of $A_{>0}$ in \eqref{eq:A>0} imply
	\begin{align*}
	\abs{\nu_\beta(\lambda)}&\le\abs{1+\nu'_\beta(0)\Re\lambda+i\nu'_\beta(0)\Im\lambda}+\abs{\nu_\beta(\lambda)- 1-\nu'_\beta(0)\lambda}\\
	&\le\abs{1-d\Re\lambda+i3d\Im\lambda}+\frac{|\lambda|^2}{c} \\
	&<1.
	\end{align*}    
	Along with equation \eqref{eq:rhoUnubeta0} for $\rho(\monoU_\lambda^\beta(s))$, this proves that $\rho(\monoU_\lambda^\beta(s))<1$ for all $\lambda\in A_{>0}$, so part (iii) holds.
	Alternatively, let $\lambda\in A_{<0}$.
	An analogous argument, which we omit to avoid repetition, shows that $\abs{\nu_\beta(\lambda)}>1$ for all $\lambda\in A_{<0}$.
	Along with equation \eqref{eq:rhoUnubeta0} this proves that $\rho(\monoU_\lambda^\beta(s))>1$ for all $\lambda\in A_{<0}$, so part (iv) holds.
\end{proof}

\subsection{Proof of Corollary \ref{cor:01}}
\label{sec:01}

\begin{proof}[Proof of Corollary \ref{cor:01}]
    By the definition of $\nu_\star$ in \eqref{eq:nuast}, for $r\in[0,1]$,
    \begin{align}\label{eq:nuastr}
	    {\nu_\star(r)}&={\frac{(1-\vr_1-(1-r))(1-\vr_2-(1-r))}{(1-\vr_1)(1-\vr_2)}}\\
	    \notag
	    &=1-\frac{(1-e^{-\alpha}+r)(1-r)}{(1-\vr_1)(1-\vr_2)}\\
	    \notag
	    &=-1-\frac{\Delta-(r-r_0)^2}{(1-\vr_1)(1-\vr_2)}.
    \end{align}
    It follows that $\nu_\star(r)<1$ for all $r\in(0,1)$.
    Solving for $\nu_\star(r)<-1$ yields the characterization of $\stable_\star(0)$ in parts (i)--(iv) and the characterization of $\unstable_\star(0)$ in part (v).
    
    Next, we complete the proofs of parts (i) and (iii).
    Suppose $\alpha>0$.
    Let $\bnbhdone(\alpha,f)\ge\buniq$ and $A_{<1}\subset\CC_{<1}$ be as in Theorem \ref{thm:2a}.
    By Theorem \ref{thm:2a}(i) with $\lambda=0$, there exists $\eta>0$ such that $[1-\eta,1)\subset A_{<1}$.
    Suppose $\Delta<0$.
    Then it follows from equation \eqref{eq:nuastr} for $\nu_\star(r)$ and the assumption $\alpha>0$ that
    	$$\gamma_a=\sup\lcb\abs{\nu_\star(r)}:r\in[0,1-\eta]\rcb<1.$$
    Set $\delta_a=1-\gamma_a>0$ so that $[0,1)\subset\stable_\star(\delta_a)\cup A_{<1}$.
    Let $\bmono(\alpha,f,[0,1-\eta],\delta_a)\ge\buniq$ be as in Theorem \ref{thm:1a}.
    By Theorems \ref{thm:1a} and \ref{thm:2a}(iii), if $\beta>\max(\bmono(\alpha,f,[0,1-\eta],\delta_a),\bnbhdone(\alpha,f))$ then $\stable_\star(\delta_a)\cup A_{<1}\subset\stable$.
    This proves part (i).
    Now suppose $\Delta>0$ and $\ve\in(0,\sqrt{\Delta})$.
    It follows from equation \eqref{eq:nuastr} for $\nu_\star(r)$ and the assumption $\alpha>0$ that 
    	$$\gamma_c=\sup\lcb\abs{\nu_\star(r)}:r\in\lsb0,r_0-\sqrt{\Delta}-\ve\rsb\cup\lsb r_0+\sqrt{\Delta}+\ve,1-\eta\rsb\rcb<1.$$
    Set $\delta_c=1-\gamma_c>0$ so that  $[0,r_0-\sqrt{\Delta}-\ve)\cup(r_0+\sqrt{\Delta}+\ve,1)\subset\stable_\star(\delta_c)\cup A_{<1}$.
    Let $\bmono(\alpha,f,[0,1-\eta],\delta_c)\ge\buniq$ be as in Theorem \ref{thm:1a}.
    By Theorems \ref{thm:1a} and \ref{thm:2a}(iii), if $\beta>\max(\bmono(\alpha,f,[0,1-\eta],\delta_c),\bnbhdone(\alpha,f))$ then $\stable_\star(\delta_c)\cup A_{<1}\subset\stable$.
    This proves part (iii).
    
    Next, we complete the proofs of parts (ii) and (iv).
    Suppose $\alpha=0$.
    Let $\bnbhdone(0,f)\ge\buniq$ and $A_{<1}\subset\CC_{<1}$ be as in Theorem \ref{thm:2a}, and let $\bnbhdzero(f)\ge\buniq$ and $A_{>0}\subset\CC_{>0}$ be as in Theorem \ref{thm:nbhdof0}.
    By Theorem \ref{thm:2a}(i) with $\lambda=0$ and Theorem \ref{thm:nbhdof0}(i) with $\lambda=1$, there exists $\eta>0$ such that $[1-\eta,1)\subset A_{<1}$ and $(0,\eta]\subset A_{>0}$.
    Suppose $\Delta<0$.
    By equation \eqref{eq:nuastr} for $\nu_\star(r)$, 
    $$\gamma_b=\sup\{\abs{\nu_\star(r)}:r\in[\eta,1-\eta]\}<1.$$
    Set $\delta_b=1-\gamma_b>0$ so that $(0,1)\subset A_{>0}\cup\stable_\star(\delta_b)\cup A_{<1}$.
    Let $\bmono(\alpha,f,[\eta,1-\eta],\delta_b)\ge\buniq$ be as in Theorem \ref{thm:1a}.
    By Theorems \ref{thm:1a}, \ref{thm:2a}(iii) and \ref{thm:nbhdof0}(iii), if $\beta>\max(\bmono(\alpha,f,[\eta,1-\eta],\delta_b),\bnbhdone(0,f),\bnbhdzero(f))$ then $ A_{>0}\cup\stable_\star(\delta_b)\cup A_{<1}\subset\stable$.
    This proves part (ii).
    
    Now suppose $\Delta>0$ and $\ve\in(0,\sqrt{\Delta})$.
    By equation \eqref{eq:nuastr} for $\nu_\star(r)$,
    	$$\gamma_d=\sup\lcb\abs{\nu_\star(r)}:r\in\lsb\eta,r_0-\sqrt{\Delta}-\ve\rsb\cup\lsb r_0+\sqrt{\Delta}+\ve,1-\eta\rsb\rcb<1.$$
    Set $\delta_d=1-\gamma_d>0$ so that $(0,r_0-\sqrt{\Delta}-\ve)\cup(r_0+\sqrt{\Delta}+\ve,1)\subset A_{>0}\cup\stable_\star(\delta_d)\cup A_{<1}$.
    Let $\bmono(\alpha,f,[\eta,1-\eta],\delta_d)\ge\buniq$ be as in Theorem \ref{thm:1a}.
    By Theorems \ref{thm:1a} and \ref{thm:2a}(iii), if $\beta>\max(\bmono(\alpha,f,[\eta,1-\eta],\delta_d),\bnbhdone(0,f),\bnbhdzero(f))$ then $A_{>0}\cup\stable_\star(\delta_d)\cup A_{<1}\subset\stable$.
    This proves part (iv).
    
    Lastly, we complete the proof of part (v).
    Suppose $\Delta>0$ and $\ve\in(0,r_0-\sqrt{\Delta})$.
    By equation \eqref{eq:nuastr} for $\nu_\star(r)$,
   		$$\gamma_e=\inf\lcb\abs{\nu_\star(r)}:r\in\lsb r_0-\sqrt{\Delta}+\ve,r_0+\sqrt{\Delta}-\ve\rsb\rcb>1.$$
    Set $\delta_e=\gamma_e-1>0$ so that $(r_0-\sqrt{\Delta}+\ve,r_0+\sqrt{\Delta}-\ve)\subset\unstable_\star(\delta_e)$.
    By Theorem \ref{thm:1a}, if $\beta>\bmono(\alpha,f,[r_0-\sqrt{\Delta},r_0+\sqrt{\Delta}],\delta_e)$ then $\unstable_\star(\delta_e)\subset\unstable$.
    This proves part (v).
\end{proof}

\section{Proofs of main results}
\label{sec:mainproofs}

We now use Theorem \ref{thm:msf} on the relation between extended characteristic multipliers and stability of a synchronous SOPS, along with our description for the asymptotics of the extended characteristic multipliers given in Section \ref{sec:asym_char} to prove our main results.
Throughout this section we recall that if a synchronous SOPS is linearly stable then it is asymptotically stable with an exponential phase, and if a synchronous SOPS is linearly unstable then it is unstable (see Remark \ref{rem:linearstability}).

\subsection{Proof of Theorem \ref{thm:1}}
\label{sec:generalproof}

In this section we use the limits of the extended characteristic multipliers established in Theorem \ref{thm:1a} to prove our main result on systems with general coupling.

\begin{proof}[Proof of Theorem \ref{thm:1}]
	Let $\bmono=\bmono(\alpha,\beta,f,G)\ge\buniq$ be as in Theorem \ref{thm:1a}.
	Let $\beta>\bmono$, $n\ge2$ and $G\in\M_n^1$ be such that $\sigma(G)\subset K$.
	Suppose the eigenvalue $\lambda=1$ of $G$ has simple algebraic multiplicity and $\abs{\nu_\star(\lambda)}<1-\delta$ for all other eigenvalues $\lambda$ of $G$.
	In view of Theorem \ref{thm:1a}, this implies that
		$$\rho(\mono_\lambda)<\abs{\nu_\star(\lambda)}+\delta<1,\qquad\text{ for all }\lambda\in\sigminus(G).$$
	Thus, by Theorem \ref{thm:msf}, the synchronous SOPS $\sopsn$ for the $n$-dimensional coupled DDE \eqref{eq:cdde} associated with $(\alpha,\beta,f,G)$ is linearly stable, and therefore asymptotically stable with an exponential phase.
	On the other hand, suppose $\abs{\nu_\star(\lambda)}>1+\delta$ for some eigenvalue $\lambda$ of $G$.
	In view of Theorem \ref{thm:1a}, this implies that
		$$\rho(\mono_\lambda)>\abs{\nu_\star(\lambda)}-\delta>1,\qquad\text{ for some }\lambda\in\sigma(G).$$
	Hence, by Theorem \ref{thm:msf}, the synchronous SOPS $\sopsn$ for the $n$-dimensional coupled DDE \eqref{eq:cdde} associated with $(\alpha,\beta,f,G)$ is linearly unstable, and therefore unstable.
\end{proof}

\subsection{Proof of Theorem \ref{thm:2}}
\label{sec:weakproof}

\begin{proof}[Proof of Theorem \ref{thm:2}]
	Let $A_{<1}\subset\CC_{<1}$, $A_{>1}\subset\CC_{>1}$ and $\bnbhdone=\bnbhdone(\alpha,f)\ge\buniq$ be as in Theorem \ref{thm:2a}.
	Let $\beta>\bnbhdone$, $n\ge2$ and $\gmatrix\in\M_n^0$.
	Suppose $\sigminusz(\gmatrix)\subset\CC_{<0}$.
	It follows from the expression \eqref{eq:sigmaweakmatrix} for $\sigminusz(I_n+\eta\gmatrix)$ and parts (i) and (iii) of Theorem \ref{thm:2a} that there exists $\eta_\gmatrix>0$ such that $\sigminus(I_n+\eta\gmatrix)\subset A_{<1}\subset\stable$ for all $\eta\in(0,\eta_\gmatrix)$.
	Linear stability of the synchronous SOPS then follows from Theorem \ref{thm:msf}.
	Therefore the synchronous SOPS is asymptotically stable with an exponential phase.
	The case that $\sigminusz(\gmatrix)\subset\CC_{>0}$ can be shown using an exactly analogous argument (and by choosing $\eta_\gmatrix>0$ possibly smaller), and so we omit the details.
	This proves (i).

	On the other hand, suppose $\sigminusz(\gmatrix)\cap{\CC}_{>0}\ne\emptyset$.
	Let $\lambda\in\sigminus(\gmatrix)\cap{\CC}_{>0}$.
	By the expression \eqref{eq:sigmaweakmatrix} for $\sigminusz(I_n+\eta\gmatrix)$ and parts (ii) and (iv) of Theorem \ref{thm:2a}, and by possibly choosing $\eta_\gmatrix>0$ smaller, we see that $1+\eta\lambda\in A_{>1}\subset\unstable$ for all $\eta\in(0,\eta_\gmatrix)$.
	Thus, $\sigminus(I_n+\eta \gmatrix)\cap\unstable\ne\emptyset$ for all $\eta\in(0,\eta_\gmatrix)$.
	Linear instability of the synchronous SOPS then follows from Theorem \ref{thm:msf}.
	Therefore the synchronous SOPS is unstable.
	The case that $\sigminusz(\gmatrix)\cap\CC_{<0}\ne\emptyset$ can be shown using an exactly analogous argument (and by choosing $\eta_\gmatrix>0$ possibly smaller), and so we omit the details.
	This proves part (ii). 
\end{proof}

\subsection{Proof of Theorem \ref{thm:3}}
\label{sec:uniformproof}

\begin{proof}[Proof of Theorem \ref{thm:3}]
	Suppose $\alpha>0$.
	Let $\delta=\frac{1}{2}\lb1-\nu_\star(0)\rb>0$, $K=\overline{\ball(0,\delta)}\cup\{1\}$, and $\bmono=\bmono(\alpha,f,\delta,K)\ge\buniq$ be as in Theorem \ref{thm:1a}.
	Let $\beta>\bmono$, $n\ge 2$ and $\gmatrix\in\M_n^0$.
 	Recall that ${\bf 1}_n$ is an eigenvector of $\MF_n$ (resp.\ $\gmatrix$) associated with eigenvalue 1 (resp.\ 0), and any eigenvector ${\bf v}$ of $\MF_n+\eta\gmatrix$ are either in the span of ${\bf 1}_n$ or is orthogonal to ${\bf 1}_n$.
	If ${\bf v}$ is an eigenvector of $\MF_n+\eta\gmatrix$ associated with eigenvalue $\lambda$ that is orthogonal to ${\bf 1}_n$, then $(\MF_n+\eta\gmatrix){\bf v}=\eta\gmatrix{\bf v}=\eta\lambda{\bf v}$.
	Thus,
		\be\label{eq:sigminusrelation}\sigminus(\MF_n+\eta\gmatrix)=\eta\sigminusz(\gmatrix).\ee
	Therefore, we can choose $\eta_\gmatrix>0$ such that $\eta\sigminusz(H)\subset\ball(0,\delta)$ for all $\eta\in(-\eta_\gmatrix,\eta_\gmatrix)$.
	Thus, $\sigma(\MF_n+\eta\gmatrix)=(\eta\sigminus(\gmatrix))\cup\{1\}\subset K$ for all $\eta\in(-\eta_H,\eta_H)$ and
		$$\rho(\mono_\lambda)<\nu_\star(0)+\delta=\frac{1}{2}(1+\nu_\star(0))<1\qquad\text{for all }\lambda\in\sigminusz(H).$$
	Linear stability of the synchronous SOPS then follows from Theorem \ref{thm:msf}.
	Therefore the synchronous SOPS is asymptotically stable with an exponential phase.

	Let $A_{>0}\subset\CC_{>0}$, $A_{<0}\subset\CC_{<0}$ and $\bnbhdzero^\dag=\bnbhdzero^\dag(f)\ge\buniq$ be as in Theorem \ref{thm:nbhdof0}.
	Let $\beta>\bnbhdzero^\dag$, $n\ge2$ and $\gmatrix\in\M_n^0$.
	Suppose $\sigminusz(\gmatrix)\subset\CC_{>0}$.
	It follows from relation \eqref{eq:sigminusrelation} and parts (i) and (iii) of Theorem \ref{thm:nbhdof0}, and by possibly choosing $\eta_\gmatrix>0$ smaller, that $\sigminus(\MF_n+\eta\gmatrix)\subset A_{>0}\subset\stable$ for all $\eta\in(0,\eta_\gmatrix)$.
	Linear stability of the synchronous SOPS then follows from Theorem \ref{thm:msf}.
	Therefore the synchronous SOPS is asymptotically stable with an exponential phase.
	The case that $\sigminusz(\gmatrix)\subset\CC_{<0}$ can be shown using an exactly analogous argument (and by choosing $\eta_\gmatrix>0$ possibly smaller), and so we omit the details.
	This proves (ii).
	
	On the other hand, suppose $\sigminusz(\gmatrix)\cap{\CC}_{<0}\ne\emptyset$.
	Let $\lambda\in\sigminus(\gmatrix)\cap{\CC}_{<0}$.
	By relation \eqref{eq:sigminusrelation} and parts (ii) and (iv) of Theorem \ref{thm:nbhdof0}, and by possibly choosing $\eta_\gmatrix>0$ smaller, we see that $\eta\lambda\in A_{<0}\subset\unstable$ for all $\eta\in(0,\eta_\gmatrix)$.
	Thus, $\sigminus(\MF_n+\eta \gmatrix)\cap\unstable\ne\emptyset$ for all $\eta\in(0,\eta_\gmatrix)$.
	Linear instability of the synchronous SOPS then follows from Theorem \ref{thm:msf}.
	Therefore the synchronous SOPS is unstable.
	The case that $\sigminusz(\gmatrix)\cap\CC_{>0}\ne\emptyset$ can be shown using an exactly analogous argument (and by choosing $\eta_\gmatrix>0$ possibly smaller), and so we omit the details.
	This proves part (iii). 	
	
	This completes the proof with $\bnbhdzero(\alpha,f)=\max(\bmono(\alpha,f,\delta,K),\bnbhdzero^\dag(f))$.
\end{proof}

\subsection{Proof of Theorem \ref{thm:4}}
\label{sec:01proof}

\begin{proof}[Proof of Theorem \ref{thm:4}]
First recall that if $G\in\M_n^1$ is an irreducible doubly nonnegative matrix, then by the Perron-Frobenius theorem (see, e.g., \cite[Theorem 8.4.4]{horn2012matrix}) and the fact that $G$ has nonnegative eigenvalues, we have $\sigminus(G)\subset[0,1)$.
If $G$ is also positive definite, then $G$ has positive eigenvalues and so $\sigminus(G)\subset(0,1)$.

Suppose $\Delta<0$.
Let $\bzoa=\bzoa(\alpha,f)\ge\buniq$ be as in Corollary \ref{cor:01}.
Let $\beta>\bzoa$, $n\ge2$ and $G\in\M_n^1$ be an irreducible doubly nonnegative matrix.
If $\alpha>0$ then by the fact that $\sigminus(G)\subset[0,1)$ and Corollary \ref{cor:01}, $\sigminus(G)\subset\stable$.
Alternatively, if $\alpha=0$ and $G$ is also positive definite, then $\sigminus(G)\subset(0,1)\subset\stable$.
In either case, Theorem \ref{thm:msf} implies that the synchronous SOPS is linearly stable and so (i) holds.

Next, suppose $\Delta>0$ and let $\ve$ satisfy the bounds \eqref{eq:ve}.
Let $\beta>\bzob$, $n\ge2$ and $G\in\M_n^1$ be an irreducible doubly nonnegative matrix.
Suppose $\abs{\lambda-r_0}>\sqrt{\Delta}+\ve$ for all $\lambda\in\sigminus(G)$.
If $\alpha>0$ then by the fact that $\sigminus(G)\subset[0,1)$ and and Corollary \ref{cor:01}, $\sigminus(G)\subset[0,r_0-\sqrt{\Delta}-\ve)\cup(r_0+\sqrt{\Delta}+\ve,1)\subset\stable$.
Alternatively, if $\alpha=0$ and $G$ is also positive definite then by the fact that $\sigminus(G)\subset(0,1)$ and Corollary \ref{cor:01}, $\sigminus(G)\subset(0,r_0-\sqrt{\Delta}-\ve)\cup(r_0+\sqrt{\Delta}+\ve,1)\subset\stable$.
In either case, Theorem \ref{thm:msf} implies that the synchronous SOPS is linearly unstable and so (ii) holds.
Now suppose there is $\lambda\in\sigma(G)$ such that $\abs{\lambda-r_0}<\sqrt{\Delta}-\ve$. 
By the fact that $\sigminus(G)\subset[0,1)$ and Corollary \ref{cor:01}, $\lambda\in(r_0-\sqrt{\Delta}+\ve,r_0+\sqrt{\Delta}-\ve)\subset\unstable$.
Thus, Theorem \ref{thm:msf} implies (iii) holds.
\end{proof}

\section{Examples}
\label{sec:examples}

In this section we apply our results to examples of systems of DDEs with mean-field coupling and systems of coupled DDEs arranged in a ring with nearest neighbor coupling, which are commonly studied coupling topologies in both the scientific and mathematical literature.

\subsection{Systems of DDEs with mean-field coupling}
\label{ex:meanfield}

Systems of coupled oscillators with mean field coupling were among the first to be addressed (in the setting without delays) \cite{Kuramoto1984,Winfree1967}.
Suppose $\alpha\ge0$ and $f$ satisfies Assumption \ref{ass:main}.
For each $n\geq2$ and $\kappa\in\R$, define the coupling matrix $M_{n,\kappa}=(M_{n,\kappa}^{jk})\in\M_n^1$ by
\be\label{eq:mfmatrix}
M_{n,\kappa}^{jk}=
\begin{cases}
	\displaystyle
	1-\frac{(n-1)\kappa}{n}&\text{if }j=k,\\
	\displaystyle
	\frac{\kappa}{n}&\text{if }j\neq k.
\end{cases}
\ee
The spectrum of $M_{n,\kappa}$ satisfies $\sigma(M_{n,\kappa})=\{1-\kappa,1\}$, where $\lambda=1$ is a simple eigenvalue of $M_{n,\kappa}$ and the eigenvalue $\lambda=1-\kappa$ has multiplicity $n-1$. 
Observe that for $\kappa\in(0,1]$ the matrix $M_{n,\kappa}$ is irreducible and doubly nonnegative, and if $\kappa<1$ also holds then $M_{n,\kappa}$ is positive definite.
Thus, we have the following immediate corollary of Theorem \ref{thm:4}.

\begin{cor}
	\label{cor:mfcoupling}
	Suppose $\alpha\ge0$ and $f$ satisfies Assumption \ref{ass:main}.
	Define $r_0$ and $\Delta$ as in \eqref{eq:r0} and \eqref{eq:Delta}, respectively.
	Suppose $\Delta<0$.
	Then there exists $\bzoa=\bzoa(\alpha,f)\ge\buniq$ such that for every $\beta>\bzoa$, $n\ge2$, and $\kappa\in(0,1]$ (if $\alpha=0$ then additionally assume $\kappa<1$), the following holds:
	\begin{itemize}
		\item[(i)] The synchronous SOPS of the $n$-dimensional coupled DDE associated with $(\alpha,\beta,f,M_{n,\kappa})$ is asymptotically stable with an exponential phase.
	\end{itemize}
	On the other hand, suppose $\Delta>0$.
	Let $\ve$ satisfy \eqref{eq:ve}.
	Then there exists $\bzob=\bzob(\alpha,f,\ve)\ge\buniq$ such that for every $\beta>\bzob$, $n\ge2$ and $\kappa\in(0,1]$ (if $\alpha=0$ then additionally assume $\kappa<1$), the following hold:
	\begin{itemize}
		\item[(ii)] If 
		$$\kappa<1-r_0-\sqrt{\Delta}-\ve\qquad\text{or}\qquad\kappa>1-r_0+\sqrt{\Delta}+\ve,$$ 
		the unique synchronous SOPS of the $n$-dimensional coupled DDE associated with $(\alpha,\beta,f,M_{n,\kappa})$ is asymptotically stable with an exponential phase.
		\item[(iii)] If 
		$$1-r_0-\sqrt{\Delta}+\ve<\kappa<1-r_0+\sqrt{\Delta}-\ve,$$ 
		the unique synchronous SOPS of the $n$-dimensional coupled DDE associated with $(\alpha,\beta,f,M_{n,\kappa})$ is unstable.
	\end{itemize}
\end{cor}

\subsection{Systems of DDEs arranged in a ring}
\label{sec:ring}

Consider a system of $n\ge 3$ coupled DDEs arranged in a ring with nearest neighbor coupling; that is, the coupling matrix $R_n=(R_n^{jk})$ satisfies $R_n^{jk}=0$ if $j-k\;(\text{mod } n)>1$.
Systems of DDEs arranged in a ring arise in biological applications, including models of neuronal networks \cite{Perlikowski2010,Popovich2011,Su2017} and cardiac rhythms \cite{Tian2017}, and the stability of the synchronous state when the dimension of the system is large has been studied in the mathematical literature \cite{Chen2000,Guo2009}.

For $n \geq 3$ and $\kappa\in\R^2\setminus\{(0,0)\}$ define the coupling matrix $R_{n,\kappa}\in\M_n^1$ by
\begin{equation}\label{eq:ring}
	R_{n,\kappa}^{jk}=\begin{cases}
		\frac{\kappa^1}{2}&\text{if }j=(k-1)\mod n,\\
		1-\frac{\kappa^1+\kappa^2}{2}&\text{if }j=k,\\
		\frac{\kappa^2}{2}&\text{if }j=(k+1)\mod n,\\
		0&\text{otherwise}.
	\end{cases}
\end{equation}
The spectrum of the coupling matrix is given by $\sigma(R_{n,\kappa})=\{z_j\}_{j=0}^{n-1}$, where
\begin{equation}
	\label{eq:lambdaj}
	z_j=z_j(n,\kappa)= 1-\frac{\kappa^1+\kappa^2}{2}\lb 1-\cos\frac{2\pi j}{n}\rb+i\frac{\kappa^1-\kappa^2}{2}\sin\frac{2\pi j}{n},\qquad j=0,\dots,n-1,
\end{equation}
and $z_0=1$ is a simple eigenvalue of $R_{n,\kappa}$ (see, e.g., \cite[Theorem 3.1]{Gray2006}).
Consider the case of symmetric coupling (i.e., $\kappa^1=\kappa^2$).
In this case
	\be\label{eq:lambdajreal}z_j=1-\kappa^1\lb 1-\cos\frac{2\pi j}{n}\rb,\qquad j=0,\dots,n-1,\ee
and so $\sigma(R_{n,\kappa})\subset[1-2\kappa^1,1]$ for all $n\ge 3$ and $\kappa^1\ge0$.
Note that for $\kappa^1\in(0,\frac{1}{2}]$ the coupling matrix $R_{n,\kappa}$ is irreducible and doubly nonnegative matrix, and if $\kappa^1<\frac{1}{2}$ also holds then $R_{n,\kappa}$ is positive definite.
We have the following corollary of Theorem \ref{thm:4}.
Given $\alpha\ge0$ and $f$ satisfying Assumption \ref{ass:main}, recall the definitions of $r_0$ and $\Delta$ given in \eqref{eq:r0} and \eqref{eq:Delta}, respectively.

\begin{cor}\label{cor:symring}
	Suppose $\alpha\ge0$ and $f$ satisfies Assumption \ref{ass:main}.
	Suppose $\Delta<0$.
	Let $\bzoa=\bzoa(\alpha,f)\ge\buniq$ be as in Theorem \ref{thm:4}.
	Then for every $\beta>\bzoa$, $n\ge3$, and $\kappa\in(0,\frac{1}{2}]$, with additional assumption $\kappa<\frac{1}{2}$ if $\alpha=0$, the following holds:
	\begin{itemize}
		\item[(i)] The synchronous SOPS of the $n$-dimensional coupled DDE associated with $(\alpha,\beta,f,R_{n,\kappa})$ is asymptotically stable with an exponential phase.
	\end{itemize}
	On the other hand, suppose $\Delta>0$.
	Suppose $\ve$ satisfies \eqref{eq:ve}.
	Let $\bzob=\bzob(\alpha,f,\ve)\ge\buniq$ be as in Theorem \ref{thm:4}.
	Then for every $\beta>\bzob$, $n\ge3$ and $\kappa\in(0,\frac{1}{2}]$ (if $\alpha=0$ then additionally assume $\kappa<\frac{1}{2}$), the following hold:
	\begin{itemize}
		\item[(ii)] If for each $1\le j\le n-1$, one of the following inequalities holds
			$$\kappa<\frac{1-r_0-\sqrt{\Delta}-\ve}{1-\cos\frac{2\pi j}{n}}\qquad\text{or}\qquad\kappa>\frac{1-r_0+\sqrt{\Delta}+\ve}{1-\cos\frac{2\pi j}{n}},$$ 
		then the unique synchronous SOPS of the $n$-dimensional coupled DDE associated with $(\alpha,\beta,f,R_{n,\kappa})$ is asymptotically stable with an exponential phase.
		\item[(iii)] If for some $1\le j\le n-1$, the following inequality holds
		\be\label{eq:kappabound}\frac{1-r_0-\sqrt{\Delta}+\ve}{1-\cos\frac{2\pi j}{n}}<\kappa<\frac{1-r_0+\sqrt{\Delta}-\ve}{1-\cos\frac{2\pi j}{n}},\ee
		then the unique synchronous SOPS of the $n$-dimensional coupled DDE associated with $(\alpha,\beta,f,R_{n,\kappa})$ is unstable.
	\end{itemize}
\end{cor}

We also have the following corollary of Theorem \ref{thm:1} for the case that $\kappa>\frac{1-r_0}{4}$.

\begin{cor}
	Suppose $\alpha\ge0$ and $f$ satisfies Assumption \ref{ass:main}.
	Given $\delta\in(0,1)$, let $\bmono=\bmono\lb\alpha,f,\delta,[-1,1+\delta]\rb\ge\buniq$ be as in Theorem \ref{thm:1}.
	Then for every $\beta>\bmono$, $n\ge 3$, and $\kappa\in\lb(1-r_0)(1+\sqrt{1+\delta})/4,1\rb$, the following hold:
	\begin{itemize}
		\item[(i)] If $n$ is even then the synchronous SOPS of the $n$-dimensional coupled DDE associated with $(\alpha,\beta,f,R_{n,\kappa})$ is unstable.
		\item[(ii)] If $n$ is odd and 
			\be\label{eq:kappanodd}\kappa>\frac{(1-r_0)(1+\sqrt{1+\delta})}{1-\cos\frac{(n-1)\pi}{n}},\ee
		then the synchronous SOPS of the $n$-dimensional coupled DDE associated with $(\alpha,\beta,f,R_{n,\kappa})$ is unstable.
	\end{itemize}
\end{cor}

\begin{remark}
	\label{rem:chen}
	We compare the corollary with the main result in \cite{Chen2000}, which states that if $\alpha>0$ and $\kappa=1$, then for any $\beta>\buniq$, and either $n$ odd and sufficiently large or $n$ even, the synchronous SOPS is unstable.
	Fix $\alpha>0$ and $\kappa=1$.
	Let $\delta>0$ be sufficiently small so that $(1-r_0)(1+\sqrt{1+\delta})<2$.
	Then, by Corollary \ref{cor:symring}, if $\beta>\bmono=\bmono(\alpha,f,\delta,[-1,1+\delta])$ and either $n$ is odd and sufficiently large so that inequality \eqref{eq:kappanodd} holds with $\kappa=1$ or $n$ is even, then the synchronous SOPS associated with $(\alpha,\beta,f,R_{n,\kappa})$ is unstable.
	In particular, we recover the main result from \cite{Chen2000} in the asymptotic regime $\beta\to\infty$.
\end{remark}

\section{Numerical simulations}
\label{sec:numerics}

In this section we plot examples of solutions of the coupled DDE \eqref{eq:cdde} to gain insight into how large $\beta$ needs to be in order for our limit theorems to be informative about the stability of the synchronous SOPS.
Let $\alpha=\frac18$, $a=24$ and $b=1$, and let $f:\R\to\R$ be the hyperbolic tangent function defined in \eqref{eq:fmf}.
In Figure \ref{fig:nu_star} we depict the corresponding level sets $\pd\stable_\star(0.1)=\{\lambda\in\CC:\abs{\nu_\star(\lambda)}=0.9\}$, $\pd\stable_\star(0)=\{\lambda\in\CC:\abs{\nu_\star(\lambda)}=1\}$ and $\pd\unstable_\star(0.1)=\{\lambda\in\CC:\abs{\nu_\star(\lambda)}=1.1\}$ defined in Section \ref{sec:general}, and we mark the following points (and their conjugates) with red dots (we also list the approximate the value of $\nu_\star$ at each point):
\begin{align*}
	\lambda_1&=\text{argmax}\{\Re\lambda:\lambda\in\stable_\star(0.1)\cap\CC_{>\frac12}\}\approx 0.99,&&\nu_\star(\lambda_1)=0.9,\\
	\lambda_2&=\text{argmax}\{\Im\lambda:\lambda\in\stable_\star(0.1)\cap\CC_{>\frac12}\}\approx 0.80+i0.16,&&\nu_\star(\lambda_2)\approx -0.49+i0.98,\\
	\lambda_3&=\text{argmin}\{\Re\lambda:\lambda\in\stable_\star(0.1)\cap\CC_{>\frac12}\}\approx 0.62,&&\nu_\star(\lambda_3)=-0.9,\\
	\lambda_4&=\text{argmax}\{\Re\lambda:\lambda\in\unstable_\star(0.1)\cap\CC_{>\frac12}\}\approx 1.01,&&\nu_\star(\lambda_4)=1.1,\\
	\lambda_5&=\text{argmax}\{\Im\lambda:\lambda\in\unstable_\star(0.1)\cap\CC_{>\frac12}\}\approx 0.80+i0.2,&&\nu_\star(\lambda_5)\approx -0.40+i0.81,\\
	\lambda_6&=\text{argmin}\{\Re\lambda:\lambda\in\unstable_\star(0.1)\cap\CC_{>\frac12}\}\approx 0.49,&&\nu_\star(\lambda_6)=-1.1.
\end{align*}
Recall the definition of the coupling matrix $R_{3,\kappa}\in\M_n^1$ defined in Equation \eqref{eq:ring}.
Using Equation \eqref{eq:lambdaj} we solve for $\kappa_{\ell}\in\R^2$ such that $\sigminus(R_{3,\kappa_{\ell}})=\{\lambda_{\ell},\bar\lambda_{\ell}\}$, for $j=1,\dots,6$.
The approximate values for $\kappa_{\ell}$ are
\begin{align*}
	\kappa_1&\approx(0.0045,0,0045),&&\kappa_2\approx(-0.03,0.15),&&\kappa_3\approx(0.13,0.13),\\
	\kappa_4&\approx(-0.0043,-0.0043),&&\kappa_5\approx(-0.05,-0.18),&&\kappa_6\approx(0.17,0.17).
\end{align*}

\begin{figure}   
    \centering
    \includegraphics[width=\textwidth]{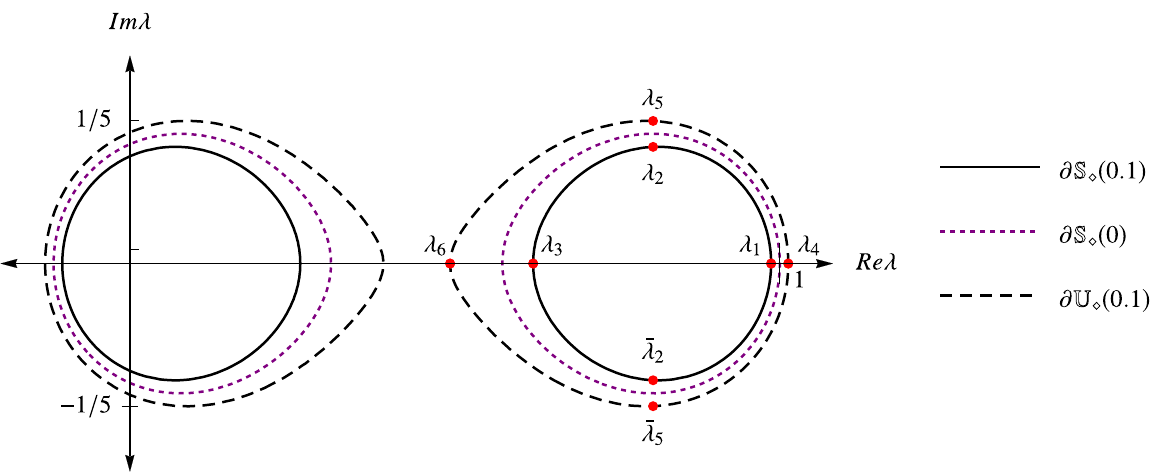}
    \caption[Level sets $\pd\stable_\star(0)$, $\pd\stable_\star(0.1)$ and $\pd\unstable_\star(0.1)$]{The level sets $\pd\stable_\star(0.1)$, $\pd\stable_\star(0)$, and $\pd\unstable_\star(0.1)$ corresponding to parameters $\alpha=\frac18$, $a=24$ and $b=1$.
    See the text for the coordinates of the red points lying in the level sets $\pd\stable_\star(0.1)$ and $\pd\unstable(0.1)$.
    }
    \label{fig:nu_star}
\end{figure}

According to Theorem \ref{thm:1}, for $\beta$ sufficiently large, the synchronous SOPS of the coupled DDE \eqref{eq:cdde} associated with $(\frac18,\beta,f,R_{3,\kappa_\ell})$ is asymptotically stable with an exponential phase (resp.\ unstable) for $\ell=1,2,3$ (resp.\ $\ell=4,5,6$).
Note that the minimal value of $\beta$ is $\beta_\text{Hopf}$, defined by
	$$\beta_\text{Hopf}=\frac{1}{\abs{f'(0)}}\frac{\theta_0}{\sin\theta_0}\approx1.65,$$
where $\theta_0\in(\tfrac{\pi}{2},\pi)$ denotes the solution to $\theta_0+\alpha\tan\theta_0=0$ if $\alpha>0$ and $\theta_0=\tfrac{\pi}{2}$ if $\alpha=0$.
At $\beta_\text{Hopf}$ there is a Hopf bifurcation such that for $\beta>\beta_\text{Hopf}$ the trivial solution of the scalar DDE \eqref{eq:dde} is linearly unstable (see, e.g., \cite{Hadeler1977}).
In order to quantify how ``synchronous'' a solution ${\bf x}=(x^1,x^2,x^3)$ is, we also plot $g_{\bf x}\in C([0,\infty),\R_+)$ defined by
	\be\label{eq:gxt}g_{\bf x}(t)=\sup\left\{|x^j(s)-x^k(s)|:j\ne k,\max(t-\period^\star,0)\le s\le t\right\},\qquad t\ge0,\ee
where we recall that $\period^\star$ is the period of the limiting SOPS $\sops^\star$ defined in Section \ref{sec:normalizedSOPS}.
In Figure \ref{fig:5}, we plot $g_{\bf x}$, where ${\bf x}$ is the solution of the coupled DDE \eqref{eq:cdde} associated with $(\frac18,\beta,f,R_{3,\kappa_{\ell}})$, for $\ell=1,\dots,6$ and different $\beta$, and initial condition $\mb{x}_0={\bs\phi}$ defined by
	$${\bs\phi}(t)=\beta t{\bf 1}_3+\frac{1}{10\sqrt{2}}(0,1,-1)^T,\qquad t\in[-1,0].$$
In Figures \ref{fig:3}--\ref{fig:8} we plot the solutions ${\bf x}$, which were generated using Mathematica's NDSolve function.
The plots suggest that in some cases the stability/instability established in our limit theorems appears to in fact hold for all $\beta>\beta_\text{Hopf}$, whereas in other cases $\beta$ must be chosen larger.
Based on Figures \ref{fig:5}--\ref{fig:3}, as well as additional numerical tests that are not included, we conclude with the following conjecture, whose proof is beyond the scope of this work.
Suppose $\sopsn$ is a synchronous SOPS of the coupled DDE \eqref{eq:cdde} associated with $(\alpha,\beta,f,G)$, where $\alpha\ge0$ and $f$ satisfies Assumption \ref{ass:main}.
We conjecture that $\sopsn$ is asymptotically stable with an exponential phase provided that $\beta>\beta_\text{Hopf}$, $f$ is \emph{monotone} and $\sigminus(G)\subset\stable_\star(0)$.
Based on Figures \ref{fig:5} and \ref{fig:4}--\ref{fig:8}, the case where $\sigma(G)\cap\unstable_\star(0)\ne\emptyset$ appears to be more complicated. 

\begin{figure}
	\begin{subfigure}{.49\textwidth}
		\centering
		\includegraphics[width=\textwidth]{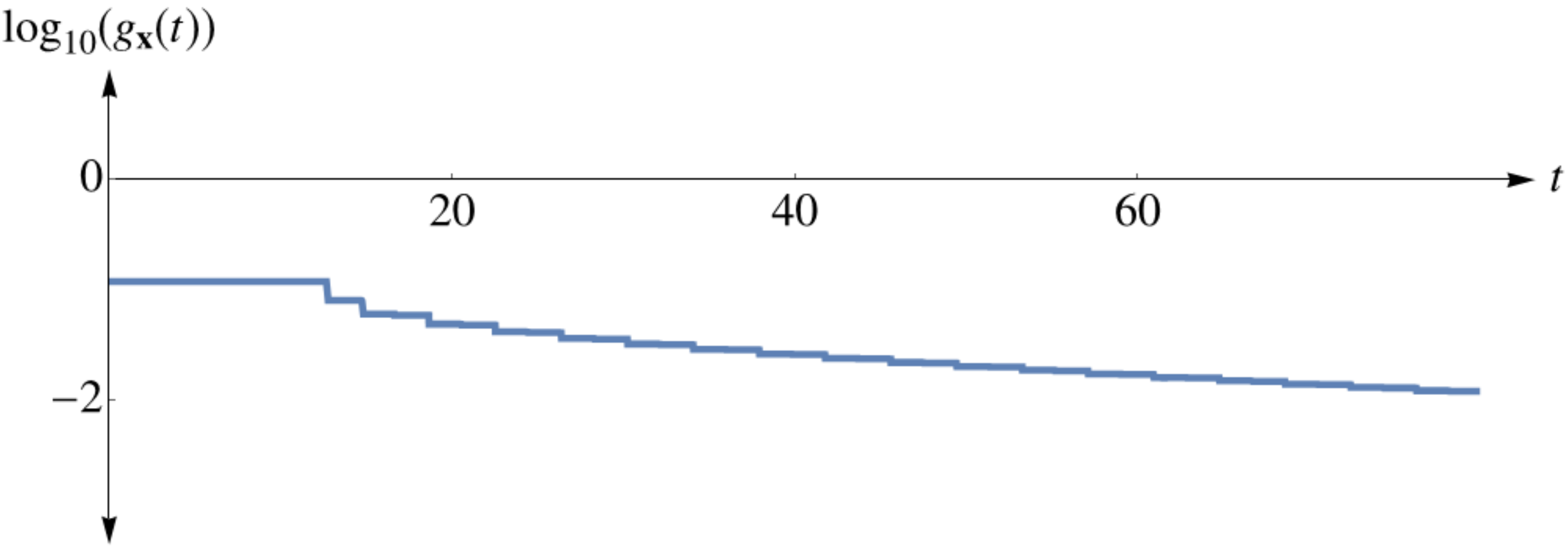}
		\caption{$\ell=1$, $\beta=\beta_\text{Hopf}+0.01$.}
	\end{subfigure}
	\begin{subfigure}{.49\textwidth}
		\centering
		\includegraphics[width=\textwidth]{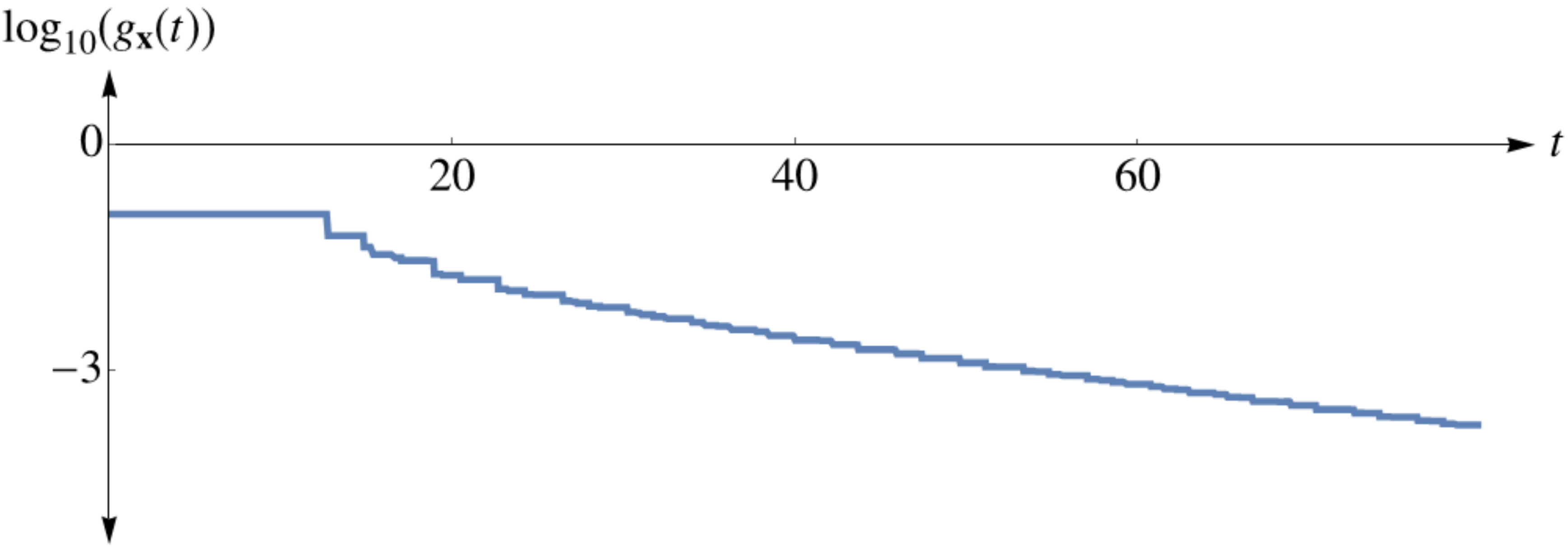}
		\caption{$\ell=2$, $\beta=\beta_\text{Hopf}+0.01$.}
	\end{subfigure}\newline\newline\newline
	\begin{subfigure}{.49\textwidth}
		\centering
		\includegraphics[width=\textwidth]{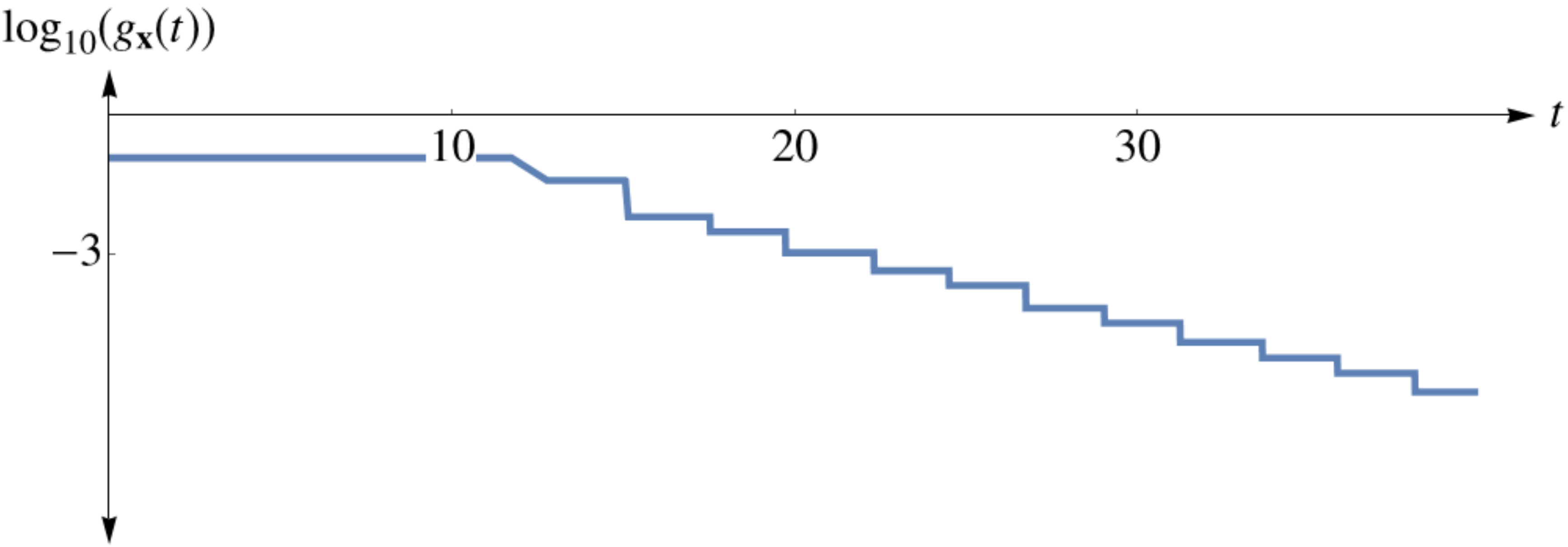}
		\caption{$\ell=3$, $\beta=\beta_\text{Hopf}+0.01$.}
	\end{subfigure}
	\begin{subfigure}{.49\textwidth}
		\centering
		\includegraphics[width=\textwidth]{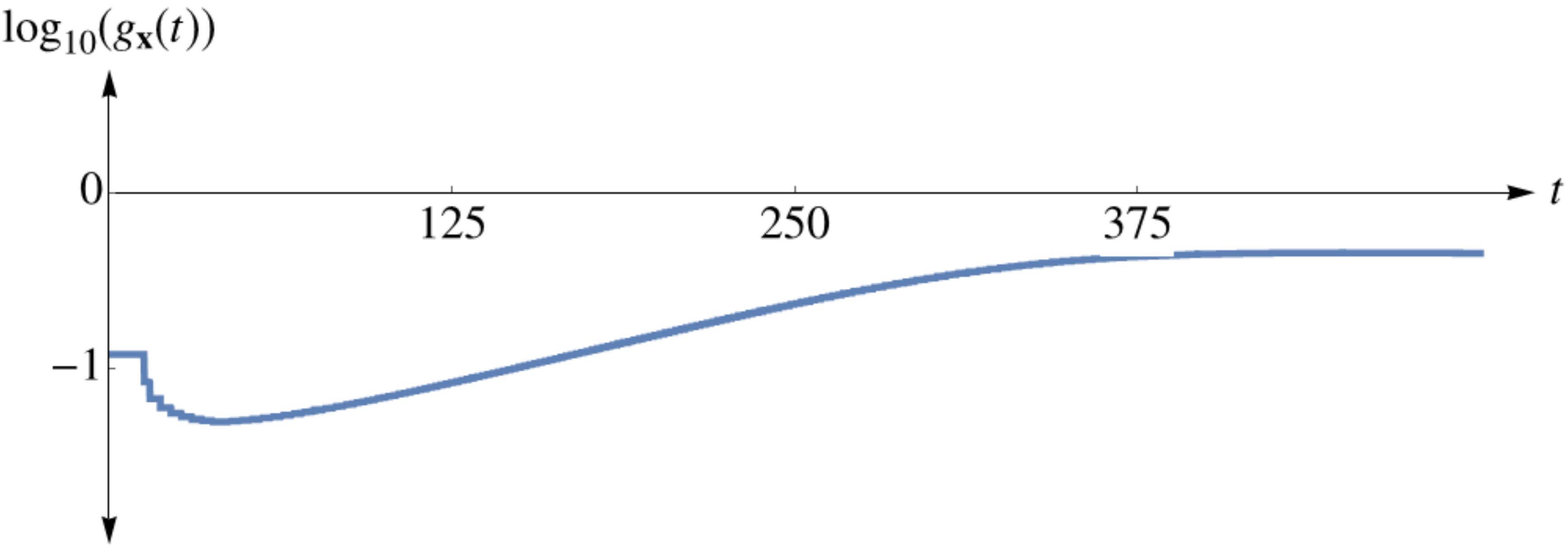}
		\caption{$\ell=4$, $\beta=\beta_\text{Hopf}+0.01$.}
	\end{subfigure}
	\begin{subfigure}{.49\textwidth}
		\centering
		\includegraphics[width=\textwidth]{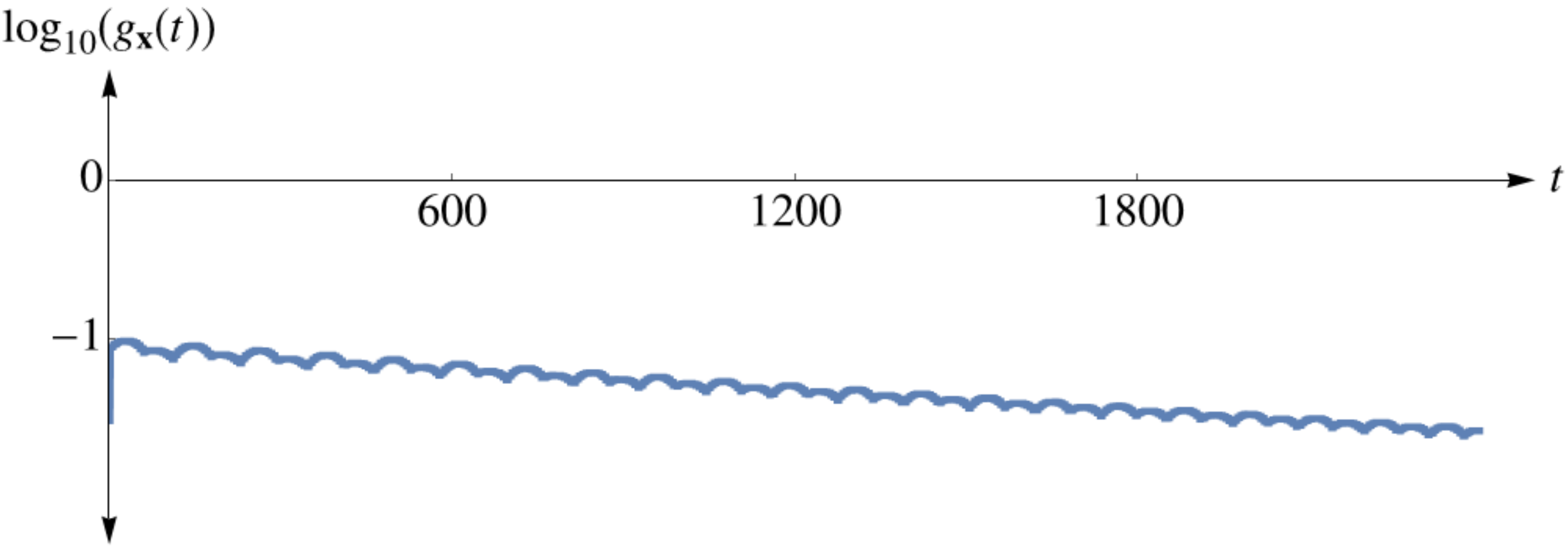}
		\caption{$\ell=5$, $\beta=6.6$.}
	\end{subfigure}
	\begin{subfigure}{.49\textwidth}
		\centering
		\includegraphics[width=\textwidth]{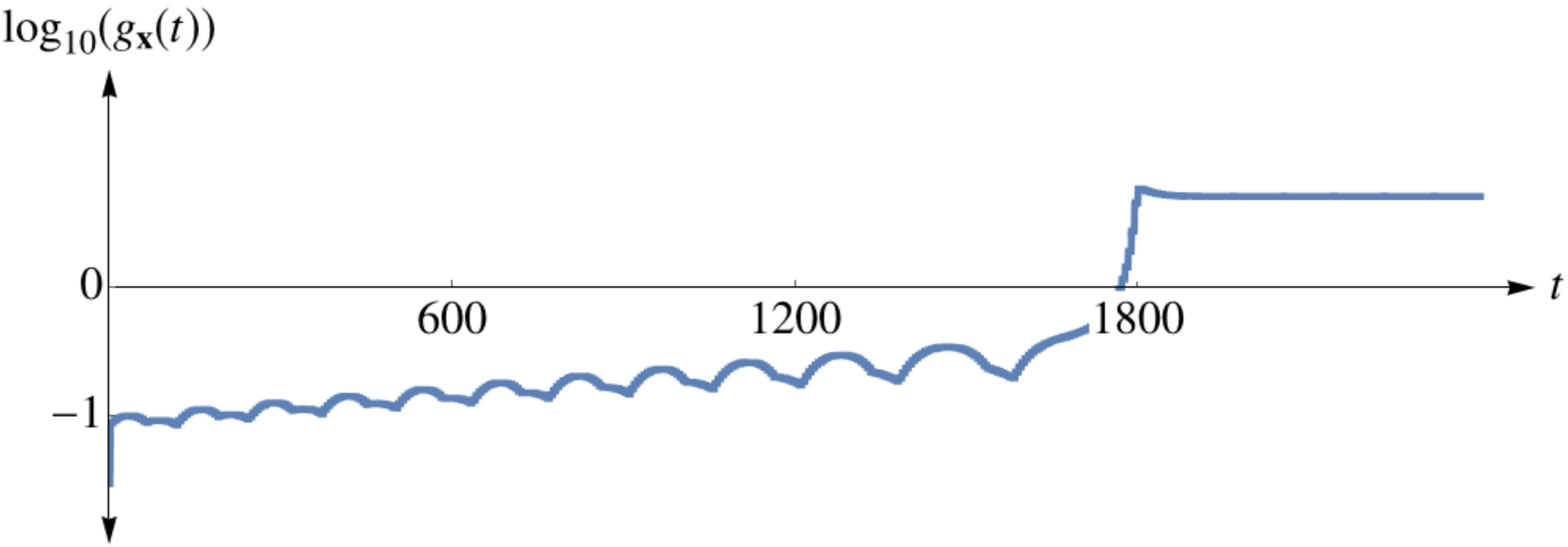}
		\caption{$\ell=5$, $\beta=6.7$.}
	\end{subfigure}
	\begin{subfigure}{.49\textwidth}
		\centering
		\includegraphics[width=\textwidth]{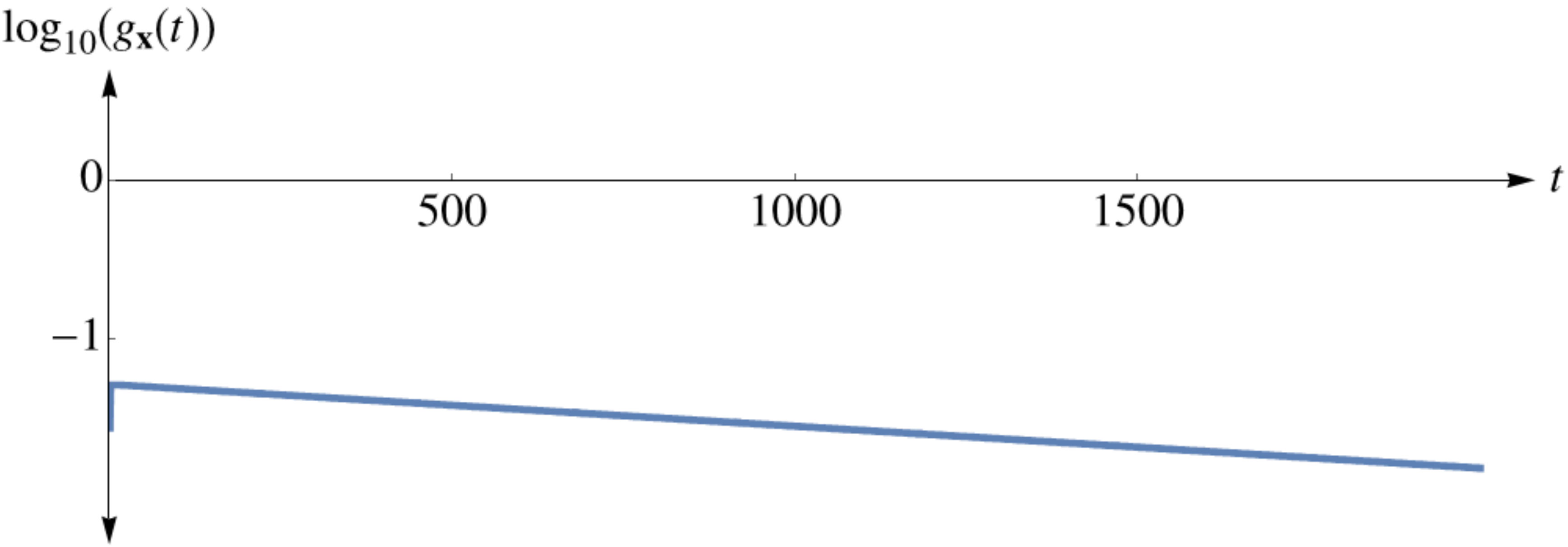}
		\caption{$\ell=6$, $\beta=7.35$.}
	\end{subfigure}
	\begin{subfigure}{.49\textwidth}
		\centering
		\includegraphics[width=\textwidth]{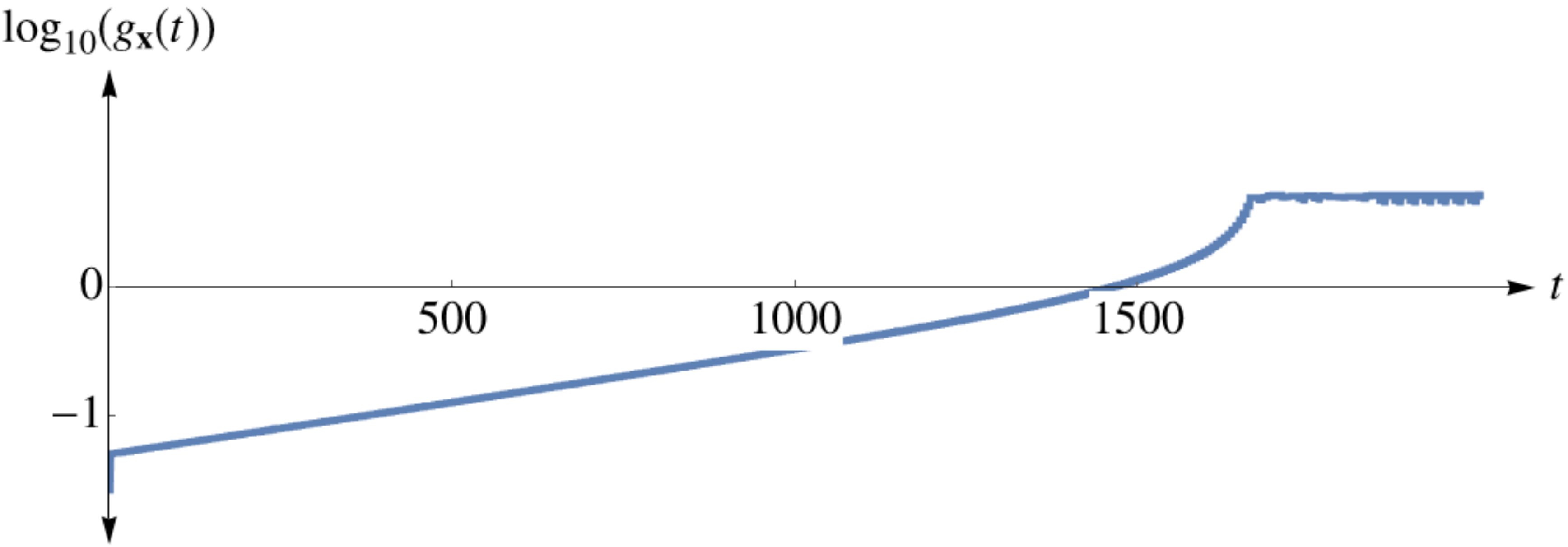}
		\caption{$\ell=6$, $\beta=7.5$.}
	\end{subfigure}
	\caption[]{Plots of $\log(g_{\bf x}(t))$. Here $g_{\bf x}(t)$ is the measure of ``synchrony'' defined in Equation \eqref{eq:gxt} and ${\bf x}$ is the solution of the coupled DDE \eqref{eq:cdde} associated with $(\frac18,\beta,f,R_{3,\kappa_\ell})$ starting at ${\bs\phi}$, for $\ell=1,\dots,6$ and varying $\beta$. Note that the axes of the plots are not identical. As stated in the text, $|\nu(\lambda_\ell)|=0.9$ for $\ell=1,2,3$, and $|\nu(\lambda_\ell)|=1.1$ for $\ell=4,5,6$. Therefore, by Theorem \ref{thm:1}, for $\beta$ sufficiently large, the synchronous SOPS of the coupled DDE \eqref{eq:cdde} associated with $(\frac18,\beta,f,R_{3,\kappa_\ell})$ is asymptotically stable with an exponential phase (resp.\ unstable) for $\ell=1,2,3$ (resp.\ $\ell=4,5,6$). For $\ell=1,2,3$ and $\beta=\beta_\text{Hopf}+0.01$, we see that ${\bf x}$ appears to converge to a synchronous SOPS of the coupled DDE \eqref{eq:cdde}, suggesting that asymptotic stability in fact holds for all $\beta>\beta_\text{Hopf}\approx1.65$. For $\ell=4$ and $\beta=\beta_\text{Hopf}+0.01$, we see that the synchronous SOPS of the coupled DDE \eqref{eq:cdde} appears to be unstable, suggesting that instability holds for all $\beta>\beta_\text{Hopf}$. On the other hand, for $\ell=5,6$, we see that there appears to be a critical $\beta$ where the synchronous SOPS of the coupled DDE \eqref{eq:cdde} transitions from asymptotically stable to unstable.}
	\label{fig:5}
\end{figure}

\begin{figure}
	\begin{subfigure}{.49\textwidth}
		\centering
		\includegraphics[width=\textwidth]{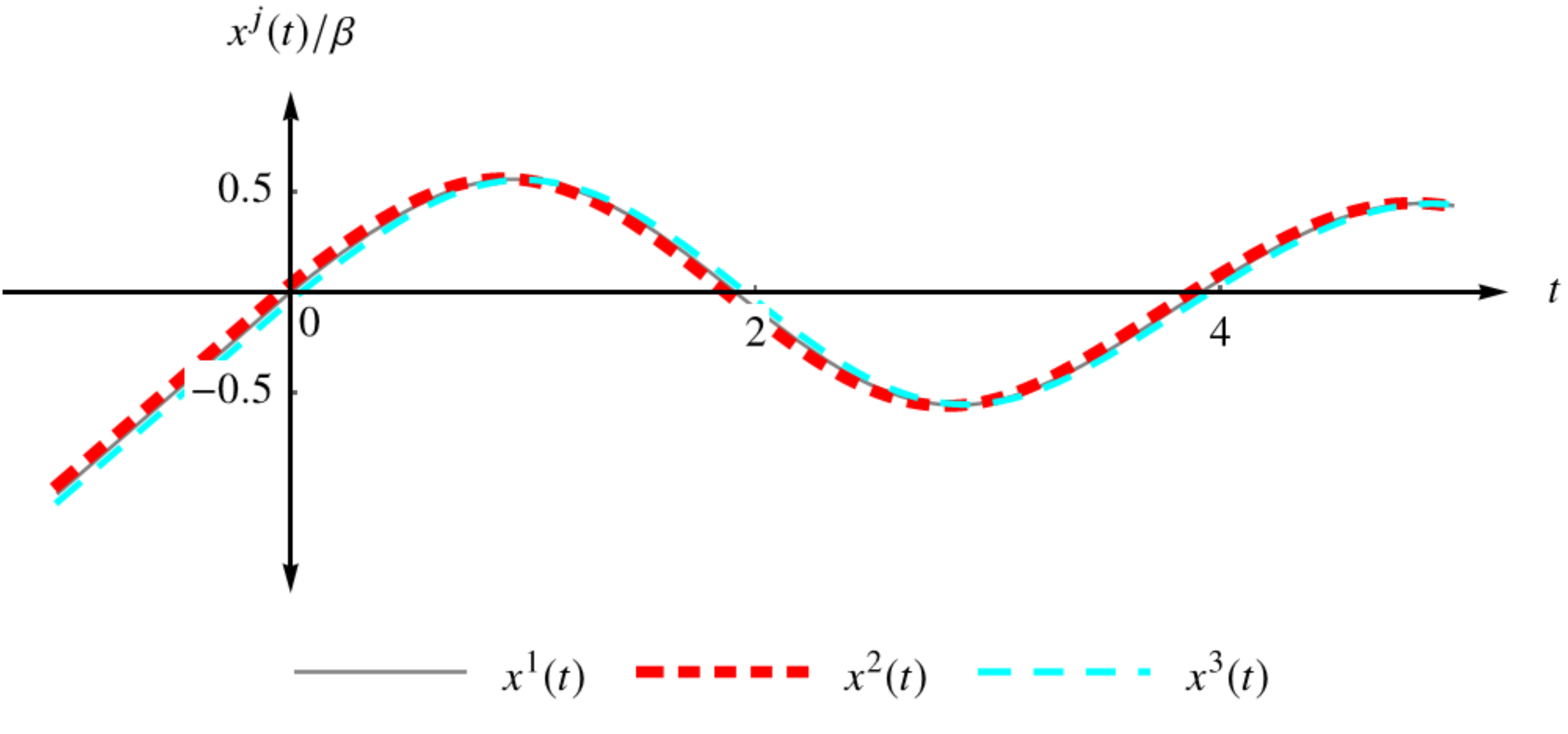}
		\caption{$\ell=1$, $\beta=\beta_\text{Hopf}+0.01$}
	\end{subfigure}
	\begin{subfigure}{.49\textwidth}
		\centering
		\includegraphics[width=\textwidth]{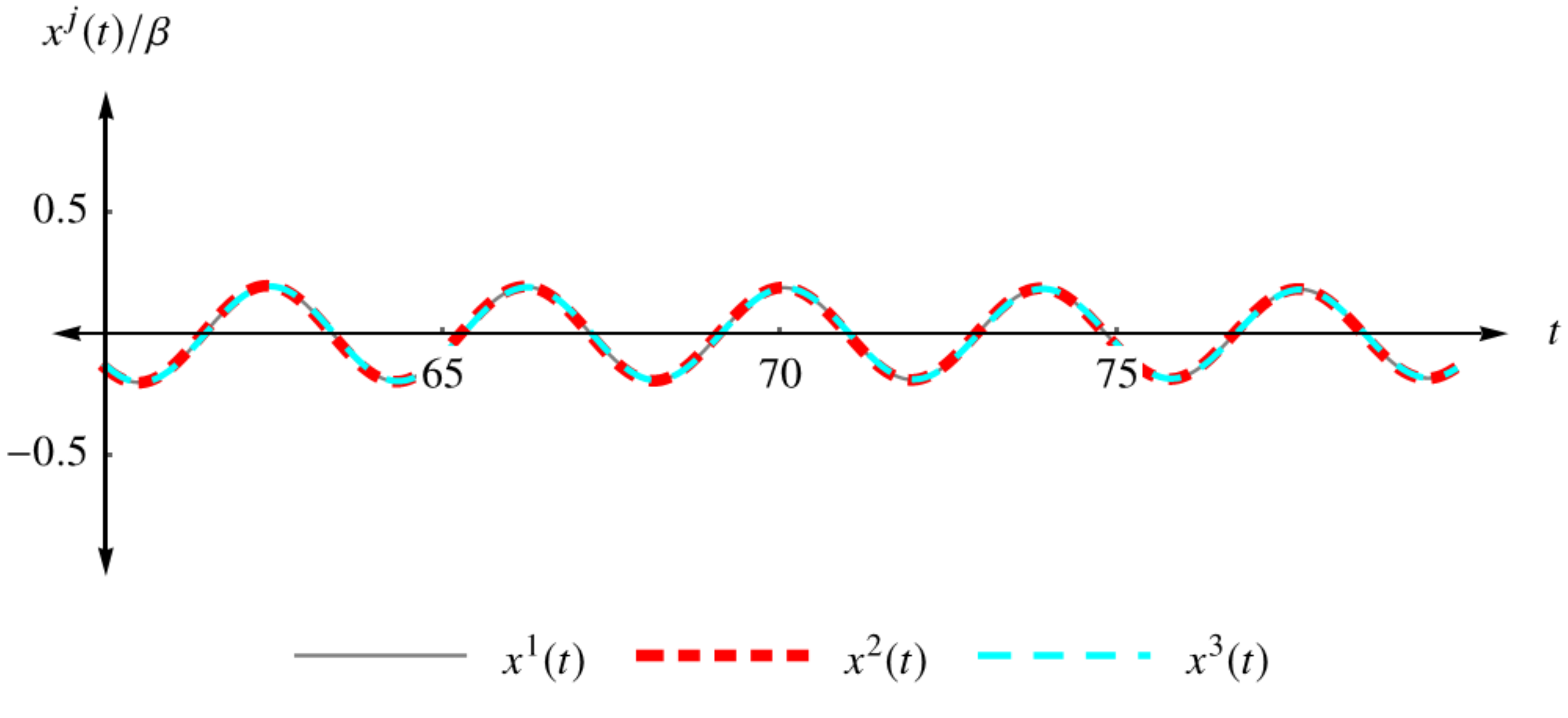}
		\caption{$\ell=1$, $\beta=\beta_\text{Hopf}+0.01$}
	\end{subfigure}\newline\newline\newline
	\begin{subfigure}{.49\textwidth}
		\centering
		\includegraphics[width=\textwidth]{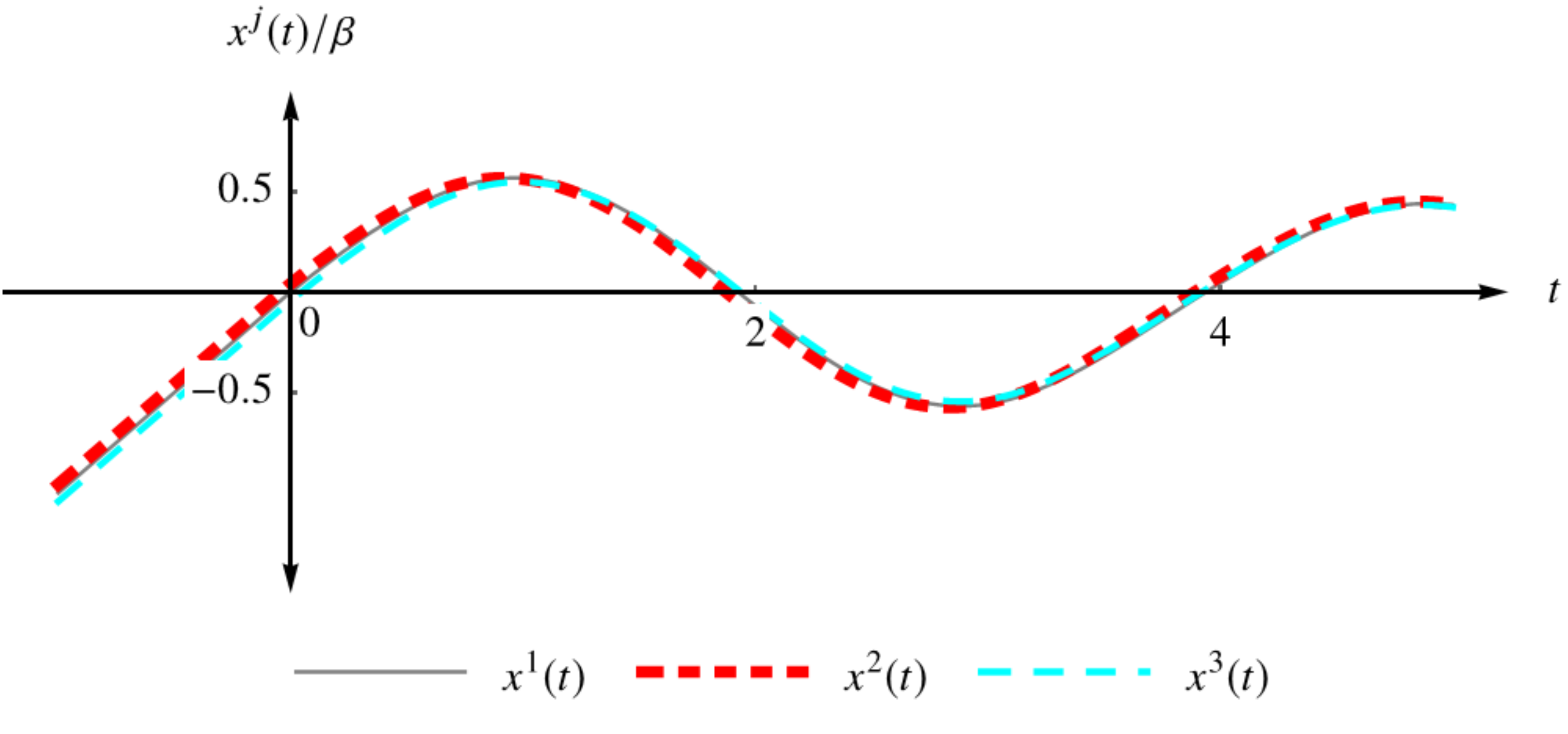}
		\caption{$\ell=2$, $\beta=\beta_\text{Hopf}+0.01$}
	\end{subfigure}
	\begin{subfigure}{.49\textwidth}
		\centering
		\includegraphics[width=\textwidth]{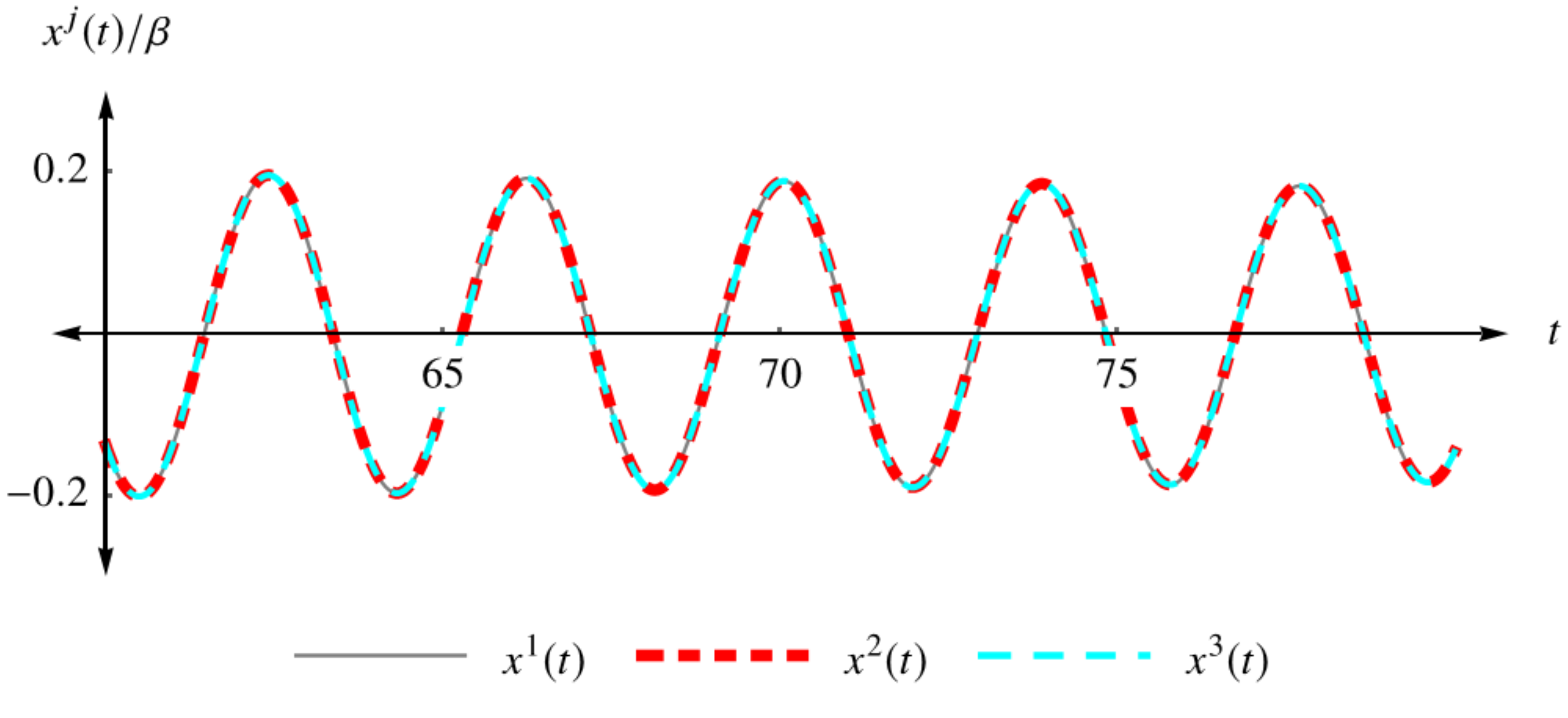}
		\caption{$\ell=2$, $\beta=\beta_\text{Hopf}+0.01$}
	\end{subfigure}\newline\newline\newline
	\begin{subfigure}{.49\textwidth}
		\centering
		\includegraphics[width=\textwidth]{Figure-3D_ring3a.pdf}
		\caption{$\ell=3$, $\beta=\beta_\text{Hopf}+0.01$}
	\end{subfigure}
	\begin{subfigure}{.49\textwidth}
		\centering
		\includegraphics[width=\textwidth]{Figure-3D_ring3c.pdf}
		\caption{$\ell=3$, $\beta=\beta_\text{Hopf}+0.01$}
	\end{subfigure}
	\caption[]{Plots of the solutions ${\bf x}$ of the coupled DDE \eqref{eq:cdde} associated with $(\frac18,\beta_\text{Hopf}+0.01,f,R_{3,\kappa_\ell})$ starting at ${\bs\phi}$, for $\ell=1,2,3$ over the time intervals $[-1,5]$ and $[60,80]$. In this case, $\sigma_{-1}(R_{3,\kappa_\ell})=\{\lambda_\ell,\bar\lambda_\ell\}$ and $|\nu_\star(\lambda_\ell)|=0.9$. Thus, according to Theorem \ref{thm:1}, for $\beta$ sufficiently large, the synchronous SOPS of the coupled DDE \eqref{eq:cdde} associated with $(\frac18,\beta,f,R_{3,\kappa_\ell})$ is asymptotically stable with an exponential phase. These plots suggest that, in each case, the synchronous SOPS is asymptotically stable for all $\beta>\beta_\text{Hopf}$.}
	\label{fig:3}
\end{figure}

\begin{figure}
	\begin{subfigure}{\textwidth}
		\centering
		\includegraphics[width=.5\textwidth]{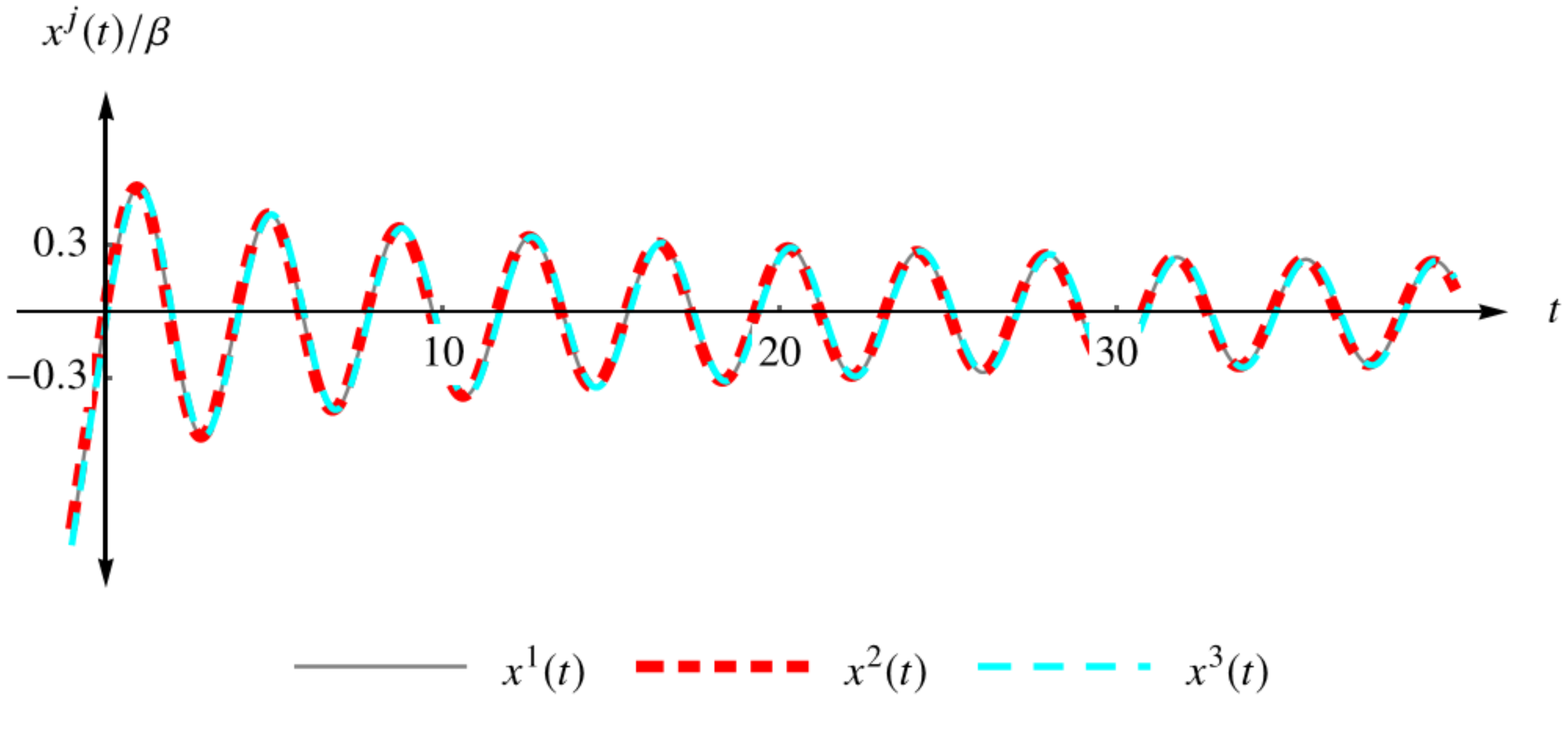}
		\caption{$\ell=4$, $\beta=\beta_\text{Hopf}+0.01$}
	\end{subfigure}\newline\newline\newline
	\begin{subfigure}{\textwidth}
		\centering
		\includegraphics[width=.5\textwidth]{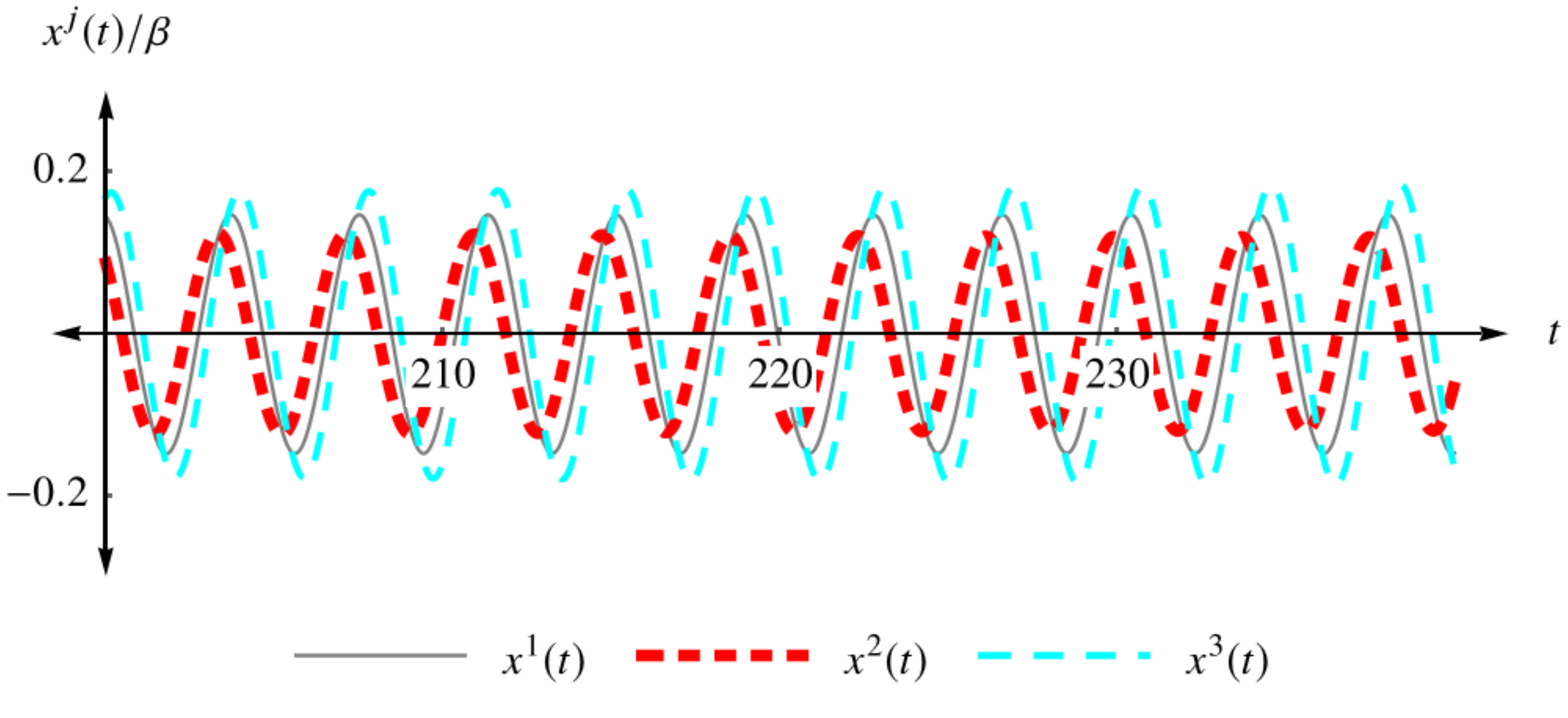}
		\caption{$\ell=4$, $\beta=\beta_\text{Hopf}+0.01$}
	\end{subfigure}\newline\newline\newline
	\begin{subfigure}{\textwidth}
		\centering
		\includegraphics[width=.5\textwidth]{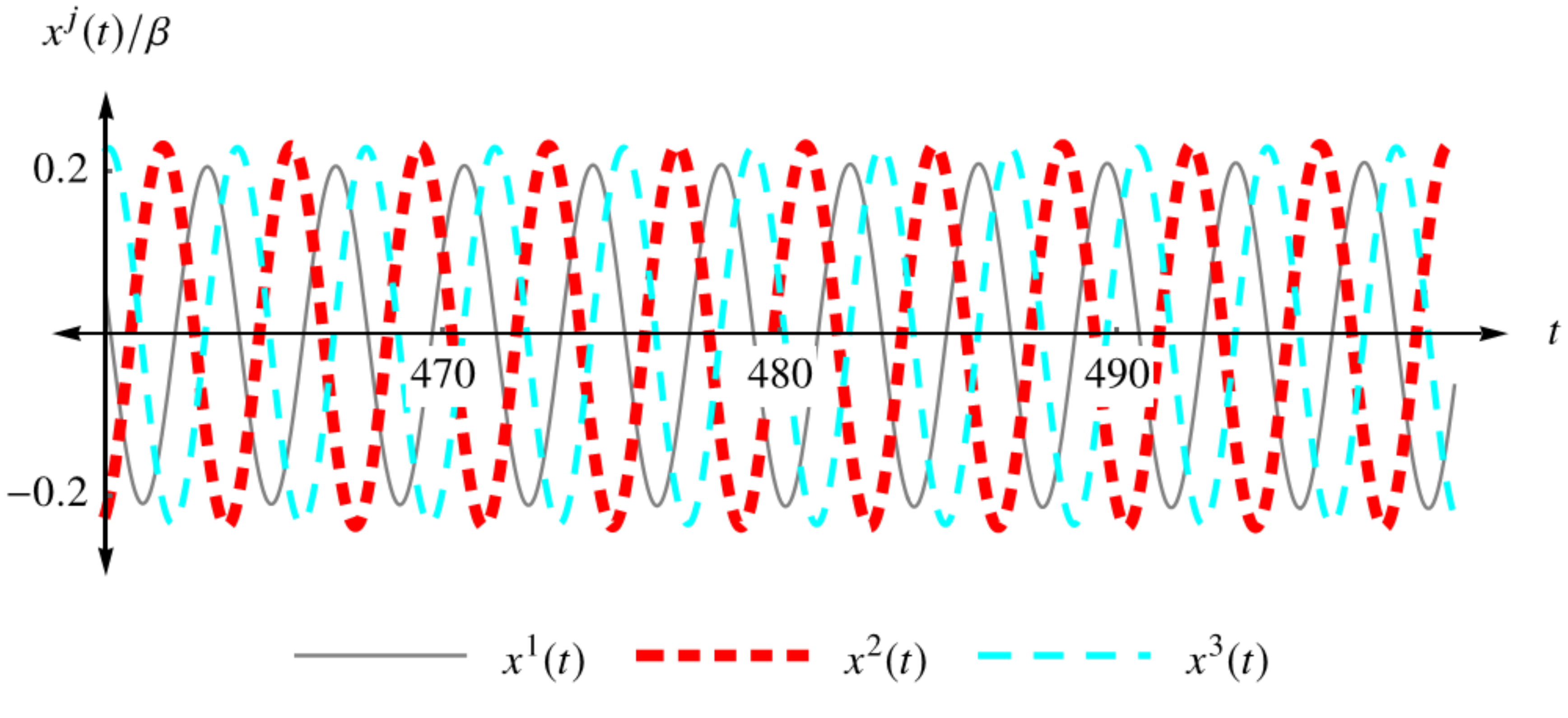}
		\caption{$\ell=4$, $\beta=\beta_\text{Hopf}+0.01$}
	\end{subfigure}%
	\caption[]{Plots of the solution ${\bf x}$ of the coupled DDE \eqref{eq:cdde} associated with $(\frac18,\beta_\text{Hopf}+0.01,f,R_{3,\kappa_4})$ starting at ${\bs\phi}$, over the time intervals $[-1,40]$, $[200,240]$ and $[460,500]$. In this case, $\sigma_{-1}(R_{3,\kappa_4})=\{\lambda_4,\bar\lambda_4\}$ and $\nu_\star(\lambda_4)=1.1$. Thus, according to Theorem \ref{thm:1}, for $\beta$ sufficiently large, the synchronous SOPS of the coupled DDE \eqref{eq:cdde} associated with $(\frac18,\beta,f,R_{3,\kappa_4})$ is unstable. These plots (and additional plots that are not shown) suggest the synchronous SOPS is unstable for all $\beta>\beta_\text{Hopf}$.}
	\label{fig:4}
\end{figure}

\begin{figure}
	\begin{subfigure}{.49\textwidth}
		\centering
		\includegraphics[width=\textwidth]{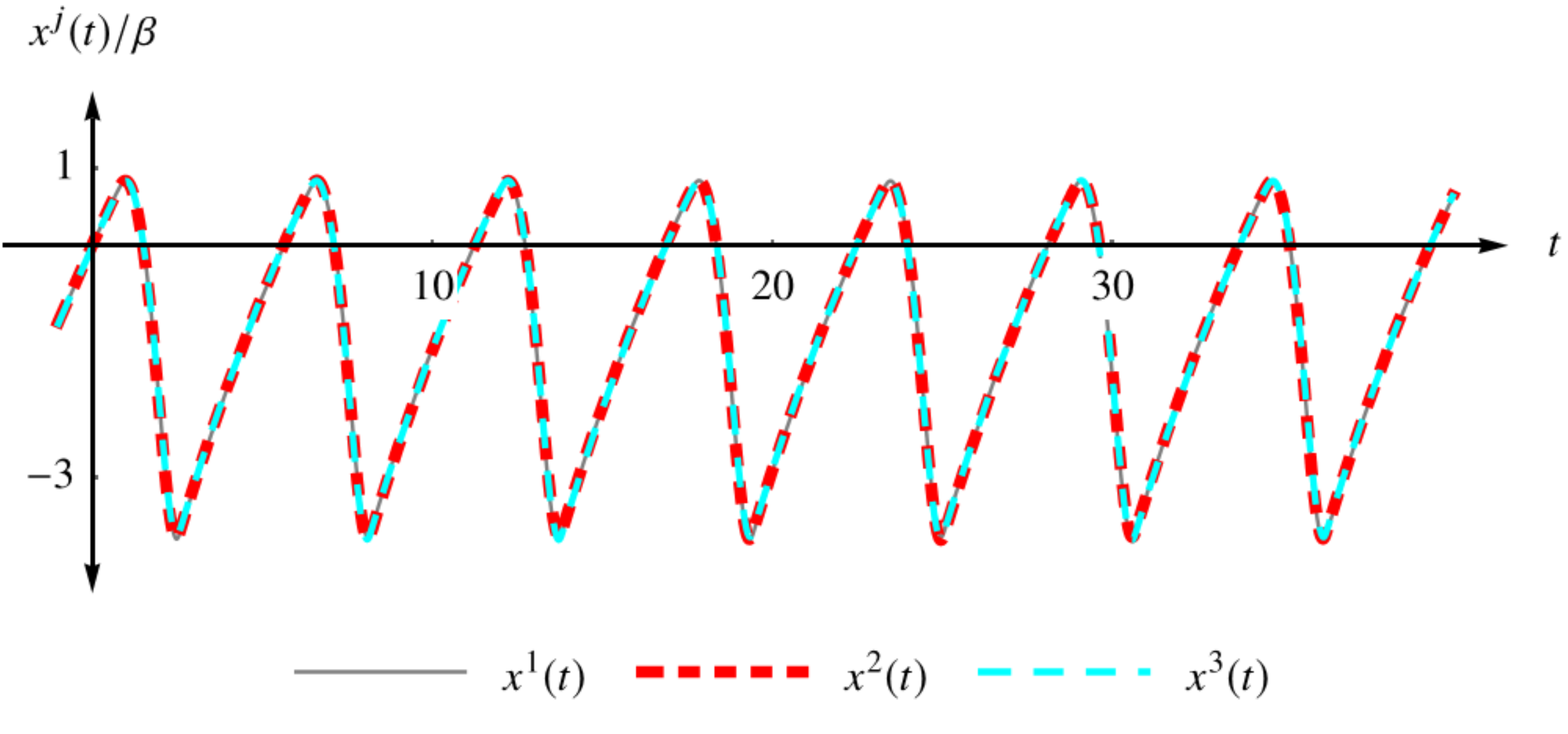}
		\caption{$\ell=5$, $\beta=6.6$.}
	\end{subfigure}
	\begin{subfigure}{.49\textwidth}
		\centering
		\includegraphics[width=\textwidth]{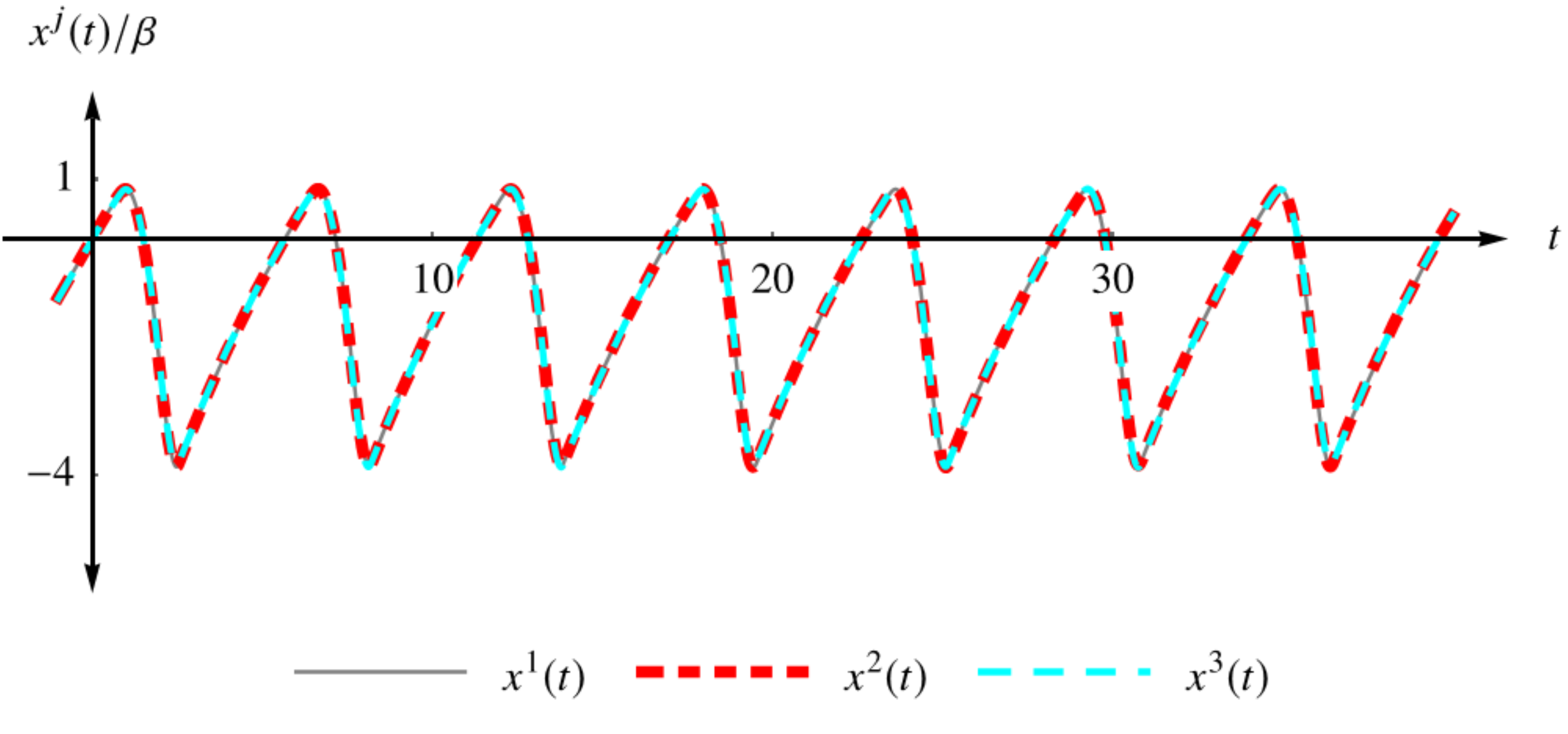}
		\caption{$\ell=5$, $\beta=6.7$}
	\end{subfigure}\newline\newline\newline
    \begin{subfigure}{.49\textwidth}
        \centering
        \includegraphics[width=\textwidth]{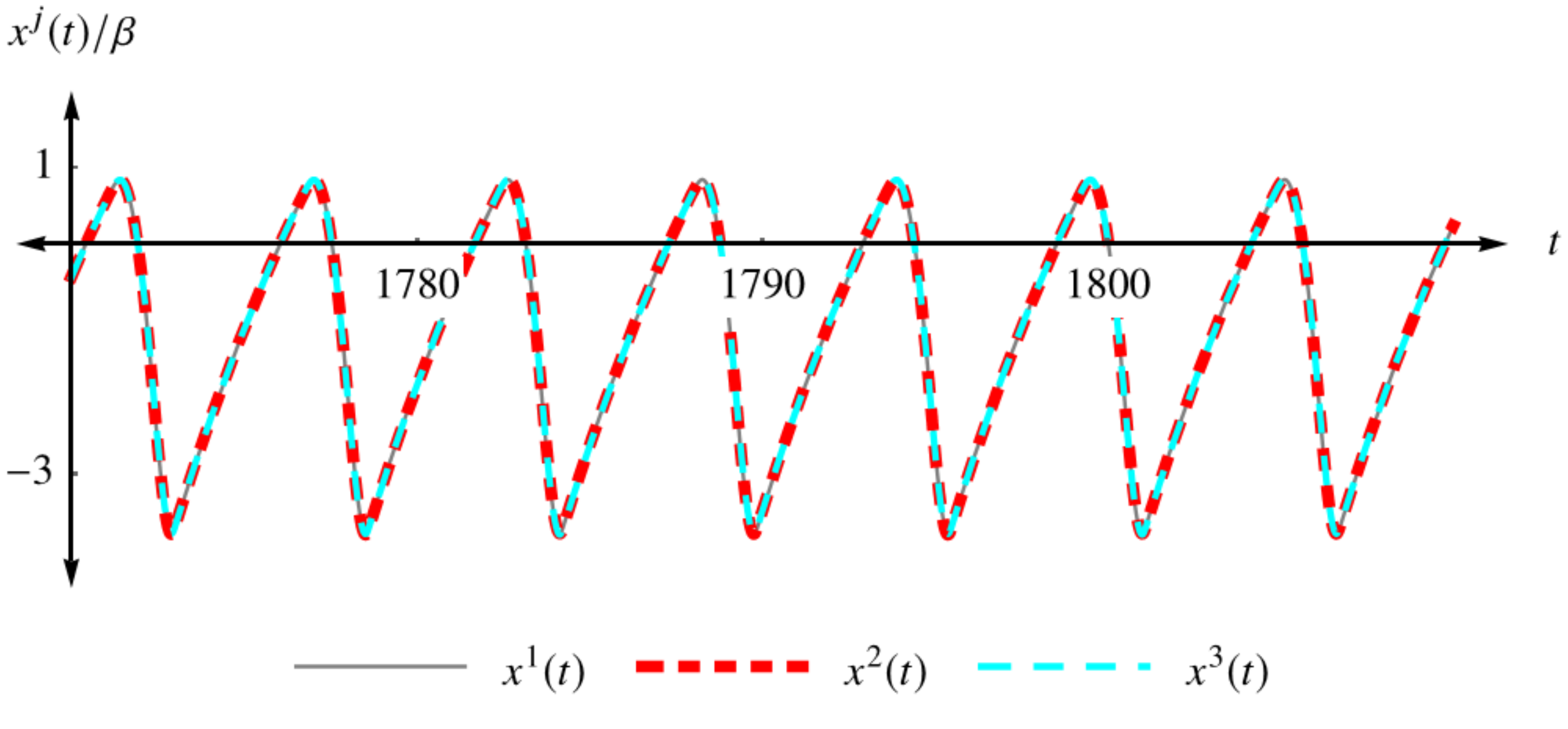}
        \caption{$\ell=5$, $\beta=6.6$}
    \end{subfigure}
	\begin{subfigure}{.49\textwidth}
		\centering
		\includegraphics[width=\textwidth]{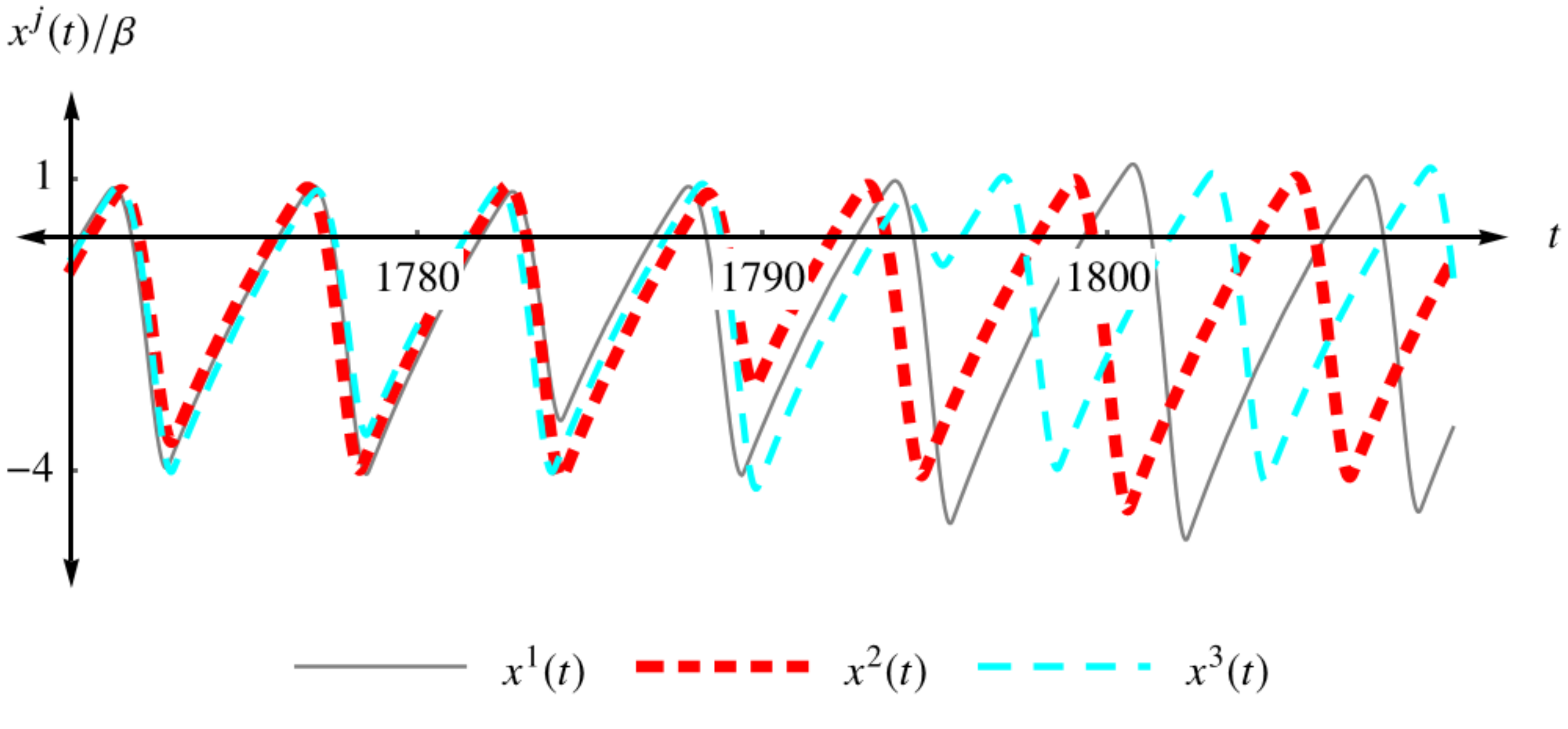}
		\caption{$\ell=5$, $\beta=6.7$}
	\end{subfigure}\newline\newline\newline
    \begin{subfigure}{.49\textwidth}
        \centering
        \includegraphics[width=\textwidth]{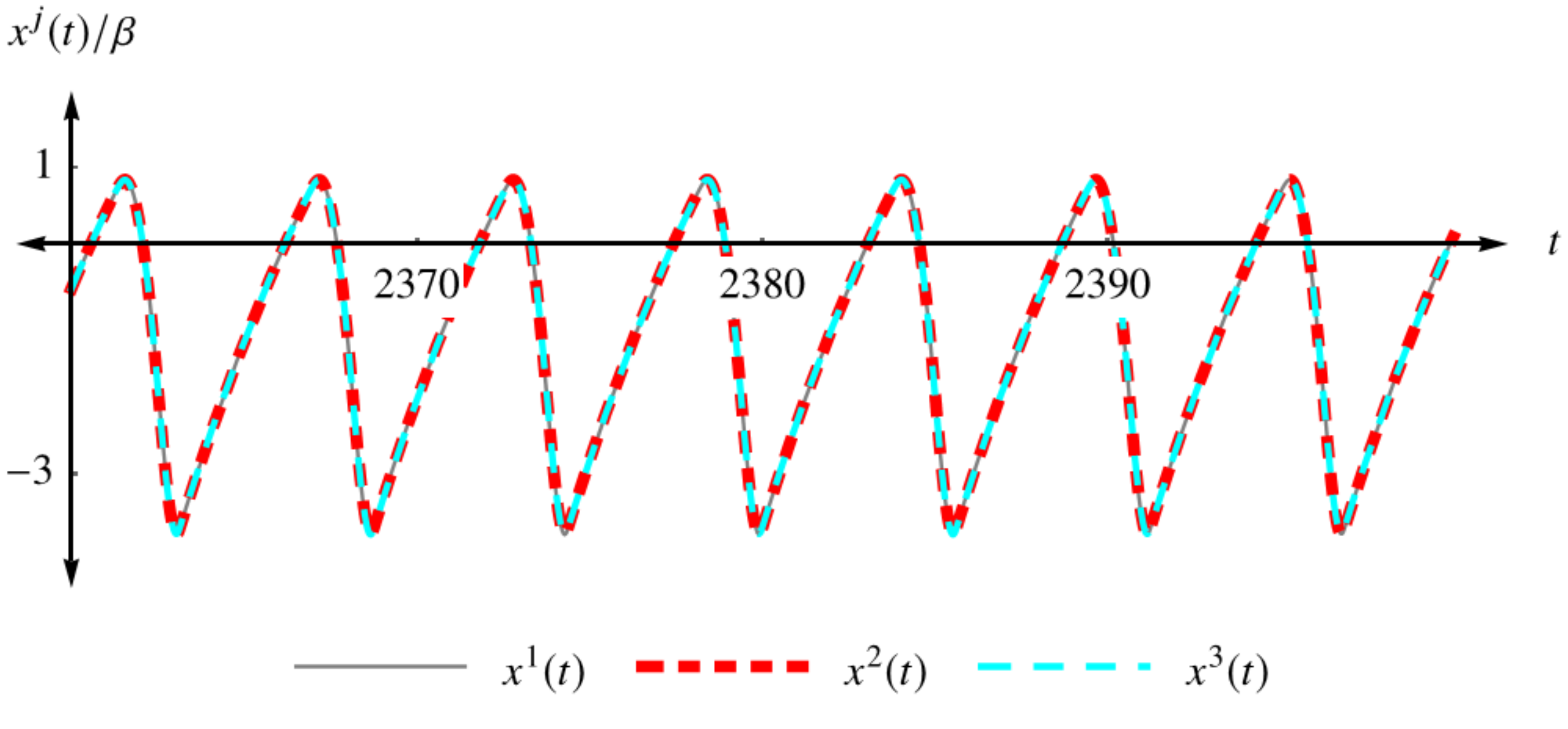}
        \caption{$\ell=5$, $\beta=6.6$}
    \end{subfigure}
	\begin{subfigure}{.49\textwidth}
		\centering
		\includegraphics[width=\textwidth]{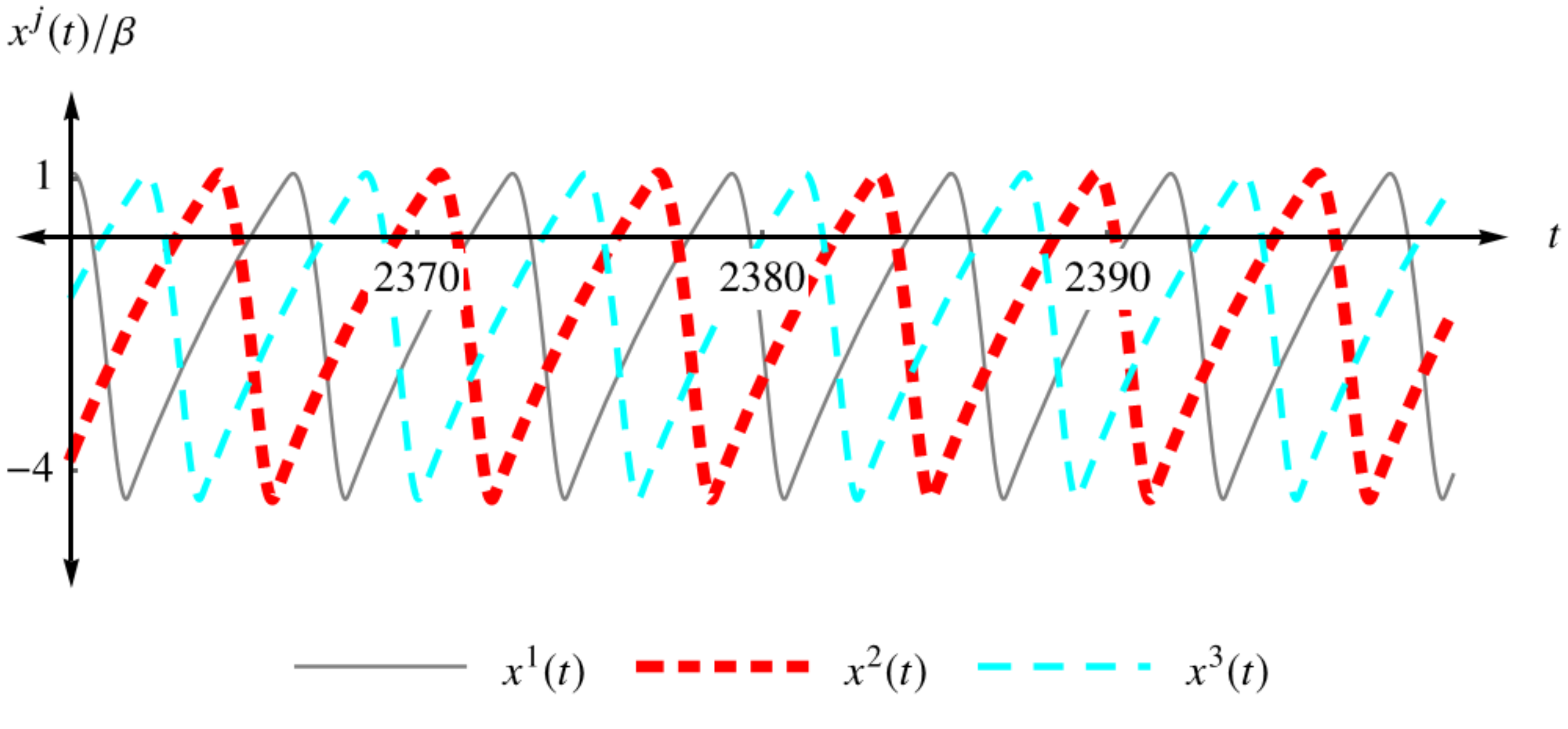}
		\caption{$\ell=5$, $\beta=6.7$}
	\end{subfigure}%
	\caption[]{
        Plots of the solutions ${\bf x}$ of the coupled DDE \eqref{eq:cdde} associated with $(\frac18,\beta,f,R_{3,\kappa_5})$ starting at ${\bs\phi}$, for $\beta=6.6,6.7$, over the time intervals $[-1,40]$, $[1770,1810]$ and $[2360,2400]$. The plots suggest there is a critical value of $\beta\in(6.6,6.7)$ such that the synchronous SOPS of the 3-dimensional coupled DDE \eqref{eq:cdde} associated with $(\frac18,\beta,f,R_{3,\kappa_5})$ is asymptotically stable with an exponential phase for $\beta$ less than the critical value and unstable for $\beta$ greater than the critical value. For comparison, $\beta_\text{Hopf}\approx1.65$.}
	\label{fig:2}
\end{figure}

\begin{figure}
	\begin{subfigure}{.49\textwidth}
		\centering
		\includegraphics[width=\textwidth]{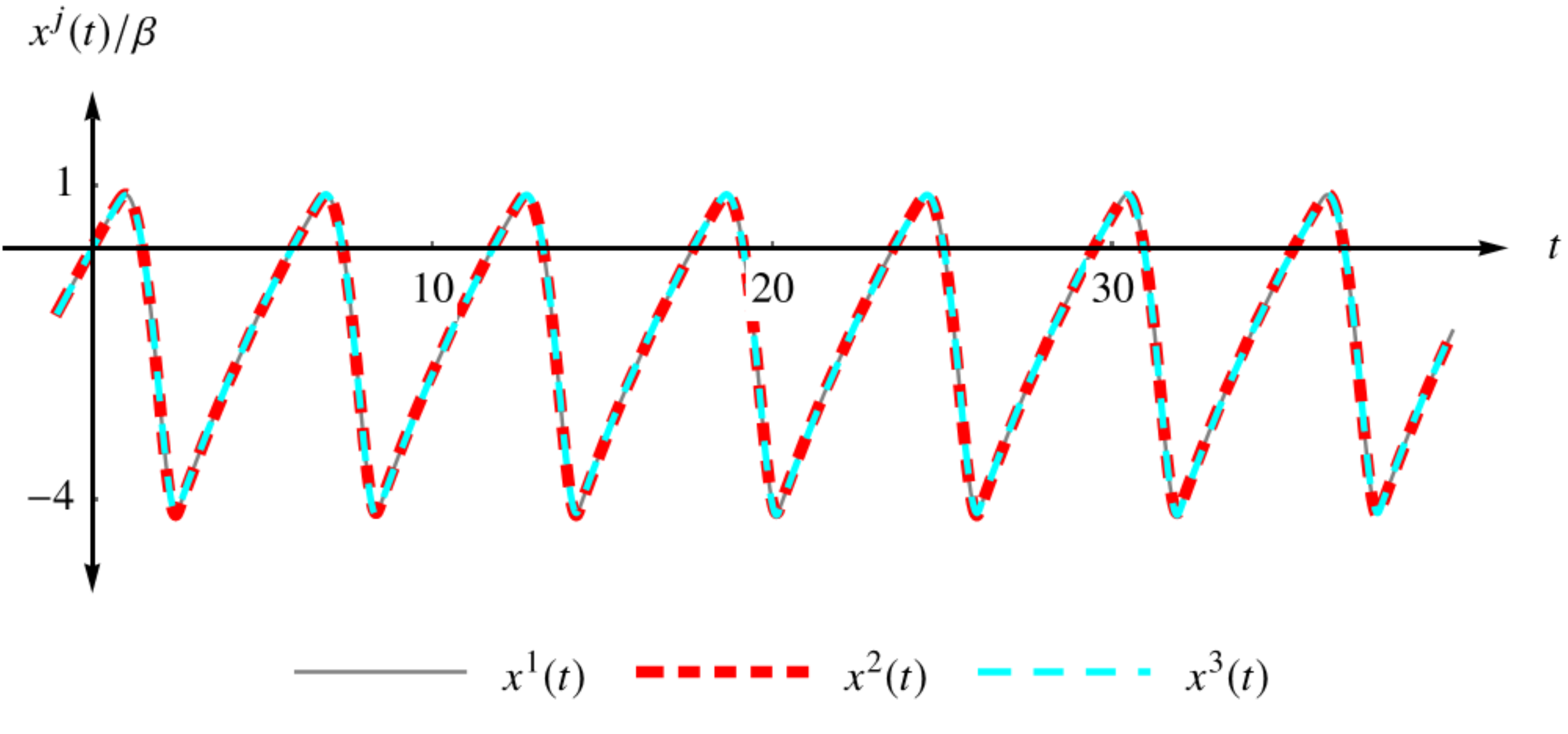}
		\caption{$\ell=6$, $\beta=7.35$.}
	\end{subfigure}\hfill
	\begin{subfigure}{.49\textwidth}
		\centering
		\includegraphics[width=\textwidth]{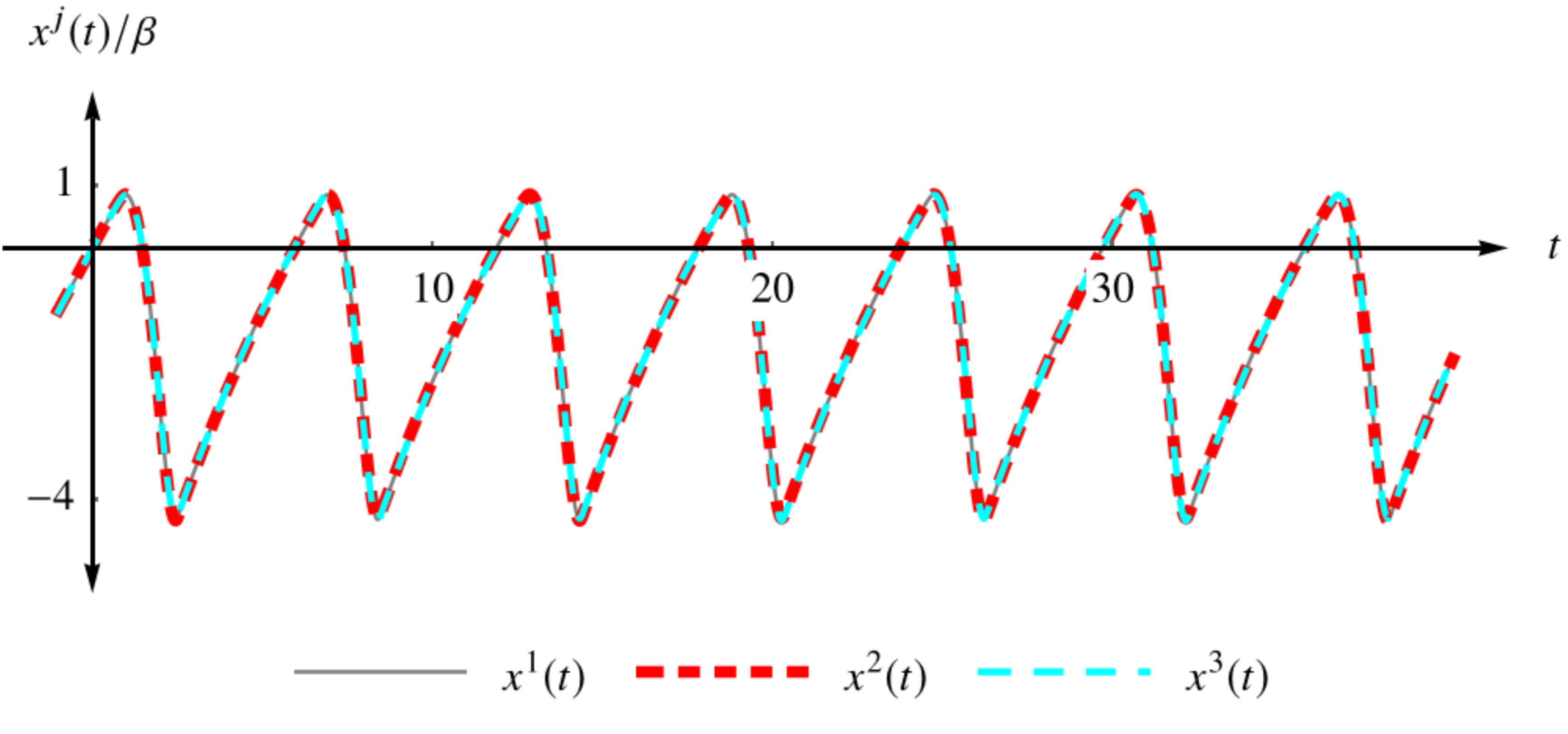}
		\caption{$\ell=6$, $\beta=7.5$.}
	\end{subfigure}\newline\newline\newline
    \begin{subfigure}{.49\textwidth}
    \centering
    \includegraphics[width=\textwidth]{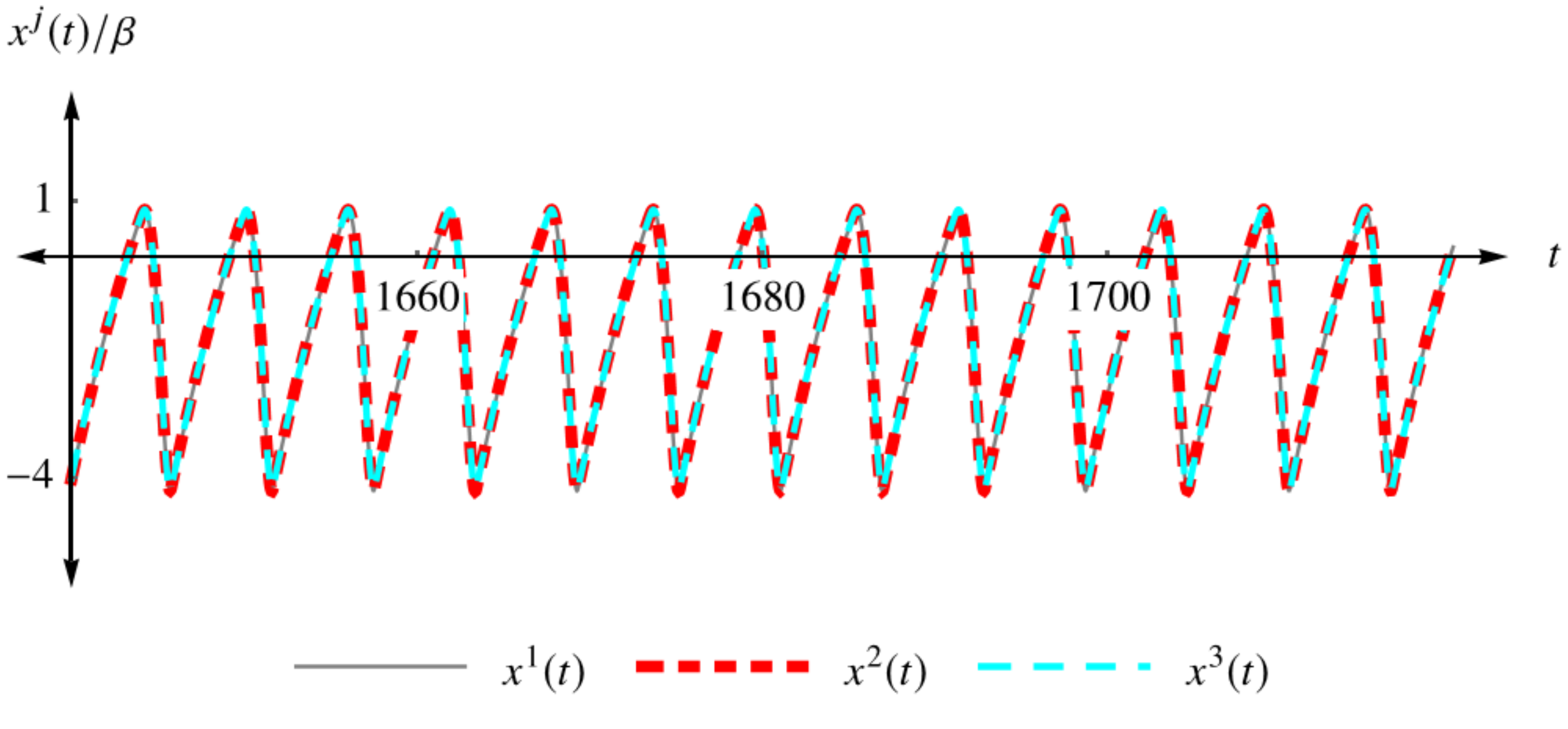}
    \caption{$\ell=6$, $\beta=7.35$.}
    \end{subfigure}
	\begin{subfigure}{.49\textwidth}
		\centering
		\includegraphics[width=\textwidth]{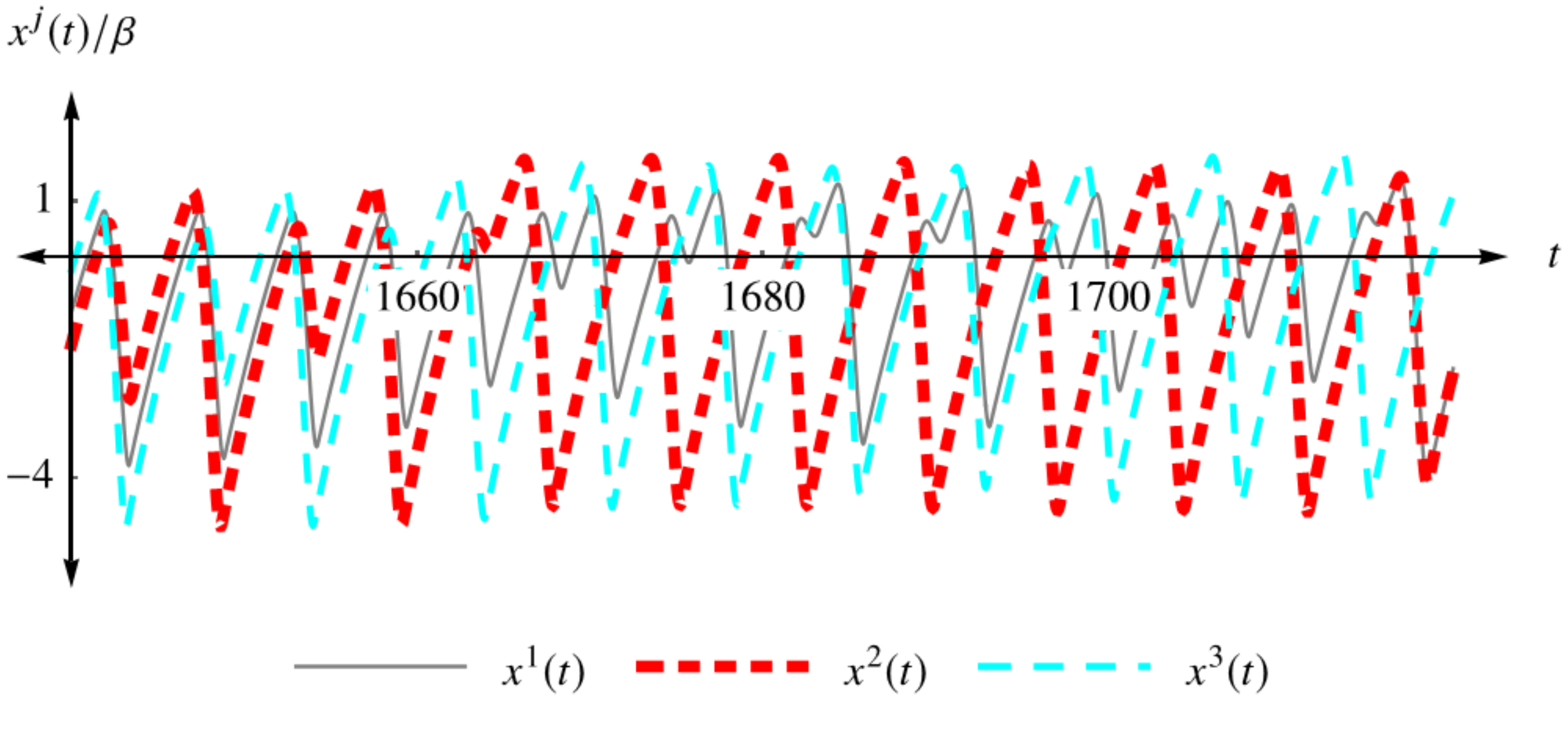}
		\caption{$\ell=6$, $\beta=7.5$.}
	\end{subfigure}\newline\newline\newline
\begin{subfigure}{.49\textwidth}
    \centering
    \includegraphics[width=\textwidth]{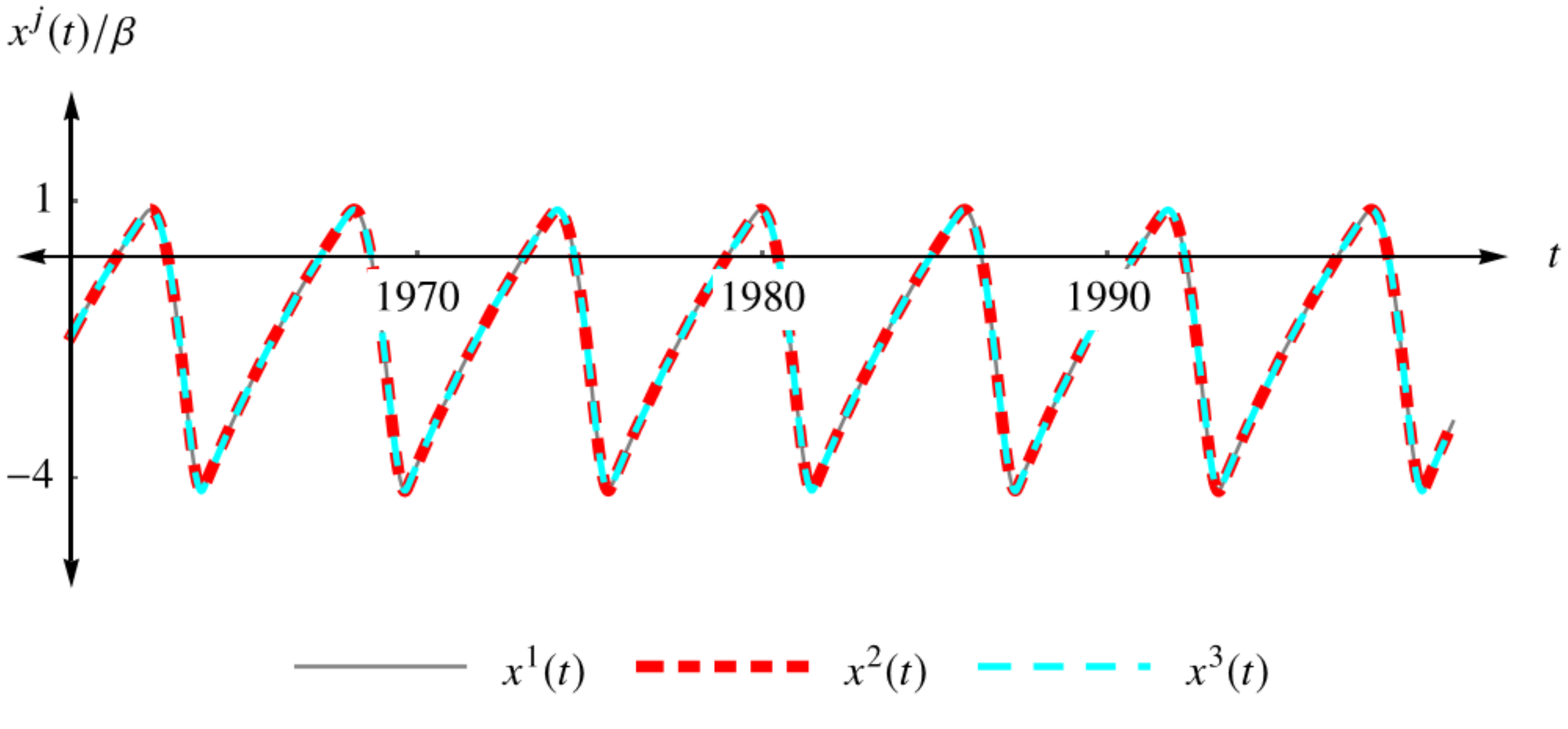}
    \caption{$\ell=6$, $\beta=7.35$.}
\end{subfigure}
	\begin{subfigure}{.49\textwidth}
		\centering
		\includegraphics[width=\textwidth]{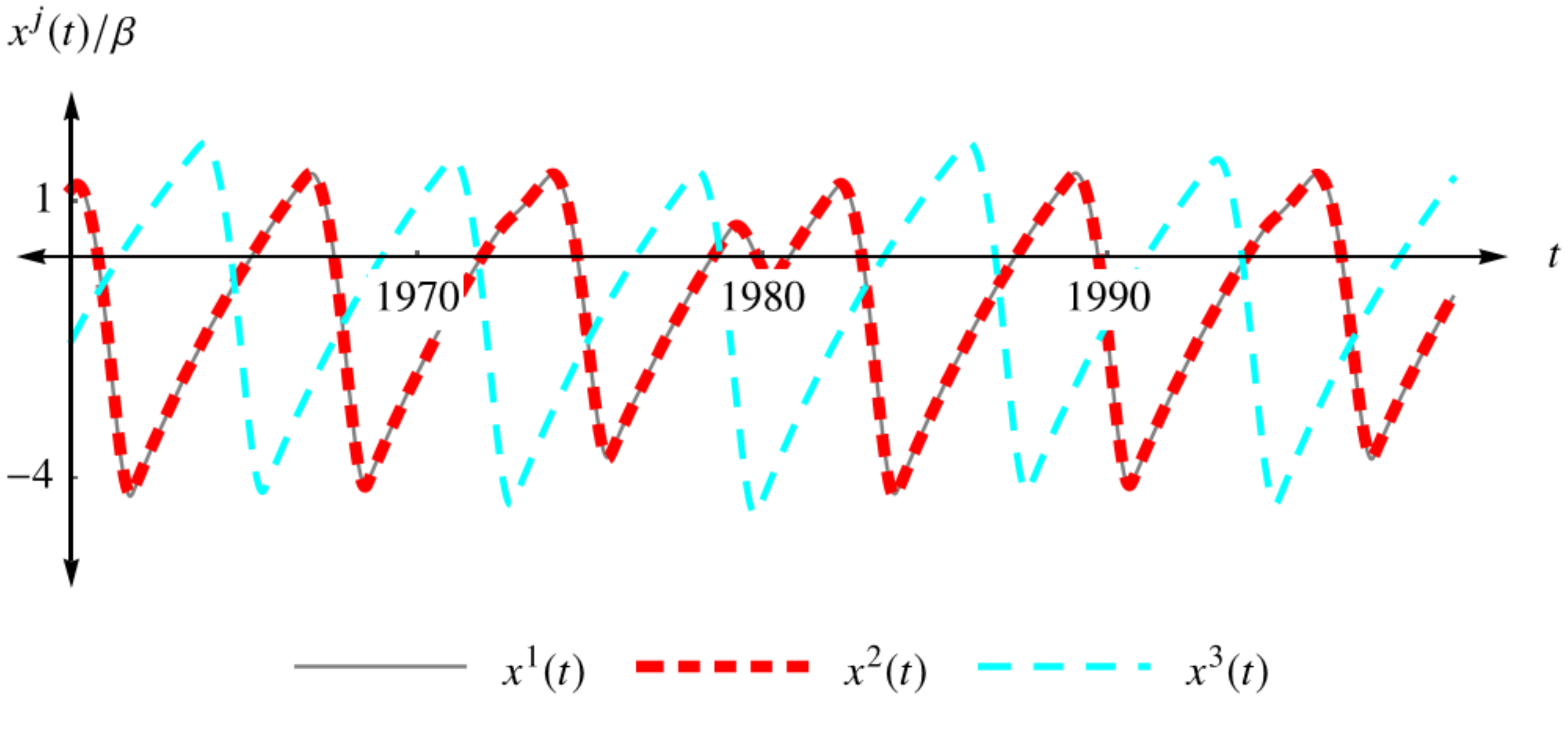}
		\caption{$\ell=6$, $\beta=7.5$.}
	\end{subfigure}
	\caption[]{
        Plots of the solutions ${\bf x}$ of the coupled DDE \eqref{eq:cdde} associated with $(\frac18,\beta,f,R_{3,\kappa_6})$ starting at ${\bs\phi}$, for $\beta=7.35,7.5$, over the time intervals $[-1,40]$, $[1640,1720]$ and $[1960,2000]$. The plots suggest there is a critical value of $\beta\in(7.35,7.5)$ such that the synchronous SOPS of the 3-dimensional coupled DDE \eqref{eq:cdde} associated with $(\frac18,\beta,f,R_{3,\kappa_\ell})$ is asymptotically stable with an exponential phase for $\beta$ less than the critical value and unstable for $\beta$ greater than the critical value. For comparison, $\beta_\text{Hopf}\approx1.65$.}
	\label{fig:8}
\end{figure}

\clearpage

\appendix

\section{Functional analysis results}
\label{apdx:meansfieldlemma}

The proofs of Lemma \ref{lem:DPhi} and Corollary \ref{cor:mu_eigen} rely on the following theorem, which is a slight variant of \cite[Theorem 10]{Xie1992} that immediately follows from the linearity of the operators.
Suppose $\banach$ is a complex Banach space and $V\in B_0(\banach)$.
For $\lambda\in\sigma(V)$ recall that $m_V(\lambda)$ is equal to the dimension of the generalized eigenspace $E_\lambda=\cup_{j=1}^\infty \Null(V-\lambda\Id_\banach)^j$.

\begin{theorem}\label{thm:eigenreduction}
	Let $\banach$ be a complex Banach space and $V\in B_0(\banach)$ satisfy $V\chi_0=\mu\chi_0$ for some $\mu\in\CC\setminus\{0\}$ and $\chi_0\in\banach\setminus\{0\}$.
	If $L\in B(X,\CC)$ is a continuous linear functional such that $L\chi_0=1$ and $W:\banach\to\banach$ is defined by $W\chi=V\chi-L(V\chi)\chi_0$ for $\chi\in X$, then $W\in B_0(\banach)$ and
	\begin{equation*}
    	m_W(\lambda)=\begin{cases}
        	m_V(\lambda) & \text{if  } \lambda\neq \mu,\\
        	m_V(\lambda)-1 & \text{if  } \lambda = \mu.
    	\end{cases}
	\end{equation*}
    In particular, if $\mu=1$ then $\sigminus(V)=\sigma(W)$.
\end{theorem}

The following proof of Lemma \ref{lem:DPhi} closely parallels the proof of \cite[Theorem 9]{Xie1992}.

\begin{proof}[Proof of Lemma \ref{lem:DPhi}]
	Define $L\in B_0(\C(\R),\R)$ by $L\phi=\phi(-1)$ for all $\phi\in\C(\R)$ so that $T(\beta,\phi,t)=LS(\beta,\phi,t)$ for all $\beta>0$, $\phi\in\C(\R)$ and $t\ge0$.
	Recall that for each $\beta>\buniq$, by the definition of the monodromy operator, $D_\phi S(\beta,\sops_0^\beta,\period^\beta)=\monoU_1^\beta(0)$ (see, e.g., \cite[Chapter 2, Theorem 4.1]{Hale1993}).
	It follows that $D_\phi T(\beta,\phi,\period^{\beta})=LU_1^{\beta}(0)$ for each $\beta>\buniq$.
	Let $\beta>\buniq$.
	By Lemma \ref{lem:qbeta}, $T(\tilde\beta,\phi,q^\beta(\tilde\beta,\phi))=0$ for all $(\tilde\beta,\phi)\in W^\beta$ and so
	\begin{align*}
		0&=D_\phi\lsb T(\tilde\beta,\phi,q^\beta(\tilde\beta,\phi))\rsb\big|_{(\tilde\beta,\phi)=(\beta,\sops_0^\beta)}\\
		&=D_\phi T(\beta,\sops_0^\beta,\period^\beta-1)+D_tT(\beta,\sops_0^\beta,\period^\beta-1)D_\phi q^\beta(\beta,\sops_0^\beta)\\
		&=LU_1^\beta(0)+\dot{\sops}^\beta(-1)D_\phi q^\beta(\beta,\sops_0^\beta).
	\end{align*}
	Hence,
		$$D_\phi q^\beta(\beta,\sops_0^\beta)=-\frac{L\monoU_1^\beta(0)}{\dot{\sops}^\beta(-1)}.$$
	By the definition of $\Phi^\beta$ in \eqref{eq:Phibeta}, the derivative of $\Phi^\beta$ with respect to $\phi$ satisfies
	\begin{align*}
		D_\phi\Phi^\beta(\beta,\sops_0^\beta)&=D_\phi S(\beta,\sops_0^\beta,\period^\beta)+D_tS(\beta,\sops_0^\beta,\period^\beta)D_\phi q^\beta(\beta,\sops_0^\beta)\\
		&=\monoU_1^\beta(0)-\frac{\dot\sops_0^\beta}{\dot{\sops}^\beta(-1)}LU_1^\beta(0).
	\end{align*}
	Since $\monoU_1^\beta(0)\in B_0(\C(\R))$, we see that $D_\phi\Phi^\beta(\beta,\sops_0^\beta)\in B_0(\C(\R))$.
	Now an application of Theorem \ref{thm:eigenreduction} with $X=\C(\R)$, $V=\monoU_1^\beta(0)$, $\chi_0\in\C(\R)$ defined by $\chi_0(\theta)=\frac{\dot\sops_0^\beta(\theta)}{\dot{\sops}^\beta(-1)}$ for $\theta\in[-1,0]$, $\mu=1$, and $W=D_\phi\Phi^\beta(\beta,\sops_0^\beta)$ implies that $\sigma(D_\phi\Phi^\beta(\beta,\sops_0^\beta))=\sigminus(\monoU_1^\beta(0))$.
	By \cite[Theorem 1]{Xie1991}, 1 is a simple eigenvalue of $\monoU_1^\beta(0)$, thus completing the proof of the lemma.
\end{proof}

Given a Banach space $\banach$, $\chi\in\banach$ and $r>0$, recall that $\ball_\banach(\chi,r)$ denotes the ball of radius $r$ centered at $\chi$.

\begin{proof}[Proof of Corollary \ref{cor:mu_eigen}]
	Let $\delta\in(0,1)$, $\mu$ and $\psi$ be as in Theorem \ref{thm:mu_simple}.
    Let $\banachY^\ast=B(\banachY,\CC)$ denote the dual of $\banachY$.
    We claim, and prove below, there is a $\delta_0\in(0,\delta]$ such that for each $\chi\in\ball_\banach(\chi_0,\delta_0)$ there is a linear functional $L(\chi)\in\banachY^\ast$ such that $L(\chi)\psi(\chi)=1$, and $L$ is continuous in $\chi$.
    Assuming the claim holds, define, for each $\chi\in\ball_\banach(\chi_0,\delta_0)$, the operator $W(\chi)\in B_0(\banachY)$ by 
    	$$W(\chi)\xi=V(\chi)\xi-(L(\chi)\xi)\psi(\chi),\qquad\xi\in\banachY.$$ 
    Then $W(\chi)$ is continuous in $\chi$, and by Theorem \ref{thm:eigenreduction} and the fact that $\mu(\chi)$ is a simple eigenvalue of $V(\chi)$ for each $\chi\in \ball_\banach(\chi_0,\delta_0)$, it follows that $W(\chi)\in B_0(\banachY)$ and $\mu(\chi)\not\in\sigma(W(\chi))$ for each $\chi\in \ball_\banach(\chi_0,\delta_0)$.
    Since $W(\chi_0)$ is compact, it has an isolated spectrum, which along with the continuity of the function $\chi\mapsto \sigma(W(\chi))$, implies that by choosing $\delta_0>0$ possibly smaller, we can ensure that for each $\chi\in\ball_\banach(\chi_0,\delta_0)$, the spectrum of $W(\chi)$ does not contain any elements in $\ball(\mu_0,\delta_0)$.

    We are left to prove the claim.
    For each $\chi\in O$ let $V^\ast(\chi)\in B_0(\banachY^\ast)$ denote the adjoint of $V(\chi)$.
    Then $V^\ast:O\to B_0(\banachY^\ast)$ is continuous and for each $\chi\in\ball(\chi_0,\delta)$, $\mu(\chi)$ is the unique simple eigenvalue of $V(\chi)$.
    Moreover, $\psi_0^\ast\in\banachY^\ast$ is a unit eigenfunction associated with $\mu_0$.
	Applying Theorem \ref{thm:mu_simple} again, this time with $\banachY^\ast$, $V^\ast$ and $\psi_0^\ast$ in place of $\banachY$, $V$ and $\psi_0$, respectively, there is a $\delta_0>0$ and a continuous function $\psi^\ast:\ball(\chi_0,\delta_0)\to\banachY^\ast$ such that for each $\chi\in\ball(\chi_0,\delta_0)$, $\psi^\ast(\chi)$ is a unit eigenfunction of $V^\ast(\chi)$ associated with $\mu(\chi)$.
	Since $\psi_0^\ast(\psi_0)\neq0$, by choosing $\delta_0>0$ possibly smaller, we can ensure that $\psi^\ast(\chi)(\psi(\chi))\neq0$ for all $\psi\in\ball_\banach(\chi_0,\delta_0)$.
	For each $\chi\in\ball_\banach(\chi_0,\delta_0)$, define $L(\chi)\in\banachY^\ast$ by 
		$$L(\chi)\xi=\frac{\psi^\ast(\chi)(\xi)}{\psi^\ast(\chi)(\psi(\chi))},\qquad\xi\in\banachY.$$
	Then $L(\chi)\psi(\chi)=1$ for all $\chi\in\ball_\banach(\chi_0,\delta_0)$, and $L(\chi)$ is a continuous function of $\chi$.
	This completes the proof of the claim.          
\end{proof}

\subsection*{Acknowledgements}

We are grateful to the anonymous referee for a careful reading of this paper and for their helpful comments which led to significant improvements in the organization and presentation of our work.
We thank Siavash Golkar for helpful comments on Section \ref{sec:uniform}.

\bibliographystyle{plain}
\bibliography{sync}

\begin{thebibliography}{10}

\bibitem{Atay2006}
F.~M. Atay.
\newblock {Oscillator death in coupled functional differential equations near
  Hopf bifurcation}.
\newblock {\em J.\ Differential Eqs.}, 221(1):190--209, 2006.

\bibitem{Campbell2018}
S.~Campbell and Z.~Wang.
\newblock {Phase models and clustering in networks of oscillators with delayed
  coupling}.
\newblock {\em Physica D}, 363:44--55, 2018.

\bibitem{campbell2006multistability}
S.~A. Campbell, I.~Ncube, and J.~Wu.
\newblock Multistability and stable asynchronous periodic oscillations in a
  multiple-delayed neural system.
\newblock {\em Physica D}, 214(2):101--119, 2006.

\bibitem{Campbell2007}
S.A. Campbell.
\newblock {\em Time delays in neural systems}.
\newblock Springer-Verlag, Berlin Heidelberg, 2007.

\bibitem{Chen2000}
Y.~Chen, Y.~S. Huang, and J.~Wu.
\newblock Desynchronization of large scale delayed neural networks.
\newblock {\em Proc. Amer. Math. Soc.}, 128(8):2365--2371, 2000.

\bibitem{Choe2010}
C.~U. Choe, T.~Dahms, P.~H{\"{o}}vel, and E.~Sch{\"{o}}ll.
\newblock {Controlling synchrony by delay coupling in networks: from in-phase
  to splay and cluster states}.
\newblock {\em Phys. Rev. E}, 81:025205, 2010.

\bibitem{Chow1982}
S.-N. Chow and J.~K. Hale.
\newblock {\em Methods of Bifurcation Theory}.
\newblock Springer-Verlag, New York, 1982.

\bibitem{Conway1990}
J.~B. Conway.
\newblock {\em Functional Analysis}.
\newblock Springer-Verlag, New York, 1990.

\bibitem{Dhamala2004}
M.~Dhamala, V.~K. Jirsa, and M.~Ding.
\newblock {Enhancement of neural synchrony by time delay}.
\newblock {\em Phys. Rev. Lett.}, 92(7):074104, 2004.

\bibitem{Diekmann1991}
O~Diekmann, S.A. van Gils, S.M. Verduyn~Lunel, and H.-O. Walther.
\newblock {\em Delay Equations: Functional-, Complex-, and Nonlinear Analysis}.
\newblock Springer-Verlag, New York, 1991.

\bibitem{Ermentrout2010}
G.~Ermentrout and D.~Terman.
\newblock {\em Mathematical Foundations of Neuroscience}.
\newblock Springer, New York, 2010.

\bibitem{Flunkert2010}
V.~Flunkert, S.~Yanchuk, T.~Dahms, and E.~Sch{\"{o}}ll.
\newblock {Synchronizing distant nodes: a universal classification of
  networks}.
\newblock {\em Phys. Rev. Lett.}, 105:254101, 2010.

\bibitem{Flunkert2014}
V.~Flunkert, S.~Yanchuk, T.~Dahms, and E.~Sch{\"{o}}ll.
\newblock {Synchronizability of networks with strongly delayed links: a
  universal classification}.
\newblock {\em J. Math. Sci. (N. Y.)}, 202(6):809--824, 2014.

\bibitem{glass2016signaling}
D.~S. Glass, X.~Jin, and I.~H. Riedel-Kruse.
\newblock Signaling delays preclude defects in lateral inhibition patterning.
\newblock {\em Phys. Rev. Lett.}, 116(12):128102, 2016.

\bibitem{Gray2006}
R.~M. Gray.
\newblock {\em Toeplitz and Circulant Matrics: A Review}.
\newblock Now Publishers Inc, Hanover, MA, 2006.

\bibitem{Guo2003}
S.~Guo and L.~Huang.
\newblock {Hopf bifurtcating periodic orbits in a ring of neurons with delays}.
\newblock {\em Phys. D}, 183(1):19--44, 2003.

\bibitem{Guo2007}
S.~Guo and L.~Huang.
\newblock {Stability of nonlinear waves in a ring of neurons with delays}.
\newblock {\em J. Differential Equations}, 183(2):343--374, 2007.

\bibitem{Guo2009}
S.~Guo and Y.~Yuan.
\newblock {Pattern formation in a ring network with delay}.
\newblock {\em Math. Models Methods Appl. Sci.}, 19(10):1797--1852, 2009.

\bibitem{Hadeler1977}
K.~P. Hadeler and J.~Tomiuk.
\newblock Periodic solutions of difference-differential equations.
\newblock {\em Arch.\ Ration.\ Mech.\ Anal.}, 65:87--95, 1977.

\bibitem{Hale1977}
J.~K. Hale.
\newblock {\em Theory of Functional Differential Equations}.
\newblock Springer, New York, 1977.

\bibitem{Hale1993}
J.~K. Hale and S.~M. Verduyn~Lunel.
\newblock {\em Introduction to Functional Differential Equations}.
\newblock Springer, New York, 1993.

\bibitem{Hoppensteadt}
F.~Hoppensteadt and E.~Ishikvich.
\newblock {\em Weakly Connected Neural Networks}.
\newblock Springer, New York, 1997.

\bibitem{horn2012matrix}
R.~A. Horn and C.~R. Johnson.
\newblock {\em Matrix Analysis}.
\newblock Cambridge University Press, Cambridge, 2012.

\bibitem{Kinzel2009}
W.~Kinzel, A.~Englert, G.~Reents, M.~Zigzag, and I.~Kanter.
\newblock {Synchronization of networks of chaotic units with time-delayed
  couplings}.
\newblock {\em Phys. Rev. E}, 79:056207, 2009.

\bibitem{Kopell2002}
N.~Kopell and G.~Ermentrout.
\newblock {\em Mechanisms of phase-locking and frequency control in pairs of
  coupled neural oscillators}, pages 3--54.
\newblock Elsevier, Amsterdam, 2002.

\bibitem{Kuramoto1984}
Y.~Kuramoto.
\newblock {\em Chemical Oscillations, Waves, and Turbulence}.
\newblock Springer, Berlin, 1984.

\bibitem{Lang1999}
S.~Lang.
\newblock {\em Complex Analysis}.
\newblock Springer, New York, fourth edition, 1999.

\bibitem{Lawrence1972}
J.~D. Lawrence.
\newblock {\em A Catalog of Special Plane Curves}.
\newblock Dover, Mineola, New York, 1972.

\bibitem{Marcus1988}
C.~M. Marcus and R.~M. Westervelt.
\newblock Stability of analog neural networks with delay.
\newblock {\em Phys. Rev. A}, 39:347--359, 1988.

\bibitem{Mather2014}
W.~Mather, J.~Hasty, and Tsimring~L.\ S.
\newblock Synchronization of degrade-and-fire oscillations via a common
  activator.
\newblock {\em Phys. Rev. Lett.}, 113(12):128102, 2014.

\bibitem{novak2008design}
B.~Nov{\'a}k and J.~J. Tyson.
\newblock Design principles of biochemical oscillators.
\newblock {\em Nat. Rev. Mol.}, 9(12):981, 2008.

\bibitem{Pecora1998}
L.~Pecora and T.~Carroll.
\newblock {Master stability functions for synchronized coupled systems}.
\newblock {\em Phys. Rev. Lett.}, 80(10):2109--2112, 1998.

\bibitem{Perlikowski2010}
P.~Perlikowski, S.~Yanchuk, O.V. Popovych, and Tass P.A.
\newblock {Periodic patterns in a ring of delay-coupled oscillators}.
\newblock {\em Phys. Rev. E}, 82:036208, 2010.

\bibitem{Popovich2011}
O.V. Popovich, S.~Yanchuck, and Tass P.A.
\newblock {Delay- and coupling-induced firing patterns in oscillatory neural
  loops}.
\newblock {\em Phys. Rev. Lett.}, 107(22):228102, 2011.

\bibitem{Rosenblum2004}
M.~Rosenblum and A.~Pikovsky.
\newblock {Delayed feedback control of collective synchrony: An approach to
  suppression of pathological brain rhythms}.
\newblock {\em Phys. Rev. E}, 70(4):041904, 2004.

\bibitem{Sieber2013}
J.~Sieber, M.~Wolfrum, M.~Lichtner, and S.~Yanchuk.
\newblock {On the stability of periodic orbits in delay equations with large
  delay}.
\newblock {\em Discrete Contin. Dyn. Syst.}, 33(7):3109--3134, 2013.

\bibitem{Su2017}
T.~Su, Lee W., Wang C., and Lo~C.
\newblock Coupled symmetric and asymmetric circuits underlying spatial
  orientation in fruit flies.
\newblock {\em Nat. Commun.}, 8:139, 2017.

\bibitem{Tian2017}
C.~Tian, H.~Bi, X.~Zhang, and Z.~Liu.
\newblock Asymmetric couplings enhance the transition from chimera state to
  synchronization.
\newblock {\em Phys. Rev. E}, 96:052209, 2017.

\bibitem{Wang2017}
Z.~Wang and S.~A. Campbell.
\newblock Symmetry, {H}opf bifurcation and the emergence of cluster solutions
  in time delayed neural networks.
\newblock {\em CHAOS}, 27(11):114316, 2017.

\bibitem{Winfree1967}
A.~Winfree.
\newblock Biological rhythms and the behavior of populations of coupled
  oscillators.
\newblock {\em J. Theor. Biol.}, 16(1):15--42, 1967.

\bibitem{Wu1998}
J.~Wu.
\newblock Symmetric functional differential equations and neuronal networks
  with memory.
\newblock {\em Trans. Amer. Math. Soc.}, 350(12):4799--4838, 1998.

\bibitem{wu2011introduction}
J.~Wu.
\newblock {\em Introduction to Neural Dynamics and Signal Transmission Delay},
  volume~6.
\newblock Walter de Gruyter, 2011.

\bibitem{Wu2015}
J.~Wu, S.~A. Campbell, and J.~B\'elair.
\newblock {\em Time-delayed Neural Networks: Stability and Oscillations}.
\newblock Springer, New York, 2015.

\bibitem{Wu1999}
J.~Wu, T.~Faria, and Y.~S. Huang.
\newblock Synchronization and stable phase-locking in a network of neurons with
  memory.
\newblock {\em Math. Comput. Modelling}, 30:117--138, 1999.

\bibitem{Xie1991}
X.~Xie.
\newblock Uniqueness and stability of slowly oscillating periodic solutions of
  delay equations with bounded nonlinearity.
\newblock {\em J. Dynam. Differential Eqs.}, 3(4):515--540, 1991.

\bibitem{Xie1992}
X.~Xie.
\newblock The multiplier equation and its applications to {$S$}-solutions of a
  differential delay equation.
\newblock {\em J. Differential Eqs.}, 95:259--280, 1992.

\end{thebibliography}

\end{document}